%% file: commutation_radial.tex
\documentclass[11pt]{article}
\usepackage{authblk}
\usepackage[toc,page]{appendix}
\usepackage[top=2cm, bottom=2cm, left=2cm, right=2cm]{geometry}

\usepackage{color}
\usepackage{helvet}         % selects\textbf{\textbf{•}} Helvetica as sans-serif font
\usepackage{courier}        % selects Courier as typewriter font
\usepackage{type1cm}        % activate if the above 3 fonts are
\usepackage{framed} 
\usepackage{tikz}
\usepackage{makeidx}         % allows index generation
\usepackage{graphicx}        % standard LaTeX graphics tool
\usepackage{multicol}        % used for the two-column index
\usepackage[bottom]{footmisc}% places footnotes at page bottom 

\usepackage{amsmath}
\usepackage{amssymb}
\usepackage{bbold}
\usepackage{amsthm}
\usepackage{subcaption}
\usepackage{sidecap}
\usepackage{floatrow}
\usepackage{pdflscape}
\usepackage{comment}
\usepackage[font=small]{caption}
\usepackage{enumitem}
\usepackage{esint}

\usepackage[hidelinks]{hyperref}

\usepackage{scalerel}
\usepackage{shuffle}
\usepackage{mathrsfs}
\usepackage{mathtools}

\newtheorem{theorem}{Theorem}
\newtheorem{corollary}[theorem]{Corollary}
\newtheorem{lemma}[theorem]{Lemma}

\newtheorem{proposition}[theorem]{Proposition}

\theoremstyle{definition}

\theoremstyle{remark}

\newtheorem{example}[theorem]{\bf Example}

\numberwithin{theorem}{section}
\numberwithin{figure}{section}
\numberwithin{equation}{section}

\begin{document}
\title{Classification of commutation relation for multi-radial SLE}
\bigskip{}
\author[1]{Chongzhi Huang\thanks{huangchzh2001prob@gmail.com. ORCID:0009-0002-7991-6267}}
\author[1]{Hao Wu\thanks{hao.wu.proba@gmail.com. ORCID:0000-0003-4265-7417}}
\affil[1]{Tsinghua University, China}
% your contribution title if the original one is too long%

\date{}
\input{tex_radial/macros}
\maketitle
\vspace{-1cm}

\begin{center}
\begin{minipage}{0.95\textwidth}
\abstract{
Locally commuting multiple radial Schramm–Loewner evolutions ($\SLE_{\kappa}$) are encoded by partition functions satisfying the radial Belavin-Polyakov-Zamolodchikov (BPZ) equations and a conformal Ward identity with spectral parameters $\lambda,\nu\in \R$. 
For $\kappa>0$ and $\lambda\in \R$, we show that the solution space of the radial BPZ system has dimension $2^n$ where $n$ is the number of variables. 
We then determine all admissible Ward parameters $\nu$ and the exact dimensions of the subspaces selected by the conformal Ward identity, covering both generic and degenerate cases. 
The classification reveals a parity difference: nonzero rotation-invariant solutions exist for every $\lambda$ when $n$ is even, but only at finitely many exceptional values when $n$ is odd. 

When $0<\kappa\le4$, $\lambda>0$ for odd $n$ or $\lambda>-3/2$ for even $n$, we construct a basis of positive solutions using multiple SLE, providing global realizations of the locally commuting SLEs. 
At $\kappa=4$, we identify a family of explicit solutions as partition functions for level lines of a Gaussian free field with suitable boundary data and interior singularities. 
}

\end{minipage}
\end{center}

\tableofcontents

\section{Introduction}
\input{tex_radial/intro}
%%%%%%%%%%%%%%%%%%%%%%%%%%%%%%%%%%%%%%
\section{Preliminaries on multiple SLE}
\label{sec::pre_multiSLE}
\input{tex_radial/pre_multiSLE}

%%%%%%%%%%%%%%%%%%%%%%%%%%%%%%%%%%%%%%%
\section{Positive solutions to radial BPZ system}
\label{sec::solutions_radialBPZ}
\input{tex_radial/solutions_radial}

%%%%%%%%%%%%%%%%%%%%%%%%%%%%%%%%%%%%%%%
\section{Analysis on the solution space}
\label{sec::PDE_analysis}
\input{tex_radial/PDE_analysis}
%%%%%%%%%%%%%%%%%%%%%%%%%
%%%%%%%%%%%%%%%%%%%%%%%%%%%%%%%%%%%%%%
\appendix
\section{Dubedat's commutation relation}
\label{sec::commutation_radial}
\input{tex_radial/commutation_radial}

\section{Chordal BPZ system}
\label{sec::chordalBPZ}
\input{tex_radial/chordalBPZ}
%{\small
%	\bibliographystyle{alpha}
%	\bibliography{bibliography}}

{\small
\newcommand{\etalchar}[1]{$^{#1}$}

}

\end{document}

%% file: tex_radial/macros.tex
%%%%general%%%%%%%%
\newcommand{\Eig}{\mathcal{E}}
\newcommand{\Prad}[1]{\mathsf{P}_{#1\mathrm{\textnormal{-}rad}}}
\newcommand{\Erad}[1]{\mathsf{E}_{#1\mathrm{\textnormal{-}rad}}}
\newcommand{\LZrad}[1]{\mathcal{Z}_{#1\mathrm{\textnormal{-}rad}}}
\newcommand{\LSrad}[1]{\mathsf{S}_{#1\mathrm{\textnormal{-}rad}}}
\newcommand{\LFrad}[1]{\mathcal{F}_{#1\mathrm{\textnormal{-}rad}}}

\newcommand{\Pchord}{\mathsf{P}}
\newcommand{\Pchordtwo}{\Pchord_{\includegraphics[scale=0.15]{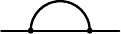}}}
\newcommand{\Echord}{\mathsf{E}}
\newcommand{\Echordtwo}{\mathsf{E}_{\includegraphics[scale=0.15]{figures/link-0}}}
\newcommand{\LSchord}[1]{\mathsf{S}_{#1\mathrm{\textnormal{-}chord}}}
\newcommand{\LFchord}[1]{\mathcal{F}_{#1\mathrm{\textnormal{-}chord}}}

\newcommand{\blm}{\mathfrak{m}}
\newcommand{\LZmix}{\mathcal{Z}_{\bs{s};\alpha}^{(p; \mu)}}
\newcommand{\LZmixp}{\mathcal{Z}_{\bs{s};\alpha}^{(p; \mu_p)}}
\newcommand{\Cmixp}{C_{\bs{s};\alpha}^{(p; \mu_p)}}
\newcommand{\LZmixpminus}{\mathcal{Z}_{\bs{s};\alpha}^{(p; -\mu_p)}}
\newcommand{\Cmixpminus}{C_{\bs{s};\alpha}^{(p; -\mu_p)}}
\newcommand{\LZmixhat}{\mathcal{Z}_{\bs{s};\hat{\alpha}}^{(p; \mu)}}
\newcommand{\LZalphaCR}{\mathcal{Z}_{\alpha{\textnormal{-}}\mathrm{cr}}^{(\ell;\mu)}}
\newcommand{\CalphaCR}{C_{\alpha{\textnormal{-}}\mathrm{cr}}^{(\ell;\mu)}}
\newcommand{\LZalphaCRhat}{\mathcal{Z}_{\hat{\alpha}{\textnormal{-}}\mathrm{cr}}^{(\ell;\mu)}}
\newcommand{\LZalphahcap}{\mathcal{Z}_{\alpha{\textnormal{-}}\hcap}^{(\lambda)}}
\newcommand{\LZalphahcaphat}{\mathcal{Z}_{\hat{\alpha}{\textnormal{-}}\hcap}^{(\lambda)}}

%%%%%mathbb%%%%%%
\global\long\def\U{\mathbb{U}}
\global\long\def\T{\mathbb{T}}
\global\long\def\HH{\mathbb{H}}
\global\long\def\R{\mathbb{R}}
\global\long\def\C{\mathbb{C}}
\global\long\def\N{\mathbb{N}}
\global\long\def\Z{\mathbb{Z}}
\global\long\def\E{\mathbb{E}}
\global\long\def\PP{\mathbb{P}}
\global\long\def\QQ{\mathbb{Q}}
\global\long\def\A{\mathbb{A}}
\global\long\def\one{\mathbb{1}}

%%%%mathrm%%%%%%
\global\long\def\CR{\mathrm{CR}}
\global\long\def\ST{\mathrm{ST}}
\global\long\def\SF{\mathrm{SF}}
\global\long\def\cov{\mathrm{cov}}
\global\long\def\dist{\mathrm{dist}}
\global\long\def\SLE{\mathrm{SLE}}
\global\long\def\hSLE{\mathrm{hSLE}}
\global\long\def\CLE{\mathrm{CLE}}
\global\long\def\GFF{\mathrm{GFF}}
\global\long\def\inte{\mathrm{int}}
\global\long\def\ext{\mathrm{ext}}
\global\long\def\inrad{\mathrm{inrad}}
\global\long\def\outrad{\mathrm{outrad}}
\global\long\def\dimH{\mathrm{dim}}
\global\long\def\capa{\mathrm{cap}}
\global\long\def\diam{\mathrm{diam}}
\global\long\def\sign{\mathrm{sgn}}
\global\long\def\cat{\mathrm{Cat}}
\global\long\def\cst{\mathrm{C}}
\global\long\def\ck{\mathrm{C}_{\kappa}}
\global\long\def\free{\mathrm{free}}
\global\long\def\hF{{}_2\mathrm{F}_1}
\global\long\def\simple{\mathrm{simple}}
\global\long\def\even{\mathrm{even}}
\global\long\def\odd{\mathrm{odd}}
\global\long\def\st{\mathrm{ST}}
%\global\long\def\ust{\mathrm{UST}}
\global\long\def\usf{\mathrm{USF}}
\global\long\def\Leb{\mathrm{Leb}}
\global\long\def\LP{\mathrm{LP}}
\global\long\def\I{\mathrm{I}}
\global\long\def\K{\mathrm{K}}
\global\long\def\J{\mathrm{J}}
\global\long\def\Y{\mathrm{Y}}
\global\long\def\Cat{\mathrm{Cat}}

\global\long\def\II{\mathrm{II}}
\global\long\def\hcap{\mathrm{hcap}}
\global\long\def\const{\mathrm{C}}
\global\long\def\Poisson{\mathrm{H}}
\global\long\def\Green{\mathrm{G}}

%%%%%mathcal%%%%%%%%%%%
\global\long\def\LA{\mathcal{A}}
\global\long\def\LB{\mathcal{B}}
\global\long\def\LC{\mathcal{C}}
\global\long\def\LD{\mathcal{D}}
\global\long\def\LF{\mathcal{F}}
\global\long\def\LK{\mathcal{K}}
\global\long\def\LE{\mathcal{E}}
\global\long\def\LG{\mathcal{G}}
\global\long\def\LI{\mathcal{I}}
\global\long\def\LJ{\mathcal{J}}
\global\long\def\LL{\mathcal{L}}
\global\long\def\LM{\mathcal{M}}
\global\long\def\LN{\mathcal{N}}
\global\long\def\OO{\mathcal{O}}
\global\long\def\LQ{\mathcal{Q}}
\global\long\def\LR{\mathcal{R}}
\global\long\def\LT{\mathcal{T}}
\global\long\def\LS{\mathcal{S}}
\global\long\def\LU{\mathcal{U}}
\global\long\def\LV{\mathcal{V}}
\global\long\def\LW{\mathcal{W}}
\global\long\def\LX{\mathcal{X}}
\global\long\def\LY{\mathcal{Y}}
\global\long\def\PartF{\mathcal{Z}}
\global\long\def\LH{\mathcal{H}}
\global\long\def\LJ{\mathcal{J}}

%%%%%mathfrak%%%%%%%%%%%%
\global\long\def\blm{\mathfrak{m}}

%%%partitionfunctions%%%%%%%%%%%%%%%%%%
\global\long\def\LZ{\mathcal{Z}}
\global\long\def\LZrp{\mathcal{Z}_{\alpha; \bs{s}}^{(p)}}
\global\long\def\LJrp{\mathcal{J}_{\alpha; \bs{s}}^{(p)}}
\global\long\def\chamberrp{\chamber_{\alpha; \bs{s}}^{(p)}}
\global\long\def\LErp{\mathcal{E}_{\alpha; \bs{s}}^{(p)}}
\global\long\def\Greenrp{G_{\alpha; \bs{s}}^{(p)}}
\global\long\def\Prp{P_{\alpha; \bs{s}}^{(p)}}
\global\long\def\norcst{\mathrm{C}_{\kappa}^{(\mathfrak{r})}}
\global\long\def\LZalphar{\mathcal{Z}_{\alpha}^{(\mathfrak{r})}}

%%%%partitionfunctionstwo%%%%%%%%%%%
\newcommand{\LZtwo}{\mathcal{Z}_{\includegraphics[scale=0.15]{figures/link-0}}}
\newcommand{\LEtwo}{\mathcal{E}_{\includegraphics[scale=0.15]{figures/link-0}}}
\newcommand{\LItwo}{\mathcal{I}_{\includegraphics[scale=0.15]{figures/link-0}}}
\newcommand{\LUtwo}{\mathcal{U}_{\includegraphics[scale=0.15]{figures/link-0}}}
\newcommand{\LZtwor}{\LZtwo^{(\mathfrak{r})}}
\newcommand{\LHtwo}{\mathcal{H}_{\includegraphics[scale=0.15]{figures/link-0}}}
\newcommand{\LFtwo}{\mathcal{F}_{\includegraphics[scale=0.15]{figures/link-0}}}

%%%partitionfunctionsfour%%%%%%%%%%%%
%\newcommand{\LZfoura}{\mathcal{Z}_{\includegraphics[scale=0.15]{figures/link-1}}}
%\newcommand{\LZfourb}{\mathcal{Z}_{\includegraphics[scale=0.15]{figures/link-2}}}
%\newcommand{\LHfoura}{\mathcal{H}_{\includegraphics[scale=0.15]{figures/link-1}}}
%\newcommand{\LHfourb}{\mathcal{H}_{\includegraphics[scale=0.15]{figures/link-2}}}
%\newcommand{\LFfoura}{\mathcal{F}_{\includegraphics[scale=0.15]{figures/link-1}}}
%\newcommand{\LFfourb}{\mathcal{F}_{\includegraphics[scale=0.15]{figures/link-2}}}
%\newcommand{\LFfouraRenorm}{\widehat{\mathcal{F}}_{\includegraphics[scale=0.15]{figures/link-1}}}
%\newcommand{\LFfourbRenorm}{\widehat{\mathcal{F}}_{\includegraphics[scale=0.15]{figures/link-2}}}

%%%%%%%%%%%%%%%%%%%%%%%%%%%%%%%
%\newcommand{\QQrainbow}{\mathbb{Q}_{\includegraphics[scale=0.15]{figures/link-2}}}
%\newcommand{\LZrainbow}{\mathcal{Z}_{\includegraphics[scale=0.15]{figures/link-2}}}
%\newcommand{\QQfusion}{\mathbb{Q}_{\includegraphics[scale=0.5]{figures/link4fusion}}}
\newcommand{\coulombGasHRenorm}{\widehat{\coulombGasH}}

\global\long\def\coulomb{\LH}
\global\long\def\auxcoulomb{\hat{\coulomb}}%{\tilde{\coulomb}}
\global\long\def\coulombGas{\LF}
\global\long\def\coulombnew{\LK}%{\LG}
\global\long\def\coulombLine{\LG}%{\LG}
\global\long\def\kfunc{p}

\global\long\def\eps{\epsilon}
\global\long\def\ov{\overline}
\global\long\def\QQrp{\QQ_{\alpha; \bs{s}}^{(p)}}

\global\long\def\bn{\mathbf{n}}
\global\long\def\MR{MR}
\global\long\def\cond{\,|\,}
\global\long\def\la{\langle}
\global\long\def\ra{\rangle}
\global\long\def\tree{\Upsilon}
\global\long\def\prob{\mathbb{P}}
\global\long\def\hm{\mathrm{Hm}}
%
%\global\long\def\sf{\mathrm{SF}}
%\global\long\def\wr{\varrho}

\global\long\def\Im{\operatorname{Im}}
\global\long\def\Re{\operatorname{Re}}

%Eves macros
\global\long\def\ud{\mathrm{d}}
\global\long\def\pder#1{\frac{\partial}{\partial#1}}
\global\long\def\pdder#1{\frac{\partial^{2}}{\partial#1^{2}}}
\global\long\def\pddder#1{\frac{\partial^{3}}{\partial#1^{3}}}
\global\long\def\der#1{\frac{\ud}{\ud#1}}

\global\long\def\bZnn{\mathbb{Z}_{\geq 0}}
\global\long\def\bZpos{\mathbb{Z}_{> 0}}
\global\long\def\bZneg{\mathbb{Z}_{< 0}}

\global\long\def\Vfunc{\LG}%{\LH}%{\LV}
\global\long\def\gfunc{g^{(\rr)}}
\global\long\def\hfunc{h^{(\rr)}}

\global\long\def\SimplexInt{\rho}
\global\long\def\CubeInt{\widetilde{\rho}}

\global\long\def\ii{\mathfrak{i}}
\global\long\def\ee{\mathrm{e}}
\global\long\def\rr{\mathfrak{r}}
\global\long\def\chamber{\mathfrak{X}}
\global\long\def\Wchamber{\mathfrak{W}}

\global\long\def\SimplexIntKappa8{\SimplexInt}

\global\long\def\nested{\boldsymbol{\underline{\Cap}}}
\global\long\def\unnested{\boldsymbol{\underline{\cap\cap}}}
%\global\long\def\unnested{\boldsymbol{\underline{\cap\scaleobj{0.85}{\,\cdots}\cap}}}
%\global\long\def\removeLink{/}
\global\long\def\unnested{\boldsymbol{\underline{\cap\cap}}}

\global\long\def\acycle{\vartheta}
\global\long\def\bcycle{\tilde{\acycle}}
%\global\long\def\Gloop{\Theta}
%\global\long\def\GGloop{\Theta'}

\global\long\def\metric{\mathrm{dist}}

\global\long\def\adj#1{\mathrm{adj}(#1)}

\global\long\def\bs{\boldsymbol}

\global\long\def\edge#1#2{\langle #1,#2 \rangle}
\global\long\def\graph{G}

\newcommand{\conn}{\varsigma}%{\vartheta}%{\LA}
\newcommand{\realacycle}{\smash{\mathring{\acycle}}}
\newcommand{\realpt}{\smash{\mathring{x}}}
\newcommand{\corrind}{\LC}
\newcommand{\bssymb}{\pi}%{\Pi}%{\zeta}
\newcommand{\PRCM}{\mu}%{\QQ}
\newcommand{\coeff}{p}
\newcommand{\MainConst}{C}

\global\long\def\removeLink{/}

\global\long\def\domainofdef{\mathfrak{U}}
\global\long\def\Test_space{C_c^\infty}
\global\long\def\Distr_space{(\Test_space)^*}

\global\long\def\bs{\boldsymbol}
\global\long\def\cst{\mathrm{C}}

\newcommand{\red}{\textcolor{red}}
\newcommand{\blue}{\textcolor{blue}}
\newcommand{\green}{\textcolor{green}}
\newcommand{\magenta}{\textcolor{magenta}}

\newcommand{\coulombGasH}{\mathcal{H}}
\newcommand{\secondbeta}{\intloop}

\newcommand{\cev}[1]{\reflectbox{\ensuremath{\vec{\reflectbox{\ensuremath{#1}}}}}}

\global\long\def\anticonf{\zeta}
\global\long\def\intloop{\varrho}
\global\long\def\Gloop{\smash{\mathring{\intloop}}}

\global\long\def\SLEmeasure{\mathrm{P}}
\global\long\def\SLEmeasureEx{\mathrm{E}}

\global\long\def\fugacity{\nu}
%\global\long\def\clweight{Q}
\global\long\def\meanderMat{\mathcal{M}}
\global\long\def\LM{\mathcal{M}}
\global\long\def\meanderMatrix{\meanderMat_{\fugacity}}
\global\long\def\meanderMatrixPrime{\meanderMat_{\fugacity(\kappa')}}
\global\long\def\meanderRenorm{\widehat{\mathcal{M}}}

\global\long\def\PartFRenorm{\widehat{\PartF}}
\global\long\def\coulombGasRenorm{\widehat{\coulombGas}}

\global\long\def\hexa{\scalebox{1.3}{\hexagon}}

\global\long\def\np{p}

\global\long\def\FKdual{\mathcal{L}}

\global\long\def\fixedindex{\flat}

%% file: tex_radial/intro.tex
Schramm–Loewner evolution (SLE), introduced by Schramm~\cite{SchrammFirstSLE}, provides a probabilistic framework for random planar curves based on conformal invariance and the domain Markov property. The latter requires that, conditional on an initial segment, the remaining curve obey the same probabilistic rule in the remaining domain, with appropriately updated marked points and boundary conditions. In the discrete models motivating SLE, this property reflects a spatial Markov structure already present at the lattice level, whereas conformal invariance emerges in the continuum scaling limit. Established examples include loop-erased random walks and uniform spanning-tree Peano curves, converging respectively to $\mathrm{SLE}_2$ and $\mathrm{SLE}_8$~\cite{LawlerSchrammWernerLERWUST};
critical percolation interfaces on the triangular lattice, converging to $\mathrm{SLE}_6$~\cite{SmirnovPercolationConformalInvariance}; contour lines of the two-dimensional discrete Gaussian free field, converging to $\mathrm{SLE}_4$~\cite{SchrammSheffieldDiscreteGFF}; and critical Ising spin and FK-Ising interfaces, converging to $\mathrm{SLE}_3$ and $\mathrm{SLE}_{16/3}$, respectively~\cite{CDCHKSConvergenceIsingSLE}.

Commutation relations provide the fundamental consistency requirement for systems with multiple macroscopic interfaces. When boundary conditions generate several interfaces, their joint law must ensure that the order in which initial segments are revealed does not affect the resulting distribution. In the framework developed by Dubédat~\cite{DubedatCommutationSLE}, this compatibility is expressed through identities between the infinitesimal generators of the driving processes. The admissible drifts are encoded by logarithmic derivatives of a common positive partition function satisfying a system of linear second-order Belavin–Polyakov–Zamolodchikov (BPZ) equations. A complementary approach by Bauer, Bernard, and Kytölä~\cite{BauerBernardKytolaMultipleSLE} further connects these equations to the martingale structure of statistical-mechanical partition functions, linking the consistency of random interfaces to the analysis of overdetermined PDEs.

In the chordal setting, Flores and Kleban investigated the solution space of the $2N$ BPZ equations supplemented by three conformal Ward identities and a power-law growth bound, for $\kappa\in(0,8)$. In a series of papers~\cite{FloresKlebanPDE1, FloresKlebanPDE2, FloresKlebanPDE3, FloresKlebanPDE4}, they proved that the dimension for the solution space is Catalan number $\frac{1}{N+1}\binom{2N}{N}$. 
%In~\cite{FloresKlebanPDE4}, they established Frobenius-type expansions at boundary collisions, including possible logarithmic terms, and characterized a distinguished basis of connectivity weights indexed by noncrossing pairings. 
%For $\kappa\in(0,4]$, Peltola and Wu~\cite{PeltolaWuGlobalMultipleSLEs} subsequently related pure partition functions to extremal multiple-SLE measures and proved the agreement of the local growth-process and global configurational constructions. 
See~\cite{BPZInfiniteConformalSymmetry2DQuantum, BPZInfiniteConformalSymmetryCritical2D} for the conformal field theory origin from null-vector relations for degenerate Virasoro representations, see~\cite{DubedatEulerIntegralsCommutingSLEs, KytolaPeltolaPurePartitionFunctions, PeltolaWuGlobalMultipleSLEs, WuHyperSLE, PeltolaBasisPDE, KytolaPeltolaConformalCovBoundaryCorrelation, AngHoldenSunYu2023, FengLiuPeltolaWu2024, zhang2025multiplechordalslekappaquantum} for Coulomb-gas, representation theory, and probabilistic constructions of solutions to the chordal BPZ equations, and see~\cite{PeltolaCFTSLE} for an overview of the partition-function/CFT perspective.

Radial geometry introduces an interior marked point and additional freedom associated with winding around it. Indeed, without reflection symmetry, even a single radial evolution satisfying conformal invariance and the domain Markov property may have a constant drift in its driving function. Compared with the chordal setting, radial BPZ equations have been comparatively less explored. 
From the CFT perspective, the interior marked point corresponds to a bulk insertion, and radial BPZ equations also arise in the study of bulk--boundary correlation functions, where their connection with the Calogero--Sutherland system has been investigated in~\cite{CardySLEDysonBM,BauerBernardCFTSLEradial}.
Krusell, Wang, and Wu~\cite{KrusellWangWuCommutationRelation} classified locally commuting pairs for $\kappa\in(0,8)$ under an additional interchangeability assumption, obtaining two one-parameter families: two-sided radial SLE with a constant spiraling rate, and chordal SLE reweighted by a power of the conformal radius. 
Feng and Wu~\cite{FengWuRadialBPZFKIsing} constructed positive solutions and related them to FK-Ising interfaces conditioned on a one-arm event. 
Zhang~\cite{zhang2025multipleradialslekappaquantum} developed the corresponding framework for an arbitrary number of boundary starting points, constructed rotation-covariant Coulomb gas solutions of the radial BPZ equations, and related these solutions, after a similarity transformation, to eigenfunctions of the quantum Calogero–Sutherland Hamiltonian. 

Motivated by these developments, we study the full smooth solution space of the radial BPZ system on the angular configuration chamber at a fixed spectral parameter. 
Let us summarize our conclusions. We first identify the dimension of the solution space of radial BPZ system in Theorem~\ref{thm::solutionspace_dimension}; we further decompose the radial BPZ solution space according to different levels in the conformal Ward identity in Theorem~\ref{thm::solutionspace_decompositon_dimension}. These two theorems provide a complete classification of Dub\'edat's commutation relation in the radial setting. Furthermore, we construct positive solutions with probabilistic interpretations in terms of multiple SLE. These constructions provide a global realization of solutions to Dub\'edat's commutation relation.

\subsection{Radial BPZ system}
Fix $\kappa>0$ and $\lambda\in\R$. 
We consider functions $\LZ$ that are defined on the space
\[\LX_n^{\U}=\{(\theta_1, \ldots, \theta_n)\in\R^n: \theta_1<\cdots<\theta_n<\theta_1+2\pi\}\]
and satisfy the radial BPZ system:
\begin{align}\label{eqn::BPZ_radial}
\frac{\kappa}{2}\partial_j^2\LZ+\sum_{\ell\neq j}\left(\cot\left(\frac{\theta_{\ell}-\theta_j}{2}\right)\partial_{\ell}\LZ-\frac{(6-\kappa)/\kappa}{4\sin^2\left(\frac{\theta_{\ell}-\theta_j}{2}\right)}\LZ\right)=\frac{\lambda}{\kappa}\LZ, \qquad \text{for }j\in\{1, \ldots, n\}. 
\end{align}
In this article, we investigate solutions to~\eqref{eqn::BPZ_radial}. 

\begin{theorem}\label{thm::solutionspace_dimension}
Fix $\kappa>0$ and $\lambda\in\R$ and $n\ge 2$. 
Define
\begin{equation}%\label{eqn::solutionspace_radial_def}
\LSrad{n}^{(\lambda)}:=\{F\in C^{2}(\LX_n^{\U}\to \R): F \text{ satisfies radial BPZ system~\eqref{eqn::BPZ_radial}}\}. 
\end{equation}
The dimension of the solution space $\LSrad{n}^{(\lambda)}$ is $2^n$: 
\begin{equation*}
\dim\LSrad{n}^{(\lambda)}=2^n.
\end{equation*}
Moreover, when $\kappa\in (0,4]$, we construct $2^n$ positive linearly independent solutions when $\lambda>0$ for odd $n$ and when $\lambda>-3/2$ for even $n$. Each basis element admits a global probabilistic realization in terms of multiple SLEs.
\end{theorem}

When $n=2$, one may check by splitting of variables that $\dim\LSrad{2}^{(\lambda)}=4$, see~\cite{KrusellWangWuCommutationRelation}. 
We focus on the general case $n\ge 2$ and prove Theorem~\ref{thm::solutionspace_dimension} in Sections~\ref{subsec::dimension_upper}-\ref{subsec::dimension_lower}. 
The main idea of our proof is to show $Z \longmapsto \bigl(\partial^\beta Z(\theta^0)\bigr)_{\beta\in\{0,1\}^n}$
is a linear isomorphism between $\LSrad{n}^{(\lambda)}$ and $\mathbb{R}^{2^n}$ (see Corollary~\ref{cor::iso_LS_R2n}).
Indeed, elliptic regularity implies that every solution is analytic,
and the BPZ equations recursively determine all higher derivatives
from these square-free derivatives, proving injectivity.
For surjectivity, we rewrite the BPZ system as a first-order linear
system for the square-free derivatives. The commutator identity
for the BPZ operators guarantees its compatibility, and integration
along paths in the convex configuration chamber finishes the construction. Our argument only uses finite-dimensional linear algebra, together with standard ODE theory. This simple approach works uniformly for all $\kappa>0$, including $\kappa\geq8$ (where fewer tools from multiple SLE are available). Compared with~\cite{FloresKlebanPDE1,FloresKlebanPDE2}, we do not impose the power law bound (PLB) condition in Theorem~\ref{thm::solutionspace_dimension}.

We give one explicit example of the solutions in $\LSrad{n}^{(\lambda)}$. Fix $\kappa\in (0,4]$ and $\mu\in\R$ and $n\ge 2$.
Define 
\begin{equation}\label{eqn::LZrad_mu_U}
	\LZrad{n}^{(\mu)}(\bs{\theta}) := \prod_{1\le i<\ell\le n}|\ee^{\ii\theta_{\ell}}-\ee^{\ii\theta_i}|^{\frac{2}{\kappa}} \times \exp\left( \frac{\mu}{\kappa}\sum_{j=1}^{n} \theta_j \right), \qquad \text{for }\bs{\theta}=(\theta_1, \ldots, \theta_n) \in \LX_n^{\U}. 
\end{equation}
One may check that $\LZrad{n}^{(\mu)}(\bs{\theta})$ satisfies the radial BPZ system~\eqref{eqn::BPZ_radial} with $\lambda=\frac{1}{2}(\mu^2+1-n^2)$.  
This is the partition function of multi-radial SLE with spiral introduced in~\cite{KrusellWangWuCommutationRelation,HuangPeltolaWuMultiradialSLE}, see Section~\ref{subsec::multiradialSLE}. 
In general, the solutions in $\LSrad{n}^{(\lambda)}$ are not explicit. 
We construct $2^n$ positive linearly independent solutions to the radial BPZ system~\eqref{eqn::BPZ_radial} when $\kappa\in (0,4]$, $\lambda>0$ for odd $n$ and $\lambda>-3/2$ for even $n$ in Section~\ref{sec::solutions_radialBPZ}. For these parameters, the locally commuting SLEs encoded by the BPZ equations admit global realizations.
The building blocks of the construction are multi-chordal SLE~\cite{KozdronLawlerMultipleSLEs, PeltolaWuGlobalMultipleSLEs} and multi-radial SLE with spiral~\cite{HealeyLawlerNSidedRadialSLE,KrusellWangWuCommutationRelation, HuangPeltolaWuMultiradialSLE}. In particular, the construction relies crucially on their boundary perturbation properties which are proved in~\cite{PeltolaWuGlobalMultipleSLEs} for the chordal setting and in~\cite{HuangPeltolaWuMultiradialSLE} for the radial setting. 

\subsection{Conformal Ward identity}

We summarize the axioms in Dub\'edat's commutation relation in Appendix~\ref{sec::commutation_radial}. 
To classify all solutions to Dub\'edat's commutation relation, analysis on the solution space $\LSrad{n}^{(\lambda)}$ alone is not sufficient. 
The BPZ system~\eqref{eqn::BPZ_radial} follows from commutation relation on the infinitesimal generators. Besides the BPZ system, conformal invariance of the probability measures in the axioms requires extra rotation covariance on the partition function (see Proposition~\ref{prop::commutation_radial_classify}). 
Therefore, in order to classify all solutions to Dub\'edat's commutation relation, one needs to understand the decomposition of  $\LSrad{n}^{(\lambda)}$ according to the following conformal Ward identity~\eqref{eqn::ward_radial}.
Fix $\kappa>0$ and $n\ge 2$. For $\lambda\in\R$ and $\nu\in\R$, we consider solutions $\LZ\in \LSrad{n}^{(\lambda)}$ that also satisfy the conformal Ward identity:
\begin{align}\label{eqn::ward_radial}
(\sum_{j=1}^n\partial_j)\LZ =\frac{\nu}{\kappa}\LZ.
\end{align}
Define 
\begin{align*}%\label{eqn::solutionspace_decompositon_radial_def}
	\LSrad{n}^{(\lambda;\nu)}:=\left\{F\in C^{2}(\LX_n^{\U}\to \R):
	\begin{array}{l}
		F \text{ satisfies radial BPZ system~\eqref{eqn::BPZ_radial}}\\
		\text{and the conformal Ward identity~\eqref{eqn::ward_radial}}
	\end{array}\right\}.
\end{align*}
The subspaces $\LSrad{n}^{(\lambda;\nu)}$ form a decomposition of the space $\LSrad{n}^{(\lambda)}$. The precise decomposition is different for odd $n$ and for even $n$.

\begin{theorem} \label{thm::solutionspace_decompositon_dimension}
Fix $\kappa>0$ and $n\ge 2$. We denote 
\begin{align}\label{eqn::parity_n_p}
\mathsf{P}_n=\{p\in\{1, \ldots, n\}: n-p\text{ is even}\}.
\end{align}
Suppose $\lambda\in\R$ and $\nu\in\R$. Define 
\begin{align}\label{eqn::nup_def}
\nu_p=p \sqrt{2\lambda+p^2-1}, \qquad\text{for }p^2\ge 1-2\lambda\text{ and } p\in\mathsf{P}_n. 
\end{align}
\begin{itemize}
\item If $n$ is odd, the dimension of the solution space $\LSrad{n}^{(\lambda;\nu)}$ is
\begin{align*} %\label{eqn::soltuionspace_decomposition_odd_dimension}
	\dim\LSrad{n}^{(\lambda;\nu)}=\begin{cases}
		\binom{n}{(n-p)/2},	&\nu\in\{\pm \nu_p\}, \quad p^2\ge 1-2\lambda,\quad p\in\mathsf{P}_n;\\
		0,&\text{otherwise}.
	\end{cases}
\end{align*}
Moreover, when $\lambda>0$, we have the following decomposition
\begin{align*}%\label{eqn::soltuionspace_decomposition_odd}
	\LSrad{n}^{(\lambda)}=		\bigoplus_{p\in\mathsf{P}_n}
		\left(\LSrad{n}^{\left(\lambda; \nu_p \right)}
		\oplus\LSrad{n}^{\left(\lambda; -\nu_p \right)}\right).
\end{align*}
\item If $n$ is even, the dimension of the solution space $\LSrad{n}^{(\lambda;\nu)}$ is
\begin{align*} %\label{eqn::soltuionspace_decomposition_even_dimension}
	\dim\LSrad{n}^{(\lambda;\nu)}=
	\begin{cases}
		\binom{n}{(n-p)/2},	& \nu\in\{\pm \nu_p\},\quad	p^2>1-2\lambda,\quad p\in\mathsf{P}_n;\\
		\binom{n}{n/2},&\nu=0,\quad \lambda\in\R;\\
		0,&\text{otherwise}.
	\end{cases}
\end{align*}
Moreover, when $\lambda>-\frac{3}{2}$, we have the following decomposition
\begin{align*}%\label{eqn::soltuionspace_decomposition_even}
\LSrad{n}^{(\lambda)}=
		\LSrad{n}^{(\lambda;0)}
		\oplus\bigoplus_{p\in\mathsf{P}_n}
		\left(\LSrad{n}^{\left(\lambda;\nu_p \right)}
		\oplus\LSrad{n}^{\left(\lambda;-\nu_p \right)}\right).
\end{align*}
\end{itemize}
\end{theorem}
\begin{figure}[ht!]
	\centering
	\includegraphics[width=\textwidth]{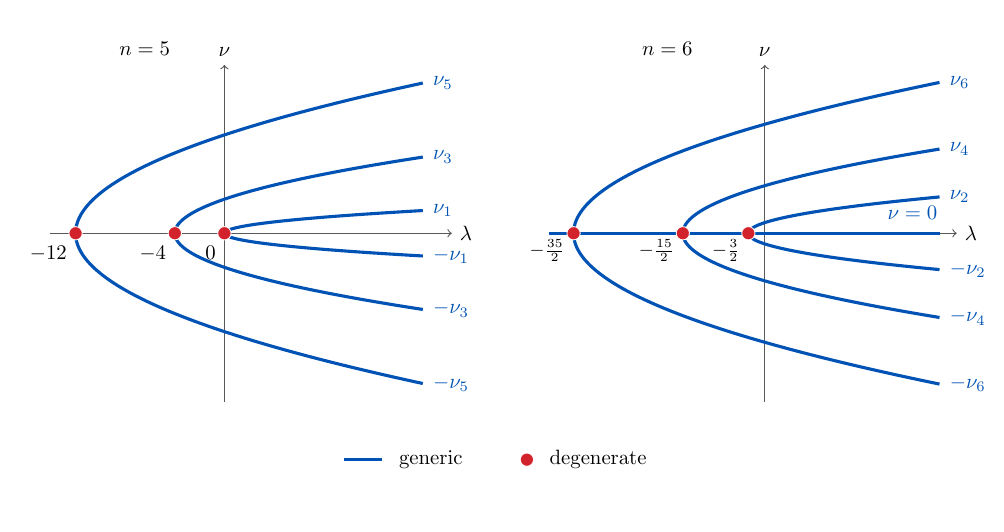}
	\caption{Illustration for the admissible pairs
$(\lambda,\nu)$ in Theorem~\ref{thm::solutionspace_decompositon_dimension}
for $n=5$ and $n=6$. The blue curves correspond to the generic case $\lambda>(1-p^2)/2$ and $\nu_p=\pm p\sqrt{2\lambda+p^2-1}$. The red dots correspond to the degenerate case $\lambda=(1-p^2)/2$ and $\nu=0$. Note that the line $\nu=0$ is admissible for even $n$ and is not admissible for odd $n$ except the degenerate case. }
	\label{fig::ward_spectrum}
\end{figure}

Theorem~\ref{thm::solutionspace_decompositon_dimension} places several previous results on radial BPZ equations into a common framework.
Simmons and Kleban~\cite{SimmonsKlebanCFTsixpointsol} considered the special case $n=4$, $\nu=0$ and $\lambda=(16-\kappa^2)/32$ (corresponding to the bulk field $\Phi_{1/2,0}$ in CFT), after a conformal change of coordinates, they derived the $6$-dimensional solution space $\LSrad{4}^{((16-\kappa^2)/32;0)}$ using Appell hypergeometric functions. For $n=2$, Krusell, Wang, and Wu~\cite{KrusellWangWuCommutationRelation} derive the solution space using some elementary functions and hypergeometric functions. For general $n$ and $\kappa\in (0,8)$, Zhang~\cite[Section~1.3]{zhang2025multipleradialslekappaquantum} constructed families of solutions by Coulomb-gas integrals and left the complete classification as an open problem.  	
Theorem~\ref{thm::solutionspace_decompositon_dimension} gives such a classification for arbitrary $n\ge 2$, $\kappa>0$, and $\lambda\in\mathbb{R}$, including $\kappa\geq8$ (where fewer tools from multiple SLE are available). In CFT language, Theorem~\ref{thm::solutionspace_decompositon_dimension} classifies the solution spaces of the null-vector and rotation Ward equations associated with bulk--boundary correlation functions involving $n$ level-two degenerate boundary fields and one bulk primary field.

The degenerate cases in Theorem~\ref{thm::solutionspace_decompositon_dimension} have two interesting correspondences related to CFT and SLE arm-exponent.
\begin{itemize}
	\item The degenerate values $\lambda=(1-p^2)/2$ correspond to the spinless bulk weights $2\Delta(\mathfrak e_{0,p/2})=\widetilde{ \mathfrak b}-\lambda/\kappa$ in the Kac-type family appearing in multiradial SLE~\cite{HuangPeltolaWuMultiradialSLE}.
	\item For $p=2j$, the same bulk weights become $2\Delta(\mathfrak e_{0,j})=\frac{16j^2-(\kappa-4)^2}{8\kappa}$, which is precisely the interior $2j$-arm exponent for
	$\mathrm{SLE}_\kappa$, see e.g.~\cite{WuAlternatingArmIsing, FengWuYangIsing}. When $\kappa=3$, this specializes to
	$(16j^2-1)/24$, the alternating interior arm exponent of the critical planar Ising model.
\end{itemize}

We prove Theorem~\ref{thm::solutionspace_decompositon_dimension} in Section~\ref{subsec::decomposition}. 
The key idea is to translate the conformal Ward identity~\eqref{eqn::ward_radial} as a finite-dimensional eigenvalue problem. Theorem~\ref{thm::solutionspace_dimension} identifies a solution of the radial BPZ system with its $2^n$ square-free derivatives at one point. Under this identification, the infinitesimal rotation operator $\kappa\sum_j\partial_j$ becomes a $2^n\times2^n$ matrix $B$, so that $\LSrad{n}^{(\lambda;\nu)}$ is naturally identified with the $\nu$-eigenspace of $B$. We first determine the spectrum and its multiplicities for $\kappa\in (0,4]$ and $\lambda>0$ using the positive partition functions constructed in Section~\ref{sec::solutions_radialBPZ}, and then extend the resulting characteristic and annihilating polynomials to all parameters algebraically. The generic case follows from simple roots, while the degenerate cases reduce to an analysis of the Jordan blocks associated with the zero eigenvalue. See Fig.~\ref{fig::ward_spectrum} for the diagram on $(\lambda, \nu)$.

\subsection{GFF level lines}
For $\kappa=4$, the BPZ solution space enjoys a particularly concrete interpretation in terms of level lines of the Gaussian free field (GFF). The coupling between $\mathrm{SLE}_4$ and GFF level lines was established in~\cite{DubedatSLEFreeField,SchrammSheffieldContinuumGFF} and subsequently developed for more general boundary data, see e.g., \cite{WangWuLevellinesGFFI,PeltolaWuGlobalMultipleSLEs,AruSepulvedaWernerBTLStwoDGFF,AruLupuSepulvedaFPSGFF,KarrilaPeltolaSchougGFFdegeneratelevelline}. In this section, we present a concrete example of the relation between solutions to radial BPZ system~\eqref{eqn::BPZ_radial} and GFF level lines. 

Fix $n\ge 2$ and $\kappa=4$. Fix $\mu>0$ and $\sigma=(\sigma_1, \ldots, \sigma_n)\in\{\pm1\}^n$. 
Define 
\begin{align}\label{eqn::LFrad_sigma_def}
\LFrad{n}^{(\sigma; \mu)}(\bs{\theta}):=\prod_{1\le i<\ell\le n}|\ee^{\ii\theta_{\ell}}-\ee^{\ii\theta_i}|^{\frac{1}{2}\sigma_{\ell}\sigma_i}\times \exp\left(\frac{\mu}{4}\sum_{j=1}^n \sigma_j\theta_j\right),\quad \text{for }\bs{\theta}=(\theta_1, \ldots, \theta_n)\in\LX_n^{\U}. 
\end{align}
It can be checked by direct calculation that $\LFrad{n}^{(\sigma;\mu)}$ satisfies the radial BPZ system~\eqref{eqn::BPZ_radial} with 
\begin{align*}
\kappa=4\qquad \text{and}\qquad\lambda=\frac{1}{2}\bigg(\mu^2+1-(\sum_{j=1}^n \sigma_j)^2 \bigg).
\end{align*}
Moreover, these explicit solutions are partition functions for level lines of GFF with properly chosen mean value. 

\begin{proposition}\label{prop::GFF_rad}
Fix $n\ge 2$ and $\kappa=4$. Fix $\mu\in \R$ and $\sigma=(\sigma_1, \ldots, \sigma_n)\in\{\pm1\}^n$. 
Suppose $\Phi$ is Dirichlet GFF in $\U$. For $\bs{\theta}=(\theta_1,\ldots,\theta_n)\in\LX_n^{\U}$, define $\varphi_{\sigma}(\bs{\theta}; \cdot)$ to be the harmonic function on $\U\setminus\{0\}$ by
\begin{equation*}%\label{eqn::sigma_mu_U}
	\varphi_{\sigma}(\bs{\theta};z):=\frac{\pi}{2}\sum_{i=1}^n\sigma_i
	-\frac12\sum_{i=1}^n\sigma_i\arg\left(\frac{\ee^{\ii\theta_i}-z}{1-\overline z\ee^{\ii\theta_i}}\right)
	+\frac12 \sum_{i=1}^n\sigma_i \arg z.
\end{equation*}
The branches of the argument function are chosen continuously on the universal cover of $\U\setminus\{0\}$ so that the boundary data of $\varphi_{\sigma}(\bs{\theta}; \cdot)$ is
\begin{align*}%\label{eqn::data_sigma_rad}
	\pi\sum_{i=1}^j\sigma_i\text{ on }(\ee^{\ii\theta_j}, \ee^{\ii\theta_{j+1}}), \qquad 1\le j\le n-1; \qquad \pi\sum_{i=1}^n\sigma_i\text{ on }(\ee^{\ii\theta_n},\ee^{\ii\theta_1}).
\end{align*}
For $j\in\{1, \ldots, n\}$, the level line of 
\begin{align*}%\label{eqn::GFF_sigma_mu_U}
\Phi(z)+\varphi_{\sigma}(\bs{\theta};z)+\frac{\mu}{2}\log|z|
\end{align*}
starting from $\ee^{\ii\theta_j}$ with height $\pi\sum_{i=1}^{j-1}\sigma_i$ has the same law as the radial Loewner chain with the following driving function, up to the swallowing time of $\{\ee^{\ii\theta_{j-1}}, \ee^{\ii\theta_{j+1}}, 0\}$:
\begin{align}\label{eqn::LFrad_sigma_SDE}
\ud \xi_t=2\ud B_t+4\left(\partial_j\log\LFrad{n}^{(\sigma;\mu)}\right)(\phi_t(\theta_1), \ldots, \phi_t(\theta_{j-1}), \xi_t, \phi_t(\theta_{j+1}), \ldots, \phi_t(\theta_n))\ud t;
\end{align}
where $\xi_0=\theta_j$ and $B_t$ is one-dimensional Brownian motion and $\LFrad{n}^{(\sigma; \mu)}$ is defined in~\eqref{eqn::LFrad_sigma_def}. 
\end{proposition}

In our previous work~\cite{HuangPeltolaWuMultiradialSLE}, we introduced multi-radial SLE with spiral as a family of multiple radial SLEs characterized by a resampling property and boundary perturbation. A natural question is whether this family arises from an underlying statistical-mechanical model. Ideally, one would like to identify a discrete model whose interfaces converge to these curves. Although such a discrete realization remains open, Proposition~\ref{prop::GFF_rad} provides a continuum counterpart at $\kappa=4$: when $\sigma_1=\cdots=\sigma_n=1$, the function $\LFrad{n}^{(\sigma; \mu)}$ in~\eqref{eqn::LFrad_sigma_def} coincides with the partition function of multi-radial $\mathrm{SLE}_4$ with spiral i.e. the partition function $\LZrad{n}^{(\mu)}$ in~\eqref{eqn::LZrad_mu_U} with $\kappa=4$, and the corresponding SLE dynamics arise naturally as level lines of the Gaussian free field. In this sense, the GFF provides a continuum statistical-mechanical realization of the multi-radial $\SLE_4$ with spiral.

We complete the proof of Proposition~\ref{prop::GFF_rad} in Section~\ref{subsec::GFFrad}. The proof follows from the martingale characterization of the GFF--SLE coupling. Starting from the Loewner evolution whose drift is given by the logarithmic derivative of $\LFrad{n}^{(\sigma; \mu)}$, we track the harmonic mean of the field under the Loewner maps and show that it is a local martingale. We then compute its quadratic covariation and verify that it exactly compensates the decrease of the Green's function in the evolving domain.

\paragraph*{Outline.}
We collect preliminaries on multi-chordal SLE and multi-radial SLE in Section~\ref{sec::pre_multiSLE}. 
We construct positive solutions to radial BPZ system in Section~\ref{sec::solutions_radialBPZ}. 
We analyze the dimensions of solution spaces $\LSrad{n}^{(\lambda)}$ and $\LSrad{n}^{(\lambda; \nu)}$ in Section~\ref{sec::PDE_analysis}. We briefly summarize Dub\'edat's commutation relation in Appendix~\ref{sec::commutation_radial}. 
We provide analogous conclusion for Theorem~\ref{thm::solutionspace_dimension} in the chordal setting in Appendix~\ref{sec::chordalBPZ}. 

\paragraph{Acknowledgment.}
We thank Mo Chen and Yilin Wang for helpful discussion. 
H.W. is supported by New Cornerstone Investigator Program 100001127. H.W. is partly affiliated at Yanqi Lake Beijing Institute of Mathematical Sciences and Applications, Beijing, China.

%% file: tex_radial/pre_multiSLE.tex
In this section, we collect preliminaries on multi-chordal SLEs and multi-radial SLEs.
Fix
\begin{align}\label{eqn::parameters}
\kappa\in (0,4], \qquad \mathfrak{b}=\frac{6-\kappa}{2\kappa}, \qquad \tilde{\mathfrak{b}}=\frac{(6-\kappa)(\kappa-2)}{8\kappa}, \qquad \mathfrak{c}=\frac{(6-\kappa)(3\kappa-8)}{2\kappa}. 
\end{align}

\subsection{Preliminaries}
\label{subsec::chordalBPZ_pre}
\paragraph*{Brownian loop measure.} Brownian loop measure is a $\sigma$-finite measure on planar unrooted Brownian loops. We denote this measure by $\blm^\mathrm{loop}$ and refer to~\cite{LawlerWernerBrownianLoopsoup} for its definition and properties. The total mass of $\blm^\mathrm{loop}$ is infinite, but the mass on macroscopic loops is finite: suppose $\Omega$ is a domain and $K_1, K_2$ are two disjoint compact subsets of $\overline{\Omega}$, the total mass of Brownian loops that stay in $\Omega$ and intersect both $K_1$ and $K_2$ is finite and we denote this quantity by $\blm(\Omega; K_1, K_2)$. 
In general, for $n\ge 2$, suppose $K_1, \ldots, K_n$ are disjoint compact subsets of $\overline{\Omega}$, we denote
\begin{equation*}%\label{eqn::blm_def}
\blm(\Omega; K_1, \ldots, K_n)=\sum_{j=2}^n \blm^{\mathrm{loop}}\left(\ell\subset\Omega: \ell\cap K_i\neq \emptyset\text{ for at least }j \text{ of the }i\in\{1, \ldots, n\}\right). 
\end{equation*}
See~\cite{LawlerPartitionFunctionsSLE} and~\cite{PeltolaWangSLELDP} for its properties and see~\cite{DubedatEulerIntegralsCommutingSLEs, DubedatCommutationSLE, KozdronLawlerMultipleSLEs, PeltolaWuGlobalMultipleSLEs} for its alternative forms.

\paragraph*{Polygon.} We say that $(\Omega; \bs{x})=(\Omega; x_1, \ldots, x_n)$ 
is a (topological) $n$-polygon if $\Omega\subsetneq\C$ is simply connected, $\partial\Omega$ is locally connected, and $x_1, x_2, \ldots, x_n \in \partial\Omega$ are distinct points lying counterclockwise along the boundary. 
We say that $(\Omega; \bs{x})=(\Omega; x_1, \ldots, x_n)$ 
is a nice polygon if we assume further that the marked boundary points $x_1, x_2, \ldots, x_n$ lie on $C^{1+\eps}$-boundary segments, for some $\eps>0$, so that derivatives of conformal maps on $\Omega$ are defined there.

\paragraph*{Green's function.}
For upper-half plane $\HH$, the Green's function is defined as 
\begin{equation*}%\label{eqn::Green_H}
	\Green_{\HH}(z,w)= \log\left|\frac{z-\overline{w}}{z-w}\right|,\qquad z,w\in\HH. 
\end{equation*}
For a simply connected domain $\Omega\subsetneq \C$, we define the Green's function via conformal invariance:
\begin{equation*}%\label{eqn::Green_inv}
    \Green(\Omega;z,w)=\Green(\HH;\varphi(z),\varphi(w)), \qquad z,w\in \Omega,
\end{equation*}
where $\varphi$ is any conformal map from $\Omega$ onto $\HH$. In particular, for unit disc $\U$, we have 
\begin{equation}\label{eqn::Green_U}
	\Green_{\U}(z,w)= \log\left|\frac{1-w\overline{z}}{z-w}\right|,\qquad z,w\in\U. 
\end{equation}

\paragraph*{Poisson kernel.}
For 1-polygon $(\HH;x)$ with $z\in\HH$, the Poisson kernel is defined as 
\begin{equation} \label{eqn::Poisson_H}
	\Poisson(\HH;x;z):=\frac{2\Im(z)}{|z-x|^2}, \qquad x\in \R,z\in \HH.
\end{equation}
For nice 1-polygon $(\Omega;x)$ with $z\in\Omega$, we extend its definition via conformal covariance:
\begin{equation} \label{eqn::Poisson_cov}
	\Poisson(\Omega;x;z):= |\varphi'(x)| \Poisson(\HH;\varphi(x);\varphi(z)),
\end{equation}
where $\varphi$ is any conformal map from $\Omega$ to $\HH$.

\paragraph*{Boundary Poisson kernel.}
For the 2-polygon $(\HH; x_1, x_2)$, the boundary Poisson kernel is defined as
\begin{equation}\label{eqn::bPoisson_H}
\Poisson(\HH; x_1, x_2)=|x_1-x_2|^{-2},\qquad x_1, x_2\in\R. 
\end{equation}
For nice 2-polygon $(\Omega; x_1, x_2)$, we extend its definition via conformal covariance: 
\begin{equation}\label{eqn::bPoisson_cov}
\Poisson(\Omega; x_1, x_2)=|\varphi'(x_1)\varphi'(x_2)|\Poisson(\HH; \varphi(x_1), \varphi(x_2)), 
\end{equation}
where $\varphi$ is any conformal map from $\Omega\to \HH$. In particular, for $(\U; \ee^{\ii\theta_1}, \ee^{\ii\theta_2})$, we have
\begin{equation*}
\Poisson(\U; \ee^{\ii\theta_1}, \ee^{\ii\theta_2})=|\ee^{\ii\theta_1}-\ee^{\ii\theta_2}|^{-2}=\left(2\sin\left(\frac{\theta_2-\theta_1}{2}\right)\right)^{-2}.
\end{equation*}
Poisson kernel satisfies the following monotonicity.
Let $(\Omega; x_1, x_2)$ be a nice 2-polygon and let $U\subset\Omega$ be a simply connected subdomain that agrees with $\Omega$ in neighborhoods of $x_1$ and of $x_2$. Then
\begin{equation}\label{eqn::bPoisson_mono}
\Poisson(U; x_1, x_2)\le \Poisson(\Omega; x_1, x_2).
\end{equation}

\paragraph*{Lie bracket of operators.}
For two linear operators (or matrices) $A_i$ and $A_j$, we define commutator of them by
\begin{equation*}
	[A_i,A_j]:=A_i A_j - A_j A_i.
\end{equation*}
\begin{lemma}\label{lem::radialBPZoperator_commute}
We denote the differential operators in the radial BPZ equations~\eqref{eqn::BPZ_radial} by
\begin{equation}\label{eqn::radialBPZ_operatordef}
	\LD_j:=\frac{\kappa}{2}\partial_j^2	+\sum_{\ell\neq j}\left(		\cot\left(\frac{\theta_{\ell}-\theta_j}{2}\right)\partial_{\ell}		-\frac{(6-\kappa)/\kappa}{4\sin^2\left(\frac{\theta_{\ell}-\theta_j}{2}\right)}\right),\qquad \text{for }j\in\{1, \ldots, n\}.
\end{equation}
and denote the differential operators in the conformal Ward identity~\eqref{eqn::ward_radial} by
\begin{equation*} %\label{eqn::conformalWard_operatordef}
	\LR:=\sum_{j=1}^n \partial_j.
\end{equation*}
Then
\begin{align}\label{eqn::radialBPZoperator_commute}
	&[\LD_i,\LD_j]+ \frac{\LD_i-\LD_j}{\sin^2\left(\frac{\theta_i-\theta_j}{2}\right)} =0,\qquad \text{for }i\neq j;\\
	&[\LD_j,\LR]=0,\quad \text{for }j\in\{1, \ldots, n\}.\label{eqn::radialBPZ_Ward_commute}
\end{align}
\end{lemma}

\begin{proof}
This can be checked by direct calculation. See also~\cite[Theorem~1.5 (ii)]{zhang2025multipleradialslekappaquantum}.
\end{proof}

\subsection{Radial Loewner chain}
\label{subsec::radialLoewner}
\paragraph*{Conformal radius.}
For a simply connected domain $\Omega\subsetneq\C$ and $z\in \Omega$, the \emph{conformal radius} of $\Omega$ seen from $z$ is defined by $\CR(\Omega; z) := 1/\phi'(z)$  where $\phi : \Omega \to \U$ is the conformal map such that $\phi(z)=0$ and $\phi'(z)>0$. 
In particular, for $\Omega=\HH$ and $z\in\HH$, we have
\begin{align}\label{eqn::CR_H}
\CR(\HH; z)=2\Im(z). 
\end{align}
Conformal radius is conformally covariant: for any conformal map $\varphi$ on $\Omega$, we have
\begin{align} \label{eqn::CR_cov}
\CR(\Omega; z)=|\varphi'(z)|^{-1}\CR(\varphi(\Omega); \varphi(z)).
\end{align}

\paragraph*{Radial Loewner chain.}
Fix $\theta\in [0,2\pi)$. Suppose $\gamma:[0,T]\to \overline{\U}$ is a continuous simple curve such that $\gamma_0=\ee^{\ii \theta}$ and $\gamma_{(0,T)} \subset \U \setminus \{0\}$. Let $U_t$ be the connected component of $\U\setminus \gamma_{[0,t]}$ containing the origin. Let $\mathfrak{g}_t:U_t \to \U$ be the unique conformal map with $\mathfrak{g}_t(0)=0$ and $\mathfrak{g}'_t(0)>0$. We say that the curve is parameterized by capacity if $\mathfrak{g}'_t(0)=\ee^t$. Then $\mathfrak{g}_t$ satisfies the radial Loewner chain:
\begin{equation*}
	\partial_t \mathfrak{g}_t(z)=\mathfrak{g}_t(z) \frac{\ee^{\ii\xi_t}+\mathfrak{g}_t(z)}{\ee^{\ii\xi_t}-\mathfrak{g}_t(z)}, \quad \mathfrak{g}_0(z)=z,
\end{equation*}
where $t\mapsto \xi_t \in \R$ is continuous and called driving function of $\gamma$. Let $\phi_t$ be the covering map of $\mathfrak{g}_t$, i.e., the continuous function such that $\mathfrak{g}_t(\ee^{\ii \theta})=\ee^{\ii \phi_t(\theta)}$ and $\phi_0(\theta)=\theta$, we have
\begin{equation*}
	\partial_t \phi_t(\theta) =\cot \left( \frac{\phi_t(\theta)-\xi_t}{2} \right).
\end{equation*}

\paragraph*{Radial $\SLE_{\kappa}^{\mu}(\rho)$ process.}
Fix $\kappa>0$ and $\mu\in\R$ and $p\ge 1$.
Suppose $\bs{\rho}=(\rho_2,\ldots,\rho_p)\in \R_{\ge 0}^{p-1}$ and $\bs{\theta}=(\theta_1,\ldots,\theta_p)\in \LX_p^{\U}$.	
The radial $\SLE_{\kappa}^{\mu}(\bs{\rho})$ in $(\U;\ee^{\ii\bs{\theta}};0)$ is defined as the radial Loewner chain $(K_t)_{t\ge 0}$ driven by a continuous function $\xi:[0,\infty)\to \R$ satisfying the SDE system
\begin{equation*}%\label{eqn::radialSLEkappa_rho_sde}
\begin{cases}
\displaystyle 	\ud \xi_t = \sqrt{\kappa} \ud B_t + \sum_{j=2}^{p} \frac{\rho_j}{2} \cot \left( \frac{\xi_t-V_t^j}{2} \right) \ud t + \mu \ud t, \qquad \xi_0=\theta_1, \\
\displaystyle 	\ud V_t^j = \cot \left( \frac{V_t^j-\xi_t}{2} \right) \ud t, \qquad V_0^{j}=\theta_j, \qquad 2\le j\le p, 
\end{cases}
\end{equation*}
where $B_t$ is a standard one-dimensional Brownian motion. In this article, we only consider the case when $\kappa\in (0,4]$ and all weights are non-negative. In such case, the process is well-defined for all time and is almost surely generated by a continuous curve $\gamma$, see e.g.~\cite{MillerSheffieldIG4, HuangPeltolaWuMultiradialSLE}. Moreover, in this case, the curve does not touch the boundary except at the beginning, and the evolution $V_t^{j}$ coincides with $\phi_t(\theta_j)$ for $2\le j\le p$.
For nice $p$-polygon $(\Omega; \bs{x})$ with $\bs{x}=(x_1, \ldots, x_p)$ and $z\in\Omega$, we define radial $\SLE_{\kappa}^{\mu}(\bs{\rho})$ in $(\Omega;\bs{x};z)$ as the pushforward measure of radial $\SLE_{\kappa}^{\mu}(\bs{\rho})$ in $(\U; \ee^{\ii\bs{\theta}};0)$ by the map $\varphi^{-1}$, where $\varphi:\Omega\to \U$ is any conformal map such that $\varphi(z)=0$ and $\varphi(x_j)=\ee^{\ii\theta_j}$ for each $j$.

\begin{lemma}[{\cite[Proposition~2.4]{HuangPeltolaWuMultiradialSLE}}] \label{lem::martingale_implies_BPZ}
Fix $\kappa>0$ and $\lambda\in\R$ and $p\ge 1$.
Assume $\gamma$ is radial $\SLE_{\kappa}$ in $(\U;\ee^{\ii \theta_1};0)$. For a function $\LZ:\LX_{p}^{\U} \to \R$, define
\begin{equation}\label{eqn::BPZ_radial_mart}
	M_t(\LZ)=\LZ (\xi_t, \phi_t(\theta_2), \ldots, \phi_t(\theta_{p})) \mathfrak{g}'_t(0)^{ - \frac{\lambda}{\kappa} } 
	\prod_{j=2}^{p} \phi'_t(\theta_j)^{\mathfrak{b}}.
\end{equation}
Then the process $M(\LZ)$ is a local martingale with respect to $\gamma$ if and only if $\LZ$ is smooth and satisfies the following BPZ equation:
\begin{equation}\label{eqn::BPZ_radial_first}
\LD_1\LZ=\frac{\lambda}{\kappa}\LZ,
\end{equation}
where $\LD_1$ is defined in~\eqref{eqn::radialBPZ_operatordef}.
\end{lemma}

In the following, we present one example of local martingale of the form~\eqref{eqn::BPZ_radial_mart}. 
Fix $\kappa\in (0,4]$ and $\mu\in\R$ and $p\ge 2$.
Recall that $\LZrad{p}^{(\mu)}(\bs{\theta})$ defined in~\eqref{eqn::LZrad_mu_U} satisfies the radial BPZ system~\eqref{eqn::BPZ_radial} with $\lambda=\frac{1}{2}(\mu^2+1-p^2)$. Thus, 
\begin{align}\label{eqn::multiradial_marginal_mart}
	M_t(\LZrad{p}^{(\mu)})=\LZrad{p}^{(\mu)}(\xi_t, \phi_t(\theta_2), \ldots, \phi_t(\theta_p))	\mathfrak{g}'_t(0)^{ \frac{p^2-1-\mu^2}{2\kappa}} \prod_{j=2}^p \phi'_t(\theta_j)^{\mathfrak{b}}
\end{align}
is a local martingale with respect to radial $\SLE_{\kappa}$ in $(\U; \ee^{\ii\theta_1}; 0)$. In fact, the law of radial $\SLE_{\kappa}$ in $(\U; \ee^{\ii\theta_1}; 0)$ weighted by the local martingale $M_t(\LZrad{p}^{(\mu)})$ is radial $\SLE_{\kappa}^{\mu}(2, \ldots, 2)$ in $(\U;\ee^{\ii\bs{\theta}};0)$ (see~\cite[Lemma~2.3]{HuangPeltolaWuMultiradialSLE}) whose driving function $\xi_t$ satisfies 
\begin{align*}
\ud \xi_t=&\sqrt{\kappa}\ud B_t+\kappa\left(\partial_1\log\LZrad{p}^{(\mu)}\right)(\xi_t, \phi_t(\theta_2), \ldots, \phi_t(\theta_p))\ud t\\
=&\sqrt{\kappa}\ud B_t+\sum_{j=2}^{p} \cot \left( \frac{\xi_t-\phi_t(\theta_j)}{2} \right) \ud t + \mu \ud t.%\label{eqn::multiradial_marginal_SDE}
\end{align*}
This process is very special, as it satisfies boundary perturbation property, see~\cite[Lemma~4.4]{HuangPeltolaWuMultiradialSLE}.

\subsection{Multi-chordal SLE}
\label{subsec::multichordalSLE}

\paragraph*{Chordal SLE.}
For $2$-polygon $(\Omega; x_1, x_2)$, we denote by $\chamber(\Omega; x_1, x_2)$ the space of continuous simple curves $\eta\subset\overline{\Omega}$ from $x_1$ to $x_2$ such that $\eta\cap\partial\Omega=\{x_1,x_2\}$. 
Fix $\kappa\in (0,4]$ and $2$-polygon $(\Omega; x_1, x_2)$.
Chordal $\SLE_{\kappa}$ is a continuous simple curve in $\chamber(\Omega;x_1,x_2)$. It has the same law as radial $\SLE_{\kappa}(\kappa-6)$ in $\Omega$ starting from $x_1$ with force point $x_2$. We denote its law by $\Pchordtwo(\Omega; x_1, x_2)$ and its partition function is given by 
\[\LZtwo(\Omega; x_1, x_2)=\Poisson(\Omega; x_1, x_2)^{\mathfrak{b}}.\]

\paragraph*{Curve space.}
Fix a $2N$-polygon $(\Omega;\bs{x})$ with $\bs{x}=(x_1,\ldots,x_{2N})$. We consider $N$ disjoint curves in $\Omega$ connecting $\{x_1, \ldots, x_{2N}\}$ pairwise. 
Topologically, these $N$ curves form a planar pair partition, which we call a link pattern. More precisely, a link pattern $\alpha$ of size $N$ is a pair partition $\alpha=\{\{a_1,b_1\},\ldots,\{a_N,b_N\}\}$ of the set $\{1,\ldots,2N\}$, where there do not exist two links $\{a_i,b_i\},\{a_j,b_j\}\in\alpha$ such that $a_i<a_j<b_i<b_j$. We denote the set of such link patterns by $\LP_N$; in particular,
\[\LP_0=\{\emptyset\}, \qquad 
\#\LP_N=\Cat_N=\frac{1}{N+1}\binom{2N}{N}.
\]
Fix a $2N$-polygon $(\Omega;\bs{x})$ with $\bs{x}=(x_1,\ldots,x_{2N})$ and $\alpha\in\LP_N$. We define the configuration space associated with $\alpha$ by
\begin{equation*}%\label{eqn::curve_space_alpha}
	\chamber_{\alpha}(\Omega;\bs{x}):=\left\{\bs{\eta}=(\eta^1,\ldots,\eta^N):\eta^j\in\chamber(\Omega;x_{a_j},x_{b_j})\text{ for all }j,\quad \eta^i\cap\eta^j=\emptyset\text{ for }i\neq j\right\}.
\end{equation*}
In other words, an element of $\chamber_{\alpha}(\Omega;\bs{x})$ is a family of pairwise disjoint continuous simple curves in $\overline{\Omega}$ whose connectivity is determined by $\alpha$; each curve hits $\partial\Omega$ only at its two marked endpoints.

\paragraph*{Multi-chordal SLE.}
Fix $\kappa\in (0,4]$ and $2N$-polygon $(\Omega; \bs{x})$ with $\bs{x}=(x_1, \ldots, x_{2N})$. 
The $N$-chordal $\SLE_{\kappa}$ in polygon $(\Omega; \bs{x})$ associated with link pattern $\alpha\in\LP_N$ is the unique probability measure on the configuration space $\chamber_{\alpha}(\Omega; \bs{x})$ with resampling property: for each $j\in\{1, \ldots, N\}$, the conditional law of $\eta^j$ given $\{\eta^i: i\neq j\}$ is chordal $\SLE_{\kappa}$ connecting $x_{a_j}$ and $x_{b_j}$ in the connected component $\Omega_j$ of the domain $\Omega\setminus\cup_{i\neq j}\eta^i$ having the end points $x_{a_j}$ and $x_{b_j}$ on its boundary. We denote its law by $\Pchord_{\alpha}=\Pchord_{\alpha}(\Omega; \bs{x})$. 
There are several literature about the existence and uniqueness of $N$-chordal $\SLE_{\kappa}$ for different range of $\kappa$: \cite{MillerSheffieldIG1, MillerSheffieldIG2, PeltolaWuGlobalMultipleSLEs, WuHyperSLE, BeffaraPeltolaWuUniqueness, AngHoldenSunYu2023, ZhanExistenceUniquenessMultipleSLE, FengLiuPeltolaWu2024}. 
They guarantee the existence and uniqueness of $N$-chordal $\SLE_{\kappa}$ for $\kappa\in (0,4]$. In this article, we focus on $\kappa\in (0,4]$. 

\paragraph*{Pure partition functions.} We denote
$\LX_n^{\HH}=\{(x_1, \ldots, x_n)\in\R^n: x_1<\cdots<x_n\}$.
Pure partition functions of multi-chordal $\SLE_{\kappa}$ are the recursive collection $\{\LZ_{\alpha} \colon \alpha \in \bigsqcup_{N\geq 0} \LP_N\}$ 
of functions $\LZ_{\alpha}(\HH; \cdot) \colon \LX_{2N}^{\HH}\to\R$
uniquely determined by the following four properties:
\begin{itemize}
	\item Chordal BPZ system with $n=2N$ and $\lambda=0$: 
	\begin{equation*}
\frac{\kappa}{2}\frac{\partial_j^2\LZ}{\LZ}+\sum_{\ell\neq j}\left(\frac{2}{x_{\ell}-x_j}\frac{\partial_{\ell}\LZ}{\LZ}-\frac{(6-\kappa)/\kappa}{(x_{\ell}-x_j)^2}\right)=0,\qquad \text{for all }1\le j\le 2N. 
\end{equation*}

	\item M\"{o}bius covariance: for all M\"obius maps $\varphi$ of the upper half-plane $\HH$ such that $\varphi(x_{1}) < \cdots < \varphi(x_{2N})$, we have
	\begin{align*}%\label{eqn::PPF_COV}
		\LZ_{\alpha}(\HH; x_{1},\ldots,x_{2N}) = 
		\prod_{j=1}^{2N} \varphi'(x_{j})^{\mathfrak{b}} 
		\times \LZ_{\alpha}(\HH; \varphi(x_{1}),\ldots,\varphi(x_{2N})).
	\end{align*}
	\item Asymptotics: with $\LZ_{\emptyset} \equiv 1$ for the empty link pattern $\emptyset \in \LP_0$, the collection $\{\LZ_{\alpha} \colon \alpha\in\LP_N\}$ satisfies the following recursive asymptotics property. Fix $j \in \{1,2, \ldots, 2N-1\}$ and $y \in (x_{j-1}, x_{j+2})$ (with the convention that $x_0 = -\infty$ and  $x_{2N+1} = +\infty$). We have
	\begin{align*}%\label{eqn::PPF_ASY} 
		\lim_{x_j,x_{j+1}\to y} \frac{\LZ_{\alpha}(\HH; x_1,\ldots, x_{2N})}{ (x_{j+1}-x_j)^{-2\mathfrak{b}} }
		= 
		\begin{cases}
			\LZ_{\alpha/\{j,j+1\}}(\HH; x_1, \ldots, x_{j-1}, x_{j+2}, \ldots, x_{2N}), 
			& \quad \text{if }\{j, j+1\}\in\alpha , \\
			0 ,
			& \quad \text{if }\{j, j+1\} \not\in \alpha ,
		\end{cases}
	\end{align*}
	where $\alpha/\{k,l\}$ denotes the link pattern in $\LP_{N-1}$ obtained by removing $\{k,l\}$ from $\alpha$ and then relabeling the remaining indices so that they are the first $2(N-1)$ positive integers. 
	\item The functions are positive and satisfy the following power-law bound:
	\begin{align*}%\label{eqn::PPF_PLB}
		0<\LZ_{\alpha}(\HH; x_1, \ldots, x_{2N})\le\prod_{\{a,b\}\in\alpha}|x_a-x_b|^{-2\mathfrak{b}}, \quad \text{for all }x_1<\cdots<x_{2N}. 
	\end{align*}
\end{itemize}

The uniqueness when $\kappa\in (0,4]$ of such collection of functions was proved in~\cite{FloresKlebanPDE2}. The existence when $\kappa\in (0,4]$ is proved in~\cite{ KytolaPeltolaPurePartitionFunctions, PeltolaWuGlobalMultipleSLEs, WuHyperSLE, AngHoldenSunYu2023, FengLiuPeltolaWu2024}. In this article, we focus on $\kappa\in (0,4]$. 
We extend the definition of $\LZ_{\alpha}$ to general nice $2N$-polygon $(\Omega; \bs{x})=(\Omega; x_1, \ldots, x_{2N})$ as 
\begin{align}\label{eqn::PPF_COV}
	\LZ_{\alpha}(\Omega; x_1, \ldots, x_{2N}):=\prod_{j=1}^{2N}|\varphi'(x_j)|^{\mathfrak{b}}\times \LZ_{\alpha}(\HH; \varphi(x_1), \ldots, \varphi(x_{2N})),
\end{align}
where $\varphi$ is any conformal map from $\Omega$ onto $\HH$ with $\varphi(x_1)<\cdots<\varphi(x_{2N})$. 
The asymptotics become
\begin{align}\label{eqn::PPF_ASY}
\lim_{x_j, x_{j+1}\to y}\frac{\LZ_{\alpha}(\Omega; \bs{x})}{\Poisson(\Omega; x_j, x_{j+1})^{\mathfrak{b}}}=\begin{cases}
\LZ_{\alpha/\{j,j+1\}}(\Omega; \ddot{\bs{x}}_j), &\text{if }\{j,j+1\}\in\alpha,\\
0,&\text{if }\{j,j+1\}\not\in\alpha,
\end{cases}
\end{align}
where $\ddot{\bs{x}}_j=(x_1, \ldots, x_{j-1}, x_{j+2}, \ldots, x_{2N})$. 
The power-law bound becomes
\begin{equation}\label{eqn::PPF_PLB}
0<\LZ_{\alpha}(\Omega; \bs{x})\le \prod_{\{a,b\}\in\alpha}\Poisson(\Omega; x_a, x_b)^{\mathfrak{b}}. 
\end{equation}

The law of  multi-chordal SLE is encoded by the pure partition functions. Fix $N\ge 1$ and $\alpha\in\LP_N$. 
Consider the pure partition functions of $N$-chordal $\SLE_{\kappa}$ associated with $\alpha$. We write
\begin{equation*}
	\LZ_{\alpha}(\bs{\theta}):=\LZ_{\alpha}(\U; \ee^{\ii \theta_1}, \ldots, \ee^{\ii \theta_{2N}}) \qquad \text{for }\bs{\theta}=(\theta_1, \ldots, \theta_{2N}) \in \LX_{2N}^{\U}. 
\end{equation*}	
One may check that $\LZ_{\alpha}(\bs{\theta})$ satisfies~\eqref{eqn::BPZ_radial_first} with $\lambda=\kappa\tilde{\mathfrak{b}}$. Thus,
\begin{equation}\label{eqn::multirchordal_marginal_mart}
	M_t(\LZ_{\alpha})=\LZ_{\alpha} (\xi_t, \phi_t(\theta_2), \ldots, \phi_t(\theta_{2N})) \mathfrak{g}'_t(0)^{-\tilde{\mathfrak{b}} } \prod_{j=2}^{2N} \phi'_t(\theta_j)^{\mathfrak{b}}
\end{equation}
is a local martingale with respect to radial $\SLE_{\kappa}$ in $(\U; \ee^{\ii\theta_1};0)$. 
Assume $\bs{\eta}=(\eta^1,\ldots,\eta^N)$ is $N$-chordal $\SLE_{\kappa}$ associated with $\alpha$ in polygon $(\U; \ee^{\ii \bs{\theta}})$ with $\eta^1$ starting from $\ee^{\ii \theta_1}$. Then the law of $\eta^1$ is the same as radial $\SLE_{\kappa}$ in $(\U; \ee^{\ii\theta_1}; 0)$ weighted by the local martingale $M_{t}(\LZ_{\alpha})$ (see~\cite[Lemma~2.3]{FengWuYangIsing}).

\paragraph*{Boundary perturbation.}
When $\kappa\in (0,4]$, there is a construction of $N$-chordal $\SLE_{\kappa}$ using Brownian loop measure that we describe in Lemma~\ref{lem::multichordalSLE_construction_blm}. Using this construction, it is clear to see the ``boundary perturbation property" of multi-chordal SLE, see Lemma~\ref{lem::multichordalSLE_bp}.  
\begin{lemma}[{\cite[Section~3]{PeltolaWuGlobalMultipleSLEs}}]
\label{lem::multichordalSLE_construction_blm}
Fix $\kappa\in (0,4]$ and $N\ge 1$. 
Fix $\alpha=\{\{a_1, b_1\}, \ldots, \{a_N, b_N\}\}\in\LP_N$ and $2N$-polygon $(\Omega;\bs{x})$ with $\bs{x}=( x_1, \ldots, x_{2N})$. 
For $1\le j\le N$, let $\eta^j\sim\Pchordtwo(\Omega; x_{a_j}, x_{b_j})$ be chordal $\SLE_{\kappa}$ in  $(\Omega; x_{a_j}, x_{b_j})$. 
Denote by $\otimes_{\{a,b\}\in\alpha}\Pchordtwo(\Omega; x_a, x_b)$ the law on $(\eta^1, \ldots, \eta^N)$ under which $\eta^1, \ldots, \eta^N$ are independent. 
Then $\Pchord_{\alpha}(\Omega; \bs{x})$ the law of $N$-chordal $\SLE_{\kappa}$ associated with $\alpha$ is the same as $\otimes_{\{a,b\}\in\alpha}\Pchordtwo(\Omega; x_a, x_b)$ weighted by 
\begin{equation*}
\frac{\prod_{\{a,b\}\in\alpha}\LZtwo(\Omega; x_a, x_b)}{\LZ_{\alpha}(\Omega; \bs{x})}\one\left\{\eta^j\cap\eta^i=\emptyset, \forall i\neq j\right\}\exp\left(\frac{\mathfrak{c}}{2}\blm(\Omega; \eta^1, \ldots, \eta^N)\right). 
\end{equation*}
\end{lemma}

\begin{lemma}\label{lem::multichordalSLE_bp} 
Fix $\kappa\in (0,4]$ and $N\ge 1$. 
Fix $\alpha=\{\{a_1, b_1\}, \ldots, \{a_N, b_N\}\}\in\LP_N$ and $2N$-polygon $(\Omega;\bs{x})$ with $\bs{x}=( x_1, \ldots, x_{2N})$. 
Suppose $K$ is a compact subset of $\overline{\Omega}$ such that $K$ has a positive distance from $\{x_1, \ldots, x_{2N}\}$ and $\{\bs{\eta}\cap K=\emptyset\}$ is not empty. Then, the $N$-chordal $\SLE_{\kappa}$ in the smaller domain $(\Omega\setminus K; \bs{x})$ is absolutely continuous with respect to that in $(\Omega; \bs{x})$, with Radon-Nikodym derivative
\begin{align*}
	%\label{eqn::multichordalSLESLEbp}
\frac{\ud \Pchord_{\alpha}(\Omega\setminus K; \bs{x})}{\ud\Pchord_{\alpha}(\Omega;\bs{x})}=	\frac{\LZ_{\alpha}(\Omega; \bs{x})}{\LZ_{\alpha}(\Omega\setminus K; \bs{x})}\one\left\{ \bs{\eta}\cap K=\emptyset \right\} \exp\left(\frac{\mathfrak{c}}{2} \blm\left(\Omega; \bs{\eta},K\right)\right), 
	\end{align*}
where $\LZ_{\alpha}(\Omega\setminus K; \bs{x})$ is the product of pure partition functions of connected components of $\Omega\setminus K$. 
\end{lemma}
\begin{proof}
See~\cite[Proposition~3.4]{PeltolaWuGlobalMultipleSLEs} for the case when $\Omega\setminus K$ is simply connected. The same proof can be adjusted to the more general setting and we briefly explain it. Suppose there are $m$ connected components of $\Omega\setminus K$ having some of $\{x_1, \ldots, x_{2N}\}$ on the boundary. We denote these $m$ connected components by $\Omega_j$ with $1\le j\le m$ and denote by $I_j=\{a\in\{1, \ldots, 2N\}: x_a\in \partial\Omega_j\}$.  Note that $\cup _{j=1}^m I_j=\{1, \ldots, 2N\}$. As $\{\bs{\eta}\cap K=\emptyset\}$ is not empty, the link pattern $\alpha$ induces  sub-link patterns $\alpha_j=\{\{a,b\}\in\alpha: a,b\in I_j\}$ on $I_j$ for $1\le j\le m$. Denote by $\bs{\eta}_j$ the subset of $\bs{\eta}$ containing the curves connecting $x_a$ and $x_b$ for $\{a,b\}\in\alpha_j$ for $1\le j\le m$. 
The pure partition function for the smaller domain $\Omega\setminus K$ is defined by\footnote{The marked points along $\partial\Omega_j$ or $\partial\Omega$ are clear from the context, and we omit them from the notations.}
\begin{equation}%\label{eqn::PPFproduct}
\LZ_{\alpha}(\Omega\setminus K; \bs{x})=\prod_{j=1}^m \LZ_{\alpha_j}(\Omega_j). 
\end{equation}
Let us compare the law of $\bs{\eta}$ under $\otimes_{j=1}^m\Pchord_{\alpha_j}(\Omega_j)$ and its law under $\Pchord_{\alpha}(\Omega)$.\footnote{As the marked points on the boundary are clear from the context, we omit them from the notations.} 
\begin{itemize}
\item On the one hand, the law of $\bs{\eta}_j$ under $\Pchord_{\alpha_j}(\Omega_j)$ is the same as $\Pchord_{\alpha_j}(\Omega)$ weighted by 
\[\frac{\LZ_{\alpha_j}(\Omega)}{\LZ_{\alpha_j}(\Omega_j)}\one\left\{\bs{\eta}_j\cap K=\emptyset\right\}\exp\left(\frac{\mathfrak{c}}{2}\blm(\Omega; \bs{\eta}_j, K)\right). \]
As a consequence, the law of $\bs{\eta}$ under $\otimes_{j=1}^m\Pchord_{\alpha_j}(\Omega_j)$ is the same as $\otimes_{j=1}^m\Pchord_{\alpha_j}(\Omega)$ weighted by 
\begin{align}\label{eqn::multichordalSLESLE_bp_aux1}
&\frac{\prod_{j=1}^m \LZ_{\alpha_j}(\Omega)}{\prod_{j=1}^m \LZ_{\alpha_j}(\Omega_j)}\one\{\bs{\eta}\cap K=\emptyset\}\exp\left(\frac{\mathfrak{c}}{2}\sum_{j=1}^m \blm(\Omega; \bs{\eta}_j, K)\right).
\end{align}
\item On the other hand, from Lemma~\ref{lem::multichordalSLE_construction_blm}, the law of $\bs{\eta}$ under $\Pchord_{\alpha}(\Omega)$ is the same as $\otimes_{j=1}^m \Pchord_{\alpha_j}(\Omega)$ weighted by 
\begin{align}\label{eqn::multichordalSLESLE_bp_aux2}
\frac{\prod_{j=1}^m \LZ_{\alpha_j}(\Omega)}{\LZ_{\alpha}(\Omega)}\one\left\{\bs{\eta}_j\cap\bs{\eta}_i=\emptyset, \forall i\neq j\right\}\exp\left(\frac{\mathfrak{c}}{2}\blm(\Omega; \bs{\eta}_1, \ldots, \bs{\eta}_m)\right). 
\end{align} 
\end{itemize}
Combining~\eqref{eqn::multichordalSLESLE_bp_aux1} and~\eqref{eqn::multichordalSLESLE_bp_aux2}, the law of $\bs{\eta}$ under $\otimes\Pchord_{\alpha_j}(\Omega_j)$ is the same as the law of $\bs{\eta}$ under $\Pchord_{\alpha}(\Omega)$ weighted by 
\begin{align*}
&\frac{\LZ_{\alpha}(\Omega)}{\prod_{j=1}^m\LZ_{\alpha_j}(\Omega_j)}\one\left\{\bs{\eta}\cap K=\emptyset\right\}\exp\left(\frac{\mathfrak{c}}{2}\sum_{j=1}^m \blm(\Omega; \bs{\eta}_j, K)-\frac{\mathfrak{c}}{2}\blm(\Omega; \bs{\eta}_1, \ldots, \bs{\eta}_m)\right)\\
=&\frac{\LZ_{\alpha}(\Omega; \bs{x})}{\LZ_{\alpha}(\Omega\setminus K; \bs{x})}\one\left\{\bs{\eta}\cap K=\emptyset\right\}\exp\left(\frac{\mathfrak{c}}{2}\sum_{j=1}^m \blm(\Omega; \bs{\eta}_j, K)-\frac{\mathfrak{c}}{2}\blm(\Omega; \bs{\eta}_1, \ldots, \bs{\eta}_m)\right).
\end{align*}
It remains to show 
\begin{equation*}%\label{eqn::multichordalSLESLE_bp_aux3}
\sum_{j=1}^m \blm(\Omega; \bs{\eta}_j, K)-\blm(\Omega; \bs{\eta}_1, \ldots, \bs{\eta}_m)=\blm(\Omega; \bs{\eta}, K). 
\end{equation*}
This is true because any Brownian loop intersecting at least two of $\{\bs{\eta}_1, \ldots, \bs{\eta}_m\}$ must intersect $K$. 
\end{proof}

\subsection{Multi-radial SLE with spiral}
\label{subsec::multiradialSLE}
\paragraph*{Multi-radial SLE with spiral.}
Multi-radial SLE with spiral is introduced in~\cite{HuangPeltolaWuMultiradialSLE} whose definition uses multi-time martingale. As the multi-time martingale has heavy notation, we introduce multi-radial SLE using its equivalent definition. Fix $\kappa\in (0,4]$ and $\mu\in\R$. Fix $p$-polygon $(\Omega; \bs{x})$ with $\bs{x}=(x_1, \ldots, x_p)$ and $z\in\Omega$. 
The $p$-radial $\SLE_{\kappa}^{\mu}$ in $(\Omega; \bs{x}; z)$ is the probability measure on $\bs{\gamma}=(\gamma^1, \ldots, \gamma^p)$ such that 
\begin{itemize}
\item $\gamma^1$ is radial $\SLE_{\kappa}^{\mu}(2, \ldots, 2)$ in $\Omega$ from $x_1$ to $z$ with force points $(x_2, \ldots, x_p)$; 
\item given $\gamma^1$, the conditional law of $(\gamma^2, \ldots, \gamma^p)$ is half-$(p-1)$-watermelon $\SLE_{\kappa}$ in $(\Omega\setminus\gamma^1; x_2, \ldots, x_p; z)$. 
\end{itemize}
We denote the law of $p$-radial $\SLE_{\kappa}^{\mu}$ in $(\Omega; \bs{x}; z)$ by $\Prad{p}^{(\mu)}(\Omega; \bs{x}; z)$. 
The partition function for $p$-radial $\SLE_{\kappa}^{\mu}$ in $(\Omega; \bs{x}; z)$ is defined to be
\begin{equation} \label{eqn::LZrad_mu_def}
	\LZrad{p}^{(\mu)}(\Omega; \bs{x}; z)= \LZrad{p}^{(0)}(\Omega; \bs{x}; z) \times \CR(\Omega;z)^{\frac{\mu^2}{2\kappa}} \exp\left( -\frac{\mu}{\kappa} p \arg \varphi'(z) + \frac{\mu}{\kappa} \sum_{j=1}^{p} \arg \varphi(x_j) \right),
\end{equation}
with
\begin{equation*}\label{eqn::LZrad_def}
	\LZrad{p}^{(0)}(\Omega; \bs{x}; z)= \CR(\Omega;z)^{\frac{(\kappa-4)^2-4p^2}{8\kappa}}\times \prod_{j=1}^p \Poisson(\Omega;x_j;z)^{\frac{2p+4-\kappa}{2\kappa}}\times \prod_{1\le i<\ell\le p} \Poisson(\Omega;x_i,x_{\ell})^{-\frac{1}{\kappa}},
\end{equation*}
where $\CR(\Omega; z)$ is conformal radius~\eqref{eqn::CR_H}-\eqref{eqn::CR_cov}, 
and $\Poisson(\Omega; x_i, z)$ is Poisson kernel~\eqref{eqn::Poisson_H}-\eqref{eqn::Poisson_cov}, 
and $\Poisson(\Omega; x_i, x_j)$ is boundary Poisson kernel~\eqref{eqn::bPoisson_H}-\eqref{eqn::bPoisson_cov}, 
and $\varphi$ is any conformal map from $\Omega$ onto $\U$ with $\varphi(z)=0$. Note that, for $\bs{\theta}=(\theta_1, \ldots, \theta_p)\in\LX_p^{\U}$, the function
$\LZrad{p}^{(\mu)}(\U; \ee^{\ii\bs{\theta}}; 0)$ is the same as $\LZrad{p}^{(\mu)}(\bs{\theta})$ given by~\eqref{eqn::LZrad_mu_U}. 

\begin{lemma}[{\cite[Proposition~1.4]{HuangPeltolaWuMultiradialSLE}}]
\label{lem::multiradialSLE_bp}
Fix $\kappa\in (0,4]$ and $\mu\in\R$. Fix $p\ge 1$ and nice $p$-polygon $(\Omega; \bs{x})$ with $\bs{x}=(x_1, \ldots, x_p)$ and $z\in\Omega$.
Suppose $K$ is a compact subset of $\overline{\Omega}$ such that $\Omega\setminus K$ is simply connected and it coincides with $\Omega$ in a neighborhood of $\{x_1, \ldots, x_p, z\}$. Then, the $p$-radial $\SLE_{\kappa}^{\mu}$ probability measure 
in the smaller polygon $(\Omega\setminus K; \bs{x}; z)$ is absolutely continuous with respect to that in $(\Omega; \bs{x}; z)$, 
with Radon-Nikodym derivative 
\begin{align*}%\label{eqn::multiradial_bp}
	\frac{\ud \Prad{p}^{(\mu)} (\Omega\setminus K;\bs{x};z) }{\ud \Prad{p}^{(\mu)} (\Omega;\bs{x}; z) }(\bs{\gamma})=\frac{\LZrad{p}^{(\mu)}(\Omega; \bs{x};z)}{\LZrad{p}^{(\mu)}(\Omega\setminus K; \bs{x};z)} \, \one\{\bs{\gamma}\cap K=\emptyset\}\, \exp\Big(\frac{\mathfrak{c}}{2} \blm(\Omega; \bs{\gamma}, K)\Big) ,
\end{align*} 
where $\LZrad{p}^{(\mu)}$ is partition function in~\eqref{eqn::LZrad_mu_def}. 
\end{lemma}

%% file: tex_radial/solutions_radial.tex
In this section, we construct two families of solutions to the radial BPZ system~\eqref{eqn::BPZ_radial}: partition function for multi-chordal SLE weighted by the conformal radius, defined in Proposition~\ref{prop::LZalphaCR} in Section~\ref{subsec::multichordalSLE_CR}, and partition function for a mixture of multi-chordal SLE and multi-radial SLE with spiral, defined in Proposition~\ref{prop::LZmix_BPZ_COV_ASY} in Section~\ref{subsec::mix_chordal_radial}. We collect all our solutions in Section~\ref{subsec::enumeration}, and show that they are linearly independent. We complete the proof of Proposition~\ref{prop::GFF_rad} in Section~\ref{subsec::GFFrad}.
Fix parameters the same as in~\eqref{eqn::parameters}:
\begin{align*}
\kappa\in (0,4], \qquad \mathfrak{b}=\frac{6-\kappa}{2\kappa}, \qquad \tilde{\mathfrak{b}}=\frac{(6-\kappa)(\kappa-2)}{8\kappa}, \qquad \mathfrak{c}=\frac{(6-\kappa)(3\kappa-8)}{2\kappa}. 
\end{align*}

\subsection{Multi-chordal SLE weighted by conformal radius}
\label{subsec::multichordalSLE_CR}

Fix $\kappa\in (0,4]$ and $\mu\in \R$. 
Fix $n=2N$ an even number and $\alpha\in\LP_N$. Suppose $(\Omega; \bs{x})$ is nice $2N$-polygon with $\bs{x}=(x_1, \ldots, x_{2N})$ and $z\in\Omega$. 
\begin{itemize}
\item Suppose $\bs{\eta}=(\eta^1, \ldots, \eta^N)$ is $N$-chordal $\SLE_{\kappa}$ in $(\Omega; \bs{x})$ associated with $\alpha$. We also write $\bs{\eta}=\cup_{j=1}^N\eta^j$ and denote its law by $\Pchord_{\alpha}=\Pchord_{\alpha}(\Omega;\bs{x})$. Its partition function is given by $\LZ_{\alpha}(\Omega;\bs{x})$. 
\item The domain $\Omega\setminus\bs{\eta}$ has $(N+1)$ connected components. 
We enumerate these connected components as $\Omega_{\bs{\eta}}^{\ell}$ for $1\le \ell\le N+1$.  
Denote by $\CR(\Omega\setminus\bs{\eta};z)$ the conformal radius of the connected component of $\Omega\setminus\bs{\eta}$ containing $z$. 
\end{itemize}
Define 
\begin{align}\label{eqn::LZalphaCR}
\LZalphaCR(\Omega; \bs{x}; z)=\LZ_{\alpha}(\Omega;\bs{x})\Echord_{\alpha}\left[\one\{z\in\Omega_{\bs{\eta}}^{\ell}\}\CR(\Omega\setminus\bs{\eta};z)^{\frac{\mu}{\kappa}}\right]. 
\end{align}
\begin{proposition}\label{prop::LZalphaCR}
Fix $\kappa\in (0,4]$ and $\mu\in \R$. 
Fix $n=2N$ an even number and $\alpha\in\LP_N$.
\begin{itemize}
\item When $\mu>\frac{\kappa}{8}(\kappa-8)$, the expectation in~\eqref{eqn::LZalphaCR} is finite: there exists a constant $\const_{\kappa}^{(\mu)}\in (0,\infty)$ depending only on $(\kappa, \mu)$ such that 
\begin{align}\label{eqn::LZalphaCR_bound}
\LZalphaCR(\Omega; \bs{x};z)\le N\const_{\kappa}^{(\mu)}\CR(\Omega;z)^{\frac{\mu}{\kappa}}\prod_{\{a,b\}\in\alpha}\Poisson(\Omega; x_a, x_b)^{\mathfrak{b}}.
\end{align}
\item If we write 
$\LZalphaCR(\bs{\theta})=\LZalphaCR(\U; \ee^{\ii\bs{\theta}}; 0)$ for $\bs{\theta}=(\theta_1, \ldots, \theta_n)\in\LX_n^{\U}$,
then $\LZalphaCR: \LX_n^{\U}\to \R_{>0}$ satisfies radial BPZ system~\eqref{eqn::BPZ_radial} with $n=2N$ and $\lambda=\frac{1}{8}(6-\kappa)(\kappa-2)+\mu$: 
\begin{align}\label{eqn::LZalphaCR_BPZ}
\LD_j\LZ=\left(\frac{(6-\kappa)(\kappa-2)}{8\kappa}+\frac{\mu}{\kappa}\right)\LZ, \qquad \text{for }j\in\{1, \ldots, n\}, 
\end{align}
where $\LD_j$ is defined in~\eqref{eqn::radialBPZ_operatordef}. 
\item If we write 
$\LZalphaCR(\bs{\theta})=\LZalphaCR(\U; \ee^{\ii\bs{\theta}}; 0)$ for $\bs{\theta}=(\theta_1, \ldots, \theta_n)\in\LX_n^{\U}$,
then $\LZalphaCR: \LX_n^{\U}\to \R_{>0}$ satisfies the following rotation invariance: for $a\in\R$, 
\begin{align}\label{eqn::LZalphaCR_rotation}
\LZalphaCR(\bs{\theta}+a)=\LZalphaCR(\bs{\theta}),\qquad \text{for }\bs{\theta}\in\LX_n^{\U}. 
\end{align}
\item The function $\LZalphaCR(\Omega; \bs{x}; z)$ has the following asymptotics: fix $j\in \{1, \ldots, n-1\}$ and $y\in (x_{j-1}x_{j+2})$, then
\begin{align}\label{eqn::LZalphaCR_ASY}
\lim_{x_j, x_{j+1}\to y}\frac{\LZalphaCR(\Omega; \bs{x}; z)}{\Poisson(\Omega; x_j, x_{j+1})^{\mathfrak{b}}}=
\begin{cases}
\LZalphaCRhat(\Omega; \ddot{\bs{x}}_j; z), &\text{if $\{j,j+1\}\in\alpha$ and $\partial\Omega_{\bs{\eta}}^{\ell}$ does not contain $(x_{j}x_{j+1})$},\\
0, &\text{if else},
\end{cases}
\end{align}
where $\hat{\alpha}=\alpha/\{j,j+1\}$ and $\ddot{\bs{x}}_j=(x_1, \ldots, x_{j-1}, x_{j+2}, \ldots, x_n)$.
\end{itemize}
\end{proposition}

\begin{lemma}
Fix $\kappa\in (0,4]$ and $\mu>\frac{\kappa}{8}(\kappa-8)$. Fix 2-polygon $(\Omega; x_1, x_2)$ with $z\in\Omega$. 
Suppose $\eta$ is chordal $\SLE_{\kappa}$ in $(\Omega; x_1, x_2)$ and we orient $\eta$ from $x_1$ to $x_2$. 
The domain $\Omega\setminus\eta$ has two connected components. We denote by $\Omega_{\eta}^L$ (resp. $\Omega_{\eta}^R$) the one to the left of $\eta$ (resp. to the right of $\eta$). 
Define 
\begin{align*}
&\LZtwo^{(L;\mu)}(\Omega;x_1, x_2;z)=\LZtwo(\Omega; x_1, x_2)\Echordtwo\left[\one\{z\in\Omega_{\eta}^L\}\CR(\Omega\setminus\eta; z)^{\frac{\mu}{\kappa}}\right],\\
&\LZtwo^{(R;\mu)}(\Omega;x_1, x_2;z)=\LZtwo(\Omega; x_1, x_2)\Echordtwo\left[\one\{z\in\Omega_{\eta}^R\}\CR(\Omega\setminus\eta; z)^{\frac{\mu}{\kappa}}\right],\\
&\LZtwo^{(\mu)}(\Omega; x_1, x_2;z)=\LZtwo(\Omega; x_1, x_2)\Echordtwo\left[\CR(\Omega\setminus\eta; z)^{\frac{\mu}{\kappa}}\right]. 
\end{align*}
Then we have the following properties.
\begin{itemize}
\item There exists a constant $\const_{\kappa}^{(\mu)}\in (0,\infty)$ depending only on $(\kappa, \mu)$ such that 
\begin{align}\label{eqn::LZtwo_mu_bound}
\LZtwo^{(\mu)}(\Omega; x_1, x_2;z)\le \const_{\kappa}^{(\mu)}\CR(\Omega;z)^{\frac{\mu}{\kappa}}\LZtwo(\Omega; x_1, x_2). 
\end{align}
\item Fix $y\in (x_1x_2)$, we have
\begin{align}\label{eqn::LZtwo_mu_left_ASY}
\lim_{x_1,x_2\to y}\frac{\LZtwo^{(L;\mu)}(\Omega; x_1, x_2; z)}{\LZtwo(\Omega;x_1,x_2)}=\CR(\Omega; z)^{\frac{\mu}{\kappa}}. 
\end{align}

\item Fix $y\in (x_1x_2)$, we have
\begin{align}\label{eqn::LZtwo_mu_right_ASY}
\lim_{x_1,x_2\to y}\frac{\LZtwo^{(R;\mu)}(\Omega; x_1, x_2; z)}{\LZtwo(\Omega;x_1,x_2)}=0. 
\end{align}
\end{itemize} 
\end{lemma}
\begin{proof}
See~\cite{SchrammSheffieldWilsonConformalRadii}. The upper bound~\eqref{eqn::LZtwo_mu_bound} is proved in~\cite[Lemma~2.3]{FengWuRadialBPZFKIsing}. 
\end{proof}

\begin{lemma}\label{lem::LZalpha_mu_bound}
Assume the same setup as in Proposition~\ref{prop::LZalphaCR}. The bound in~\eqref{eqn::LZalphaCR_bound} holds. 
\end{lemma}
\begin{proof}
This is proved in~\cite[Lemma~2.1]{FengWuRadialBPZFKIsing}. 
\end{proof}

\begin{figure}[ht!]
\begin{subfigure}[t]{0.3\textwidth}
\begin{center}
\includegraphics[width=0.8\textwidth]{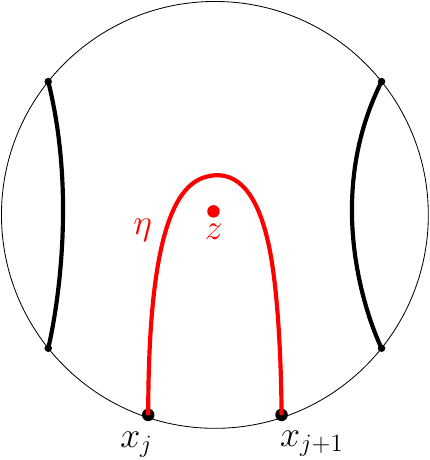}
\end{center}
\caption{Case 1}
\end{subfigure}
$\quad$
\begin{subfigure}[t]{0.3\textwidth}
\begin{center}
\includegraphics[width=0.8\textwidth]{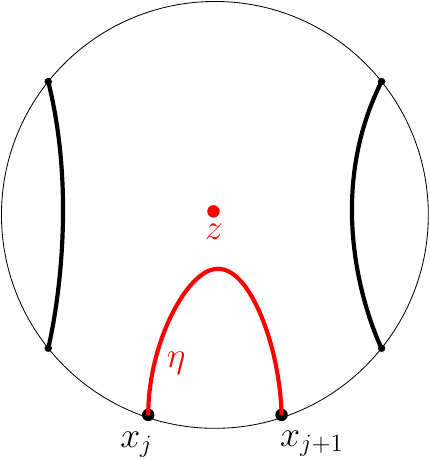}
\end{center}
\caption{Case 2}
\end{subfigure}
$\quad$
\begin{subfigure}[t]{0.3\textwidth}
\begin{center}
\includegraphics[width=0.8\textwidth]{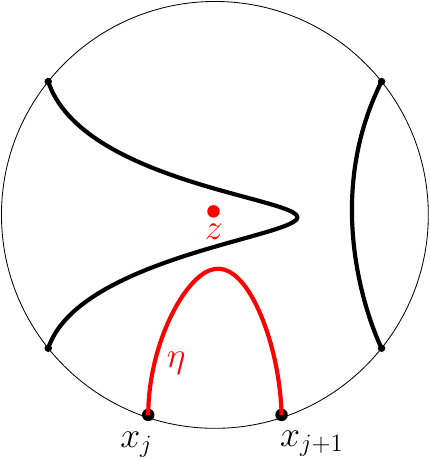}
\end{center}
\caption{Case 3}
\end{subfigure}
\caption{\label{fig::LZalphaCR_ASY} 
The three cases in the proof of Lemma~\ref{lem::LZalphaCR_ASY_aux1}. 
Subfigure~(a) indicates Case 1: $z\in \Omega_{\bs{\eta}}^{\ell}$ and $\partial\Omega_{\bs{\eta}}^{\ell}$ contains $(x_jx_{j+1})$. 
Subfigure~(b) indicates Case 2: $z\in \Omega_{\bs{\eta}}^{\ell}$ and $\partial\Omega_{\bs{\eta}}^{\ell}$ contains $(x_{j-1}x_j)$.
Subfigure~(c) indicates Case 3: $z\in \Omega_{\bs{\eta}}^{\ell}$ and $\partial\Omega_{\bs{\eta}}^{\ell}$ does not contain $(x_{j-1}x_j)$ nor $(x_jx_{j+1})$. 
}
\end{figure}

\begin{lemma}\label{lem::LZalphaCR_ASY_aux1}
Assume the same setup as in Proposition~\ref{prop::LZalphaCR}. Fix $j\in \{1, \ldots, n-1\}$ and $y\in (x_{j-1}x_{j+2})$.
\begin{itemize}
\item If $\{j,j+1\}\in\alpha$ and $\partial\Omega_{\bs{\eta}}^{\ell}$ contains $(x_{j}x_{j+1})$, we have 
\begin{align}\label{eqn::LZalphaCR_ASY_aux1b}
\lim_{x_j, x_{j+1}\to y}\frac{\LZalphaCR(\Omega; \bs{x}; z)}{\Poisson(\Omega; x_j, x_{j+1})^{\mathfrak{b}}}=0. 
\end{align}
\item If $\{j,j+1\}\in\alpha$ and $\partial\Omega_{\bs{\eta}}^{\ell}$ does not contain $(x_{j}x_{j+1})$, we have 
\begin{align}\label{eqn::LZalphaCR_ASY_aux1a}
\lim_{x_j, x_{j+1}\to y}\frac{\LZalphaCR(\Omega; \bs{x}; z)}{\Poisson(\Omega; x_j, x_{j+1})^{\mathfrak{b}}}=\LZalphaCRhat(\Omega; \ddot{\bs{x}}_j; z). 
\end{align}
\end{itemize}
\end{lemma}
\begin{proof}
Suppose $\eta$ is chordal $\SLE_{\kappa}$ in $(\Omega; x_j, x_{j+1})$ and suppose $\hat{\bs{\eta}}$ is $(N-1)$-chordal $\SLE_{\kappa}$ in $(\Omega; \ddot{\bs{x}}_j)$ associated with $\hat{\alpha}$. Denote by $\Pchordtwo\otimes \Pchord_{\hat{\alpha}}$ the probability measure under which $\eta$ and $\hat{\bs{\eta}}$ are independent. 
From Lemma~\ref{lem::multichordalSLE_construction_blm}, 
the law of $(\eta, \hat{\bs{\eta}})$ under $\Pchord_{\alpha}(\Omega; \bs{x})$ is the same as $\Pchordtwo\otimes \Pchord_{\hat{\alpha}}$ weighted by 
\[\frac{\LZtwo(\Omega; x_j, x_{j+1})\LZ_{\hat{\alpha}}(\Omega; \ddot{\bs{x}}_j)}{\LZ_{\alpha}(\Omega; \bs{x})}\one\{\eta\cap\hat{\bs{\eta}}=\emptyset\}\exp\left(\frac{\mathfrak{c}}{2}\blm(\Omega; \eta, \hat{\bs{\eta}})\right).\]
Plugging into~\eqref{eqn::LZalphaCR}, we have
\begin{align}\label{eqn::LZaphaCR_ASY_aux2}
\frac{\LZalphaCR(\Omega; \bs{x}; z)}{\Poisson(\Omega; x_j, x_{j+1})^{\mathfrak{b}}}=\LZ_{\hat{\alpha}}(\Omega; \ddot{\bs{x}}_j)\Echordtwo\otimes\Echord_{\hat{\alpha}}\left[\one\{\eta\cap\hat{\bs{\eta}}=\emptyset\}\exp\left(\frac{\mathfrak{c}}{2}\blm(\Omega; \eta, \hat{\bs{\eta}})\right)\one\{z\in\Omega_{\bs{\eta}}^{\ell}\}\CR(\Omega\setminus\bs{\eta};z)^{\frac{\mu}{\kappa}}\right].
\end{align}

For the RHS of~\eqref{eqn::LZaphaCR_ASY_aux2}, we take conditional expectation with respect to $\hat{\bs{\eta}}$, then the law of $\eta$ is the same as chordal $\SLE_{\kappa}$ in $(\Omega\setminus\hat{\bs{\eta}}; x_j, x_{j+1})$. 
Let us consider the conformal radius $\CR(\Omega\setminus\bs{\eta};z)$. 
There are $(N+1)$ connected components of $\Omega\setminus\bs{\eta}$. Among them, there are two connected components having $\eta$ on the boundary: one has $(x_jx_{j+1})$ on the boundary and one has $(x_{j-1}x_j)$ on the boundary. For the relation between $z$ and $\eta$, there are three cases, see Figure~\ref{fig::LZalphaCR_ASY}. 
\begin{itemize}
\item Case 1: $z\in \Omega_{\bs{\eta}}^{\ell}$ and $\partial\Omega_{\bs{\eta}}^{\ell}$ contains $(x_jx_{j+1})$. In this case, we have $\CR(\Omega\setminus\bs{\eta};z)=\CR(\Omega\setminus\eta;z)$. Thus, 
\begin{align*}
\frac{\LZalphaCR(\Omega; \bs{x}; z)}{\Poisson(\Omega; x_j, x_{j+1})^{\mathfrak{b}}}=\LZ_{\hat{\alpha}}(\Omega; \ddot{\bs{x}}_j)\Echord_{\hat{\alpha}}\left[\frac{\LZtwo^{(R; \mu)}(\Omega\setminus\hat{\bs{\eta}}; x_j, x_{j+1}; z)}{\LZtwo(\Omega; x_j, x_{j+1})}\right]. 
\end{align*}
For the integrand in the RHS, we have almost surely,
\begin{align*}
\lim_{x_j, x_{j+1}\to y}\frac{\LZtwo^{(R;\mu)}(\Omega\setminus\hat{\bs{\eta}}; x_j, x_{j+1}; z)}{\LZtwo(\Omega;x_j,x_{j+1})}=0;\tag{due to~\eqref{eqn::LZtwo_mu_right_ASY}}
\end{align*}
and 
\begin{align*}
\frac{\LZtwo^{(R;\mu)}(\Omega\setminus\hat{\bs{\eta}}; x_j, x_{j+1};z)}{\LZtwo(\Omega;x_j,x_{j+1})}\le \const_{\kappa}^{(\mu)}\CR(\Omega\setminus\hat{\bs{\eta}};z)^{\frac{\mu}{\kappa}};\tag{due to~\eqref{eqn::LZtwo_mu_bound}}
\end{align*}
where $\CR(\Omega\setminus\hat{\bs{\eta}};z)^{\frac{\mu}{\kappa}}$ is integrable due to Lemma~\ref{lem::LZalpha_mu_bound}. Thus, bounded convergence theorem gives~\eqref{eqn::LZalphaCR_ASY_aux1b} as desired. 

\item Case 2: $z\in \Omega_{\bs{\eta}}^{\ell}$ and $\partial\Omega_{\bs{\eta}}^{\ell}$ contains $(x_{j-1}x_j)$. In this case, we have 
\begin{align*}
\frac{\LZalphaCR(\Omega; \bs{x}; z)}{\Poisson(\Omega; x_j, x_{j+1})^{\mathfrak{b}}}=\LZ_{\hat{\alpha}}(\Omega; \ddot{\bs{x}}_j)\Echord_{\hat{\alpha}}\left[\frac{\LZtwo^{(L;\mu)}(\Omega\setminus\hat{\bs{\eta}}; x_j, x_{j+1}; z)}{\LZtwo(\Omega;x_j,x_{j+1})}\one\{z\in\Omega_{\hat{\bs{\eta}}}^{\ell}\}\right].
\end{align*}
For the integrand in the RHS, we have almost surely,
\begin{align*}
\lim_{x_j, x_{j+1}\to y}\frac{\LZtwo^{(L;\mu)}(\Omega\setminus\hat{\bs{\eta}}; x_j, x_{j+1}; z)}{\LZtwo(\Omega;x_j,x_{j+1})}=\CR(\Omega\setminus\hat{\bs{\eta}};z)^{\frac{\mu}{\kappa}};\tag{due to~\eqref{eqn::LZtwo_mu_left_ASY}}
\end{align*}
and 
\begin{align*}
\frac{\LZtwo^{(L;\mu)}(\Omega\setminus\hat{\bs{\eta}}; x_j, x_{j+1}; z)}{\LZtwo(\Omega;x_j,x_{j+1})}\le \const_{\kappa}^{(\mu)}\CR(\Omega\setminus\hat{\bs{\eta}};z)^{\frac{\mu}{\kappa}};\tag{due to~\eqref{eqn::LZtwo_mu_bound}}
\end{align*}
where $\CR(\Omega\setminus\hat{\bs{\eta}};z)^{\frac{\mu}{\kappa}}$ is integrable due to Lemma~\ref{lem::LZalpha_mu_bound}. Thus, bounded convergence theorem guarantees
\begin{align*}
\lim_{x_j, x_{j+1}\to y}\frac{\LZalphaCR(\Omega; \bs{x}; z)}{\Poisson(\Omega; x_j, x_{j+1})^{\mathfrak{b}}}=&\LZ_{\hat{\alpha}}(\Omega; \ddot{\bs{x}}_j)\Echord_{\hat{\alpha}}\left[\one\{z\in\Omega_{\hat{\bs{\eta}}}^{\ell}\}\CR(\Omega\setminus\hat{\bs{\eta}};z)^{\frac{\mu}{\kappa}}\right]
=\LZalphaCRhat(\Omega; \ddot{\bs{x}}_j; z), 
\end{align*}
as desired in~\eqref{eqn::LZalphaCR_ASY_aux1a}. 
\item Case 3: $z\in \Omega_{\bs{\eta}}^{\ell}$ and $\partial\Omega_{\bs{\eta}}^{\ell}$ does not contain $(x_{j-1}x_j)$ nor $(x_jx_{j+1})$. In this case, we have $\CR(\Omega\setminus\bs{\eta};z)=\CR(\Omega\setminus\hat{\bs{\eta}};z)$. Thus, 
\begin{align*}
\frac{\LZalphaCR(\Omega; \bs{x}; z)}{\Poisson(\Omega; x_j, x_{j+1})^{\mathfrak{b}}}=\LZ_{\hat{\alpha}}(\Omega; \ddot{\bs{x}}_j)\Echord_{\hat{\alpha}}\left[\frac{\LZtwo(\Omega\setminus\hat{\bs{\eta}}; x_j, x_{j+1})}{\LZtwo(\Omega; x_j, x_{j+1})}\one\{z\in\Omega_{\hat{\bs{\eta}}}^{\ell}\}\CR(\Omega\setminus\hat{\bs{\eta}};z)^{\frac{\mu}{\kappa}}\right]. 
\end{align*}
For the integrand in the RHS, we have almost surely,
\begin{align*}
\lim_{x_j, x_{j+1}\to y}\frac{\LZtwo(\Omega\setminus\hat{\bs{\eta}}; x_j, x_{j+1})}{\LZtwo(\Omega; x_j, x_{j+1})}=1;
\end{align*}
and
\begin{align*}
\frac{\LZtwo(\Omega\setminus\hat{\bs{\eta}}; x_j, x_{j+1})}{\LZtwo(\Omega; x_j, x_{j+1})}\one\{z\in\Omega_{\hat{\bs{\eta}}}^{\ell}\}\CR(\Omega\setminus\hat{\bs{\eta}};z)^{\frac{\mu}{\kappa}}\le \CR(\Omega\setminus\hat{\bs{\eta}};z)^{\frac{\mu}{\kappa}},  
\end{align*}
where $\CR(\Omega\setminus\hat{\bs{\eta}};z)^{\frac{\mu}{\kappa}}$ is integrable due to Lemma~\ref{lem::LZalpha_mu_bound}. Thus, bounded convergence theorem guarantees
\begin{align*}
\lim_{x_j, x_{j+1}\to y}\frac{\LZalphaCR(\Omega; \bs{x}; z)}{\Poisson(\Omega; x_j, x_{j+1})^{\mathfrak{b}}}=&\LZ_{\hat{\alpha}}(\Omega; \ddot{\bs{x}}_j)\Echord_{\hat{\alpha}}\left[\one\{z\in\Omega_{\hat{\bs{\eta}}}^{\ell}\}\CR(\Omega\setminus\hat{\bs{\eta}};z)^{\frac{\mu}{\kappa}}\right]
=\LZalphaCRhat(\Omega; \ddot{\bs{x}}_j; z), 
\end{align*}
as desired in~\eqref{eqn::LZalphaCR_ASY_aux1a}. 
\end{itemize}
\end{proof}

\begin{proof}[Proof of Proposition~\ref{prop::LZalphaCR}]
The bound~\eqref{eqn::LZalphaCR_bound} is proved in~\cite[Lemma~2.1]{FengWuRadialBPZFKIsing}. 
The BPZ equations~\eqref{eqn::LZalphaCR_BPZ} are proved in~\cite[Lemma~2.2]{FengWuRadialBPZFKIsing}.
The rotation invariance is clear from the definition. 
For the asymptotics~\eqref{eqn::LZalphaCR_ASY}, there are three cases: case~1: $\{j,j+1\}\in\alpha$ and $\partial\Omega_{\bs{\eta}}^{\ell}$ contains $(x_{j}x_{j+1})$; case~2: $\{j,j+1\}\in\alpha$ and $\partial\Omega_{\bs{\eta}}^{\ell}$ does not contain $(x_{j}x_{j+1})$; case~3: $\{j,j+1\}\not\in\alpha$. The first two cases are proved in Lemma~\ref{lem::LZalphaCR_ASY_aux1}. The third case follows from~\eqref{eqn::LZalphaCR_bound}.
\end{proof}

\subsection{Mixture of multi-chordal and multi-radial SLE}
\label{subsec::mix_chordal_radial}
Fix $\kappa\in (0,4]$ and $\mu\in\R$. 
Fix integers $N\ge 0$ and $p\ge 0$ and suppose $n=2N+p\ge 2$.
Fix $\alpha\in\LP_N$. 
Suppose $\bs{s}=\{s_1, \ldots, s_p\}$ is a subset of $\{1, \ldots, n\}$ such that $s_1<\cdots<s_p$. We write $\{v_1, \ldots, v_{2N}\}=\{1, \ldots, n\}\setminus\{s_1, \ldots, s_p\}$ such that $v_1<\cdots<v_{2N}$. Suppose $(\Omega; x_1, \ldots, x_n)$ is nice $n$-polygon and $z\in\Omega$. 
\begin{itemize}
\item Suppose $\bs{\gamma}=(\gamma^1, \ldots, \gamma^p)$ is $p$-radial $\SLE_{\kappa}^{\mu}$ in $(\Omega; x_{s_1}, \ldots, x_{s_p}; z)$. We also write $\bs{\gamma}=\cup_{j=1}^p\gamma^j$ and denote its law by 
$\Prad{p}^{(\mu)}=\Prad{p}^{(\mu)}(\Omega; x_{s_1}, \ldots, x_{s_p}; z)$.
Its partition function is given by 
\[\LZrad{p}^{(\mu)}(\Omega; x_{s_1}, \ldots, x_{s_p}; z).\] 
\item Suppose $\bs{\eta}=(\eta^1, \ldots, \eta^N)$ is $N$-chordal $\SLE_{\kappa}$ in $(\Omega; x_{v_1}, \ldots, x_{v_{2N}})$ associated with $\alpha$. We also write $\bs{\eta}=\cup_{j=1}^N\eta^j$ and denote its law by 
$\Pchord_{\alpha}=\Pchord_{\alpha}(\Omega; x_{v_1}, \ldots, x_{v_{2N}})$.
Its partition function is given by 
\[\LZ_{\alpha}(\Omega; x_{v_1}, \ldots, x_{v_{2N}}).\] 
\item Let $\Prad{p}^{(\mu)}\otimes\Pchord_{\alpha}$ be the measure under which $\bs{\gamma}\sim\Prad{p}^{(\mu)}$ and $\bs{\eta}\sim\Pchord_{\alpha}$ are independent. 
We say that the pair $(\bs{s}, \alpha)$ is allowable if $\{\bs{\gamma}\cap\bs{\eta}=\emptyset\}$ is not empty, see Figure~\ref{fig::Allowable}.
When $(\bs{s}, \alpha)$ is allowable, define 
\begin{align}\label{eqn::LZmix_def}
\begin{split}
\LZmix(\Omega; x_1, \ldots, x_n; z):=&\LZrad{p}^{(\mu)}(\Omega; x_{s_1}, \ldots, x_{s_p}; z)\times\LZ_{\alpha}(\Omega; x_{v_1}, \ldots, x_{v_{2N}})\\
&\times\Erad{p}^{(\mu)}\otimes\Echord_{\alpha}\left[\one\{\bs{\gamma}\cap\bs{\eta}=\emptyset\}\exp\left(\frac{\mathfrak{c}}{2}\blm(\Omega; \bs{\gamma}, \bs{\eta})\right)\right]. 
\end{split}
\end{align}
\end{itemize}

\begin{proposition}\label{prop::LZmix_BPZ_COV_ASY}
Fix $\kappa\in (0,4]$ and $\mu\in\R$ and we use the above notation. 
\begin{itemize}
\item The expectation in~\eqref{eqn::LZmix_def} is finite: 
\begin{align}\label{eqn::LZmix_bound}
\LZmix(\Omega; x_1, \ldots, x_n; z)\le \LZrad{p}^{(\mu)}(\Omega; x_{s_1}, \ldots, x_{s_p}; z)\times \prod_{\{a,b\}\in\alpha}\Poisson(\Omega; x_{v_a}, x_{v_b})^{\mathfrak{b}}. 
\end{align}
\item If we write 
$\LZmix(\bs{\theta})=\LZmix(\U; \ee^{\ii\bs{\theta}}; 0)$ for $\bs{\theta}=(\theta_1, \ldots, \theta_n)\in\LX_n^{\U}$,
then $\LZmix: \LX_n^{\U}\to \R_{>0}$ satisfies radial BPZ system~\eqref{eqn::BPZ_radial} with $\lambda=\frac{1}{2}(1-p^2+\mu^2)$: 
\begin{align}\label{eqn::LZmix_BPZ}
\LD_j\LZ=\left(\frac{1-p^2+\mu^2}{2\kappa}\right)\LZ, \qquad \text{for }j\in\{1, \ldots, n\}, 
\end{align}
where $\LD_j$ is defined in~\eqref{eqn::radialBPZ_operatordef}.
\item If we write 
$\LZmix(\bs{\theta})=\LZmix(\U; \ee^{\ii\bs{\theta}}; 0)$ for $\bs{\theta}=(\theta_1, \ldots, \theta_n)\in\LX_n^{\U}$,
then $\LZmix: \LX_n^{\U}\to \R_{>0}$ satisfies the following rotation covariance: for $a\in\R$, 
\begin{align}\label{eqn::LZmix_rotation}
\LZmix(\bs{\theta}+a)=\LZmix(\bs{\theta})\exp\left(\frac{\mu}{\kappa}pa\right),\qquad \text{for }\bs{\theta}\in\LX_n^{\U}. 
\end{align}
\item The partition function $\LZmix(\Omega; \bs{x}; z)$ has the following asymptotics: fix $j\in \{1, \ldots, n-1\}$ and $y\in (x_{j-1}x_{j+2})$, then
\begin{align}\label{eqn::LZmix_ASY}
\lim_{x_j, x_{j+1}\to y}\frac{\LZmix(\Omega; \bs{x}; z)}{\Poisson(\Omega; x_j, x_{j+1})^{\mathfrak{b}}}=
\begin{cases}
\LZmixhat(\Omega; \ddot{\bs{x}}_j; z), &\text{if $\{j,j+1\}$ are paired in $\alpha$},\\
0, &\text{if else}
\end{cases}
\end{align}
where by $\{j,j+1\}$ paired in $\alpha$, we mean that there is $i$ such that $v_i=j, v_{i+1}=j+1$ and $\{i,i+1\}\in\alpha$, and $\hat{\alpha}=\alpha/\{i,i+1\}$ and $\ddot{\bs{x}}_j=(x_1, \ldots, x_{j-1}, x_{j+2}, \ldots, x_n)$.
\end{itemize}
\end{proposition}

\begin{figure}[ht!]
\begin{subfigure}[t]{0.15\textwidth}
\begin{center}
\includegraphics[width=\textwidth]{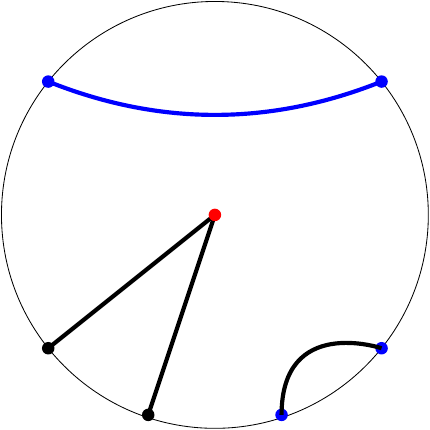}
\end{center}
\caption{}
\end{subfigure}
\begin{subfigure}[t]{0.15\textwidth}
\begin{center}
\includegraphics[width=\textwidth]{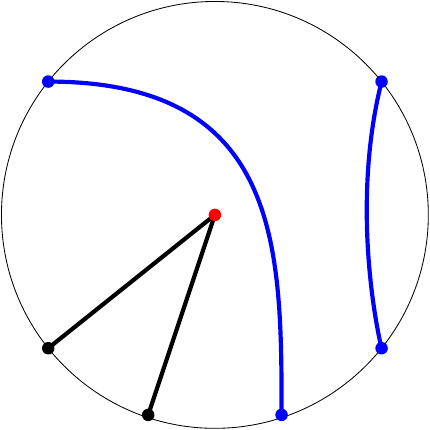}
\end{center}
\caption{}
\end{subfigure}
\begin{subfigure}[t]{0.15\textwidth}
\begin{center}
\includegraphics[width=\textwidth]{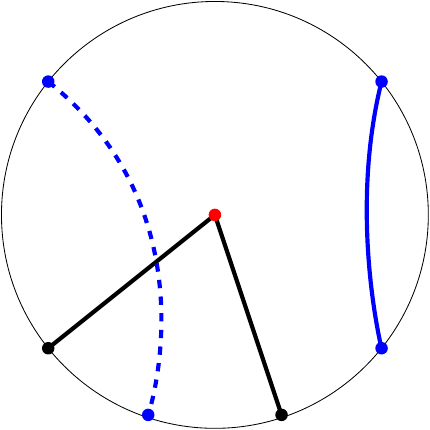}
\end{center}
\caption{}
\end{subfigure}
\begin{subfigure}[t]{0.15\textwidth}
\begin{center}
\includegraphics[width=\textwidth]{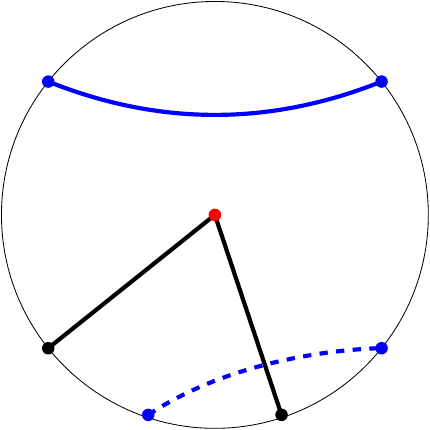}
\end{center}
\caption{}
\end{subfigure}
\begin{subfigure}[t]{0.15\textwidth}
\begin{center}
\includegraphics[width=\textwidth]{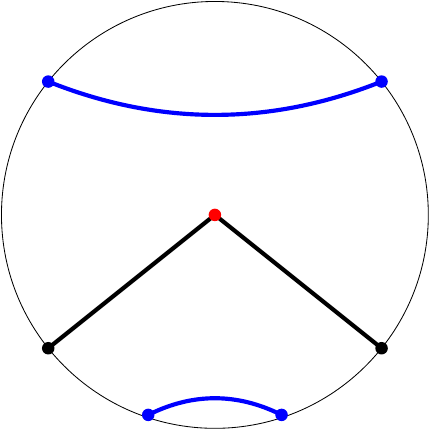}
\end{center}
\caption{}
\end{subfigure}
\begin{subfigure}[t]{0.15\textwidth}
\begin{center}
\includegraphics[width=\textwidth]{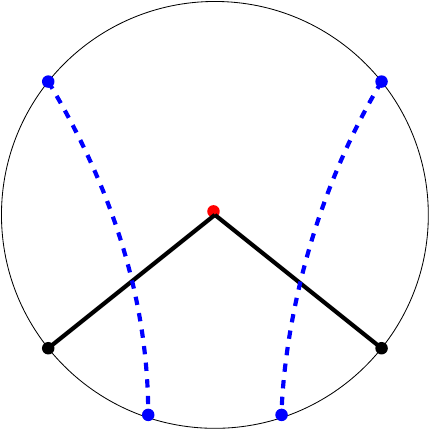}
\end{center}
\caption{}
\end{subfigure}

\caption{\label{fig::Allowable} 
Allowable link patterns for $n=6$ and $p=2$: red dot indicates $z$; black dots indicate $\{x_{s_1}, x_{s_2}\}$; and blue dots indicate $\{x_{v_1}, x_{v_2}, x_{v_3}, x_{v_4}\}$. The link patterns in (a), (b), (e) are allowable. The link patterns in (c), (d), (f) are not allowed. 
}
\end{figure}

The proof of Proposition~\ref{prop::LZmix_BPZ_COV_ASY} is split into four lemmas. 
\begin{lemma}\label{lem::LZmix_bound}
Assume the same setup as in Proposition~\ref{prop::LZmix_BPZ_COV_ASY}. 
We have
\begin{align}\label{eqn::LZmix_cascade1}
\LZmix(\Omega; x_1, \ldots, x_n; z)
	=&\LZrad{p}^{(\mu)}(\Omega; x_{s_1}, \ldots, x_{s_p}; z)
	\Erad{p}^{(\mu)}\left[\LZ_{\alpha}(\Omega\setminus\bs{\gamma}; x_{v_1}, \ldots, x_{v_{2N}})\right];\\
\label{eqn::LZmix_cascade2}
	\LZmix(\Omega; x_1, \ldots, x_n; z)
	=&\LZ_{\alpha}(\Omega; x_{v_1}, \ldots, x_{v_{2N}})
	\Echord_{\alpha}\left[\LZrad{p}^{(\mu)}(\Omega\setminus\bs{\eta}; x_{s_1}, \ldots, x_{s_p}; z)\right]. 
\end{align}
In particular, the bound~\eqref{eqn::LZmix_bound} holds. 
\end{lemma}
\begin{proof}
Let us consider the pair $(\bs{\gamma}, \bs{\eta})$ in the construction above Proposition~\ref{prop::LZmix_BPZ_COV_ASY}. 
\begin{itemize}
\item From the boundary perturbation in Lemma~\ref{lem::multichordalSLE_bp}, the conditional law of $\bs{\eta}$ given $\bs{\gamma}$ weighted by 
\[\one\{\bs{\gamma}\cap\bs{\eta}=\emptyset\}\exp\left(\frac{\mathfrak{c}}{2}\blm(\Omega; \bs{\gamma}, \bs{\eta})\right)\]
is the same as $N$-chordal $\SLE_{\kappa}$ in the connected components of $\Omega\setminus\bs{\gamma}$. Therefore, 
\begin{align}\label{eqn::etagivengamma}
\LZ_{\alpha}(\Omega; x_{v_1}, \ldots, x_{v_{2N}})\Echord_{\alpha}\left[\one\{\bs{\gamma}\cap\bs{\eta}=\emptyset\}\exp\left(\frac{\mathfrak{c}}{2}\blm(\Omega; \bs{\gamma}, \bs{\eta})\right)\cond \bs{\gamma}\right]=\LZ_{\alpha}(\Omega\setminus\bs{\gamma}; x_{v_1}, \ldots, x_{v_{2N}}). 
\end{align}
This gives~\eqref{eqn::LZmix_cascade1} as desired. 
\item From the boundary perturbation in Lemma~\ref{lem::multiradialSLE_bp}, the conditional law of $\bs{\gamma}$ given $\bs{\eta}$ weighted by 
\[\one\{\bs{\gamma}\cap\bs{\eta}=\emptyset\}\exp\left(\frac{\mathfrak{c}}{2}\blm(\Omega; \bs{\gamma}, \bs{\eta})\right)\]
is the same as $p$-radial $\SLE_{\kappa}^{\mu}$ in the connected component of $\Omega\setminus\bs{\eta}$ containing $z$. Therefore, 
\begin{align*}
\LZrad{p}^{(\mu)}(\Omega; x_{s_1}, \ldots, x_{s_p}; z)
	\Erad{p}^{(\mu)}\left[\one\{\bs{\gamma}\cap\bs{\eta}=\emptyset\}
	\exp\left(\frac{\mathfrak{c}}{2}\blm(\Omega;\bs{\gamma},\bs{\eta})\right) \cond\bs{\eta}\right] 
	= \LZrad{p}^{(\mu)}(\Omega\setminus\bs{\eta}; x_{s_1}, \ldots, x_{s_p}; z).
\end{align*}
This gives~\eqref{eqn::LZmix_cascade2} as desired. 
\end{itemize}
For the RHS of~\eqref{eqn::etagivengamma}, we have
\begin{align*}
\LZ_{\alpha}(\Omega\setminus\bs{\gamma}; x_{v_1}, \ldots, x_{v_{2N}})
\le& \prod_{\{a,b\}\in\alpha}\Poisson(\Omega\setminus\bs{\gamma}; x_{v_a}, x_{v_b})^{\mathfrak{b}}\tag{due to~\eqref{eqn::PPF_PLB}}\\
\le& \prod_{\{a,b\}\in\alpha}\Poisson(\Omega; x_{v_a}, x_{v_b})^{\mathfrak{b}}.\tag{due to~\eqref{eqn::bPoisson_mono}}
\end{align*}
Plugging into~\eqref{eqn::LZmix_cascade1}, we obtain~\eqref{eqn::LZmix_bound} as desired. 
\end{proof}

\begin{lemma}\label{lem::LZmix_BPZ_s}
Assume the same setup as in Proposition~\ref{prop::LZmix_BPZ_COV_ASY}. The function $\LZmix(\bs{\theta})$ satisfies radial BPZ equations~\eqref{eqn::LZmix_BPZ} for $j\in\{s_1, \ldots, s_p\}$.  
\end{lemma}
\begin{proof}
Without loss of generality, we assume $j=s_1=1$. Let $\bs{\theta}=(\theta_1, \ldots, \theta_n)\in\LX_n^{\U}$. 
Let $\bs{\gamma}=(\gamma^1,\ldots,\gamma^p)$ be $p$-radial $\SLE_{\kappa}^{\mu}$ in $(\U;\ee^{\ii\theta_{s_1}}, \ldots, \ee^{\ii\theta_{s_p}};0)$ with $\gamma^1$ starting from $\ee^{\ii \theta_1}$. We parameterize $\gamma^1$ by radial capacity and let $\mathfrak{g}_t$ be the radial mapping-out function and $\phi_t$ its covering map.
We have
\begin{align}
	N_t	:=&\Erad{p}^{(\mu)}\left[\LZ_{\alpha}(\U\setminus\bs{\gamma}; \ee^{\ii\theta_{v_1}},\ldots,\ee^{\ii\theta_{v_{2N}}})\cond\gamma_{[0,t]}^1 \right]\notag\\
	= & \Erad{p}^{(\mu)}\left[\LZ_{\alpha}(\U\setminus \mathfrak{g}_t(\bs{\gamma}); \ee^{\ii \phi_t(\theta_{v_1})},\ldots,\ee^{\ii\phi_t(\theta_{v_{2N}})})\cond\gamma_{[0,t]}^1 \right]
	\prod_{j=1}^{2N}\phi_t'(\theta_{v_j})^{\mathfrak b} \tag{due to~\eqref{eqn::PPF_COV}}\\
	=&\frac{\LZmix(\xi_t,\phi_t(\theta_2),\ldots,\phi_t(\theta_n))}
	{\LZrad{p}^{(\mu)}(\xi_t,\phi_t(\theta_{s_2}),\ldots,\phi_t(\theta_{s_p}))}
	\prod_{j=1}^{2N}\phi_t'(\theta_{v_j})^{\mathfrak b} \tag{due to~\eqref{eqn::LZmix_cascade1}}
\end{align}
is a local martingale with respect to $\gamma^1$. Combining with~\eqref{eqn::multiradial_marginal_mart}, we find
\begin{align*}%\label{eqn::LZmix_common_mart}
	M_t(\LZmix)=\LZmix(\xi_t,\phi_t(\theta_2),\ldots,\phi_t(\theta_n)) \mathfrak g_t'(0)^{\frac{p^2-1-\mu^2}{2\kappa}}
	\prod_{j=2}^n \phi_t'(\theta_j)^{\mathfrak b},
\end{align*}
is a local martingale with respect to the radial $\SLE_{\kappa}$ in $(\U;\ee^{\ii \theta_1};0)$. Lemma~\ref{lem::martingale_implies_BPZ} implies that the partition function $\LZmix(\bs{\theta})$ satisfies radial BPZ equations~\eqref{eqn::LZmix_BPZ} for $j=s_1=1$. By symmetry, we conclude that $\LZmix(\bs{\theta})$ satisfies radial BPZ equations~\eqref{eqn::LZmix_BPZ} for $j\in\{s_1, \ldots, s_p\}$ as desired.
\end{proof}

\begin{lemma}\label{lem::LZmix_BPZ_v}
Assume the same setup as in Proposition~\ref{prop::LZmix_BPZ_COV_ASY} and $N\ge 1$. The function $\LZmix(\bs{\theta})$ satisfies radial BPZ equations~\eqref{eqn::LZmix_BPZ} for $j\in\{v_1, \ldots, v_{2N}\}$.  
\end{lemma}
\begin{proof}
Without loss of generality, we assume $j=v_1=1$. Let $\bs{\theta}=(\theta_1, \ldots, \theta_n)\in\LX_n^{\U}$. 
Let $\bs{\eta}=(\eta^1,\ldots,\eta^N)$ be $N$-chordal $\SLE_{\kappa}$ in $(\U;\ee^{\ii\theta_{v_1}}, \ldots, \ee^{\ii\theta_{v_{2N}}})$ with $\eta^1$ starting from $\ee^{\ii \theta_1}$. We parameterize $\eta^1$ by radial capacity and let $\mathfrak{g}_t$ be the radial mapping-out function and $\phi_t$ its covering map.
We have
\begin{align}
	N_t	:=&\Echord_{\alpha}\left[\LZrad{p}^{(\mu)}(\U \setminus \bs{\eta}; \ee^{\ii\theta_{s_1}},\ldots,\ee^{\ii\theta_{s_p}};0)\cond\eta_{[0,t]}^1 \right]\notag\\
	= & \Echord_{\alpha}\left[\LZrad{p}^{(\mu)}(\U \setminus \mathfrak{g}_t(\bs{\eta}); \ee^{\ii \phi_t(\theta_{s_1})},\ldots,\ee^{\ii\phi_t(\theta_{s_p})};0)\cond\eta_{[0,t]}^1 \right] \mathfrak g_t'(0)^{\tilde{\mathfrak b}+\frac{p^2-1-\mu^2}{2\kappa}}
	\prod_{j=1}^p\phi_t'(\theta_{s_j})^{\mathfrak b} \tag{due to~\eqref{eqn::LZrad_mu_def}}\\
	=&\frac{\LZmix(\xi_t,\phi_t(\theta_2),\ldots,\phi_t(\theta_n))}
	{\LZ_{\alpha}(\xi_t,\phi_t(\theta_{v_2}),\ldots,\phi_t(\theta_{v_{2N}}))} \mathfrak g_t'(0)^{\tilde{\mathfrak b}+\frac{p^2-1-\mu^2}{2\kappa}}
	\prod_{j=1}^p\phi_t'(\theta_{s_j})^{\mathfrak b} \tag{due to~\eqref{eqn::LZmix_cascade2}}
\end{align}
is a local martingale with respect to $\eta^1$. Combining with~\eqref{eqn::multirchordal_marginal_mart}, we find
\begin{align*}%\label{eqn::LZmix_common_mart}
	M_t(\LZmix)=\LZmix(\xi_t,\phi_t(\theta_2),\ldots,\phi_t(\theta_n)) \mathfrak g_t'(0)^{\frac{p^2-1-\mu^2}{2\kappa}}
	\prod_{j=2}^n \phi_t'(\theta_j)^{\mathfrak b},
\end{align*}
is a local martingale with respect to the radial $\SLE_{\kappa}$ in $(\U;\ee^{\ii \theta_1};0)$. Lemma~\ref{lem::martingale_implies_BPZ} implies that the partition function $\LZmix(\bs{\theta})$ satisfies radial BPZ equations~\eqref{eqn::LZmix_BPZ} for $j=v_1=1$. By symmetry, we conclude that $\LZmix(\bs{\theta})$ satisfies radial BPZ equations~\eqref{eqn::LZmix_BPZ} for $j\in\{v_1, \ldots, v_{2N}\}$ as desired.
\end{proof}

\begin{lemma}\label{lem::LZmix_ASY}
Assume the same setup as in Proposition~\ref{prop::LZmix_BPZ_COV_ASY}. Fix $j\in \{1, \ldots, n-1\}$ and $y\in (x_{j-1}x_{j+2})$. If $\{j,j+1\}$ are paired in $\alpha$, we have 
\begin{equation*}
\lim_{x_j, x_{j+1}\to y}\frac{\LZmix(\Omega; \bs{x}; z)}{\Poisson(\Omega; x_j, x_{j+1})^{\mathfrak{b}}}=
\LZmixhat(\Omega; \ddot{\bs{x}}_j; z). 
\end{equation*}
\end{lemma}
\begin{proof}
As $\{j,j+1\}$ are paired in $\alpha$, there is $i$ such that $v_i=j, v_{i+1}=j+1$ and $\{i,i+1\}\in\alpha$. We denote $\hat{\alpha}=\alpha/\{i,i+1\}$. From Lemma~\ref{lem::multichordalSLE_bp} and~\eqref{eqn::etagivengamma}, we have 
\begin{align}\label{eqn::LZmix_ASY_paired_aux1}
\LZmix(\Omega;\bs{x};z)=\LZrad{p}^{(\mu)}(\Omega; x_{s_1}, \ldots, x_{s_p}; z)\Erad{p}^{(\mu)}\left[\LZ_{\alpha}(\Omega\setminus\bs{\gamma}; x_{v_1}, \ldots, x_{v_{2N}})\right]. 
\end{align}
Let us consider the asymptotics of $\LZ_{\alpha}(\Omega\setminus\bs{\gamma}; x_{v_1}, \ldots, x_{v_{2N}})$. 
Given $\bs{\gamma}$, from~\eqref{eqn::PPF_ASY}, we have almost surely, 
\begin{align*}
\lim_{x_j, x_{j+1}\to y}\frac{\LZ_{\alpha}(\Omega\setminus\bs{\gamma}; x_{v_1}, \ldots, x_{v_{2N}})}{\Poisson(\Omega; x_j, x_{j+1})^{\mathfrak{b}}}
=&\lim_{x_j, x_{j+1}\to y}\frac{\LZ_{\alpha}(\Omega\setminus\bs{\gamma}; x_{v_1}, \ldots, x_{v_{2N}})}{\Poisson(\Omega\setminus\bs{\gamma}; x_j, x_{j+1})^{\mathfrak{b}}}\frac{\Poisson(\Omega\setminus\bs{\gamma}; x_j, x_{j+1})^{\mathfrak{b}}}{\Poisson(\Omega; x_j, x_{j+1})^{\mathfrak{b}}}\\
=&\LZ_{\hat{\alpha}}(\Omega\setminus\bs{\gamma}; x_{v_1}, \ldots, x_{v_{i-1}}, x_{v_{i+2}}, \ldots, x_{v_{2N}}). 
\end{align*}
From~\eqref{eqn::bPoisson_mono} and~\eqref{eqn::PPF_PLB}, we have 
\begin{align*}
\frac{\LZ_{\alpha}(\Omega\setminus\bs{\gamma}; x_{v_1}, \ldots, x_{v_{2N}})}{\Poisson(\Omega; x_j, x_{j+1})^{\mathfrak{b}}}\le \prod_{\substack{\{a,b\}\in\alpha\\\{a,b\}\neq \{i,i+1\}}}\Poisson(\Omega; x_{v_a}, x_{v_b})^{\mathfrak{b}}. 
\end{align*}
Plugging the two observations into~\eqref{eqn::LZmix_ASY_paired_aux1}, bounded convergence theorem gives
\begin{align*}
\lim_{x_j, x_{j+1}\to y}\frac{\LZmix(\Omega; \bs{x}; z)}{\Poisson(\Omega; x_j, x_{j+1})^{\mathfrak{b}}}=&\LZrad{p}^{(\mu)}(\Omega; x_{s_1}, \ldots, x_{s_p}; z)\Erad{p}^{(\mu)}\left[\LZ_{\hat{\alpha}}(\Omega\setminus\bs{\gamma}; x_{v_1}, \ldots, x_{v_{i-1}}, x_{v_{i+2}}, \ldots, x_{v_{2N}})\right]\\
=&\LZmixhat(\Omega; \ddot{\bs{x}}_j; z),
\end{align*}
as desired. 
\end{proof}

\begin{proof}[Proof of Proposition~\ref{prop::LZmix_BPZ_COV_ASY}]
The upper bound~\eqref{eqn::LZmix_bound} is proved in Lemma~\ref{lem::LZmix_bound}. 
The BPZ equations~\eqref{eqn::LZmix_BPZ} are proved in Lemmas~\ref{lem::LZmix_BPZ_v} and~\ref{lem::LZmix_BPZ_s}. 
The rotation covariance~\eqref{eqn::LZmix_rotation} follows from~\eqref{eqn::LZrad_mu_U} and~\eqref{eqn::LZmix_def}. 
The first case in~\eqref{eqn::LZmix_ASY} is proved in Lemma~\ref{lem::LZmix_ASY}. 
The second case in~\eqref{eqn::LZmix_ASY} follows from~\eqref{eqn::LZmix_bound}.
\end{proof}

\subsection{Enumeration of solutions}
\label{subsec::enumeration}
Fix $\kappa\in (0,4]$ and $n\ge 2$ and $\lambda\in\R$ and consider the radial BPZ system~\eqref{eqn::BPZ_radial}. 
%\begin{align*}
%\frac{\kappa}{2}\frac{\partial_j^2\LZ}{\LZ}+\sum_{\ell\neq j}\left(\cot\left(\frac{\theta_{\ell}-\theta_j}{2}\right)\frac{\partial_{\ell}\LZ}{\LZ}-\frac{(6-\kappa)/\kappa}{4\sin^2\left(\frac{\theta_{\ell}-\theta_j}{2}\right)}\right)=\frac{\lambda}{\kappa}, \qquad \text{for all }1\le j\le n. 
%\end{align*}
In this section, we enumerate its solutions that we constructed in the preceding sections. The enumeration will be different for odd $n$ and for even $n$. 
Recall that $\mathsf{P}_n=\{p\in\{1, \ldots, n\}: n-p\text{ is even}\}$ is defined in~\eqref{eqn::parity_n_p}. Define 
\begin{align}\label{eqn::mu_p_def}
\begin{split}
\mu_0=\lambda-\frac{1}{8}(6-\kappa)(\kappa-2); \qquad 
\mu_p=\sqrt{2\lambda+p^2-1}>0,\qquad\text{for }p\in\mathsf{P}_n.
\end{split}
\end{align}

\begin{proposition}\label{prop::radialsolutions_linearindept_oddn}
Fix $\kappa\in (0,4]$ and $\lambda>0$ and odd $n\ge 3$. 
Recall that $\LZmix$ is the function constructed in Proposition~\ref{prop::LZmix_BPZ_COV_ASY}. 
The collection 
\begin{align}\label{eqn::radialsolutions_oddn}
\bigcup_{p\in\mathsf{P}_n}\{\LZmix(\bs{\theta}): \mu=\pm \mu_p, (\bs{s}, \alpha)\text{ is allowable}\}
\end{align}
contains $2^n$ linearly independent 
solutions to the radial BPZ system~\eqref{eqn::BPZ_radial}.
\end{proposition}

\begin{proposition}\label{prop::radialsolutions_linearindept_evenn}
Fix $\kappa\in (0,4]$ and $\lambda>-\frac{3}{2}$ and even $n\ge 2$. 
Recall that $\LZalphaCR$ is the function constructed in Proposition~\ref{prop::LZalphaCR}
and $\LZmix$ is the function constructed in Proposition~\ref{prop::LZmix_BPZ_COV_ASY}. The collection 
\begin{align}\label{eqn::radialsolutions_evenn}
\{\LZalphaCR(\bs{\theta}): \mu=\mu_0, \alpha\in\LP_{n/2}, 1\le \ell\le n/2+1\}\bigcup_{p\in\mathsf{P}_n}\{\LZmix(\bs{\theta}): \mu=\pm \mu_p, (\bs{s}, \alpha)\text{ is allowable}\}
\end{align}
contains $2^n$ linearly independent
solutions to the radial BPZ system~\eqref{eqn::BPZ_radial}.
\end{proposition}

\begin{lemma}\label{lem::allowable_pair}
Fix $p\ge 1$ and $N\ge 0$ and $n=2N+p$. The number of allowable pairs $(\bs{s},\alpha)$ in Proposition~\ref{prop::LZmix_BPZ_COV_ASY} is $\binom{n}{N}$. 
\end{lemma}
\begin{proof}
Fix $j\in \{1, \ldots, n\}$, and suppose $x_j$ is connected to $z$. We first calculate the number of allowable link patterns for the remaining $2N+p-1$ points.  Denote by $\Cat_m=\frac{1}{m+1}\binom{2m}{m}$ the $m$th Catalan number.  
Assume the same notation as in Proposition~\ref{prop::LZmix_BPZ_COV_ASY}. 
The domain $\Omega\setminus\bs{\gamma}$ has $p$ connected components. In each connected component, the boundary points are connected pairwise by some of $\{\eta^1, \ldots, \eta^N\}$. Thus the number of different link patterns is given by 
\[\sum_{\substack{m_1, \ldots, m_p\ge 0\\ m_1+\cdots+m_p=N}} \Cat_{m_1}\cdots\Cat_{m_p}.\]
This is $p$-fold convolution of Catalan numbers, which is known to be 
\[\binom{2N+p-1}{N}-\binom{2N+p-1}{N-1}=\frac{p (2N+p-1)!}{N!(N+p)!}.\]
\medbreak
Next, let us calculate the number of allowable pairs $(\bs{s}, \alpha)$. 
Summing over $j\in \{1, \ldots, n\}$ in the above enumeration, each allowable link patterns is counted $p$ times. Thus the number of allowable link patterns is
\[\frac{(2N+p)}{p}\frac{p (2N+p-1)!}{N!(N+p)!}=\frac{(2N+p)!}{N!(N+p)!}=\binom{n}{N}.\]
\end{proof}

\begin{proof}[Proof of Proposition~\ref{prop::radialsolutions_linearindept_oddn}]
First, let us check that there are $2^n$ solutions in the collection~\eqref{eqn::radialsolutions_oddn}. 
For each fixed $p\in\mathsf{P}_n$, 
we write $2N=n-p$. There are $2\binom{n}{N}$ solutions in $\{\LZmix(\bs{\theta}): \mu=\pm \mu_p, (\bs{s}, \alpha)\text{ is allowable}\}$ due to Lemma~\ref{lem::allowable_pair}. Thus, the number of solutions in the collection~\eqref{eqn::radialsolutions_oddn} is 
\begin{align*}
\sum_{N=0}^{\lfloor n/2\rfloor}2\binom{n}{N}=\sum_{N=0}^n \binom{n}{N}=2^n
\end{align*}
as desired. 
\medbreak
Next, let us check that the $2^n$ solutions in~\eqref{eqn::radialsolutions_oddn} are linearly independent. 
Suppose $\{\Cmixp, \Cmixpminus\}$ are constants such that 
\begin{align*}
\sum_{p\in\mathsf{P}_n}\sum_{(\bs{s},\alpha)\text{ allowable}}\left(\Cmixp\LZmixp(\bs{\theta})+\Cmixpminus\LZmixpminus(\bs{\theta})\right)=0,\qquad\text{for }\bs{\theta}\in\LX_n^{\U}.
\end{align*}
We apply~\eqref{eqn::LZmix_rotation}, for $a\in\R$, we have, for $\bs{\theta}\in\LX_n^{\U}$, 
\begin{align*}
\sum_{p\in\mathsf{P}_n}\exp\left(\frac{\mu_p}{\kappa}pa\right)\sum_{(\bs{s},\alpha)\text{ allowable}}\Cmixp\LZmixp(\bs{\theta})+\sum_{p\in\mathsf{P}_n}\exp\left(-\frac{\mu_p}{\kappa}pa\right)\sum_{(\bs{s},\alpha)\text{ allowable}}\Cmixpminus\LZmixpminus(\bs{\theta})=0.
\end{align*}
Therefore, for $p\in\mathsf{P}_n$, 
\begin{align*}
\sum_{(\bs{s},\alpha)\text{ allowable}}\Cmixp\LZmixp(\bs{\theta})=0, \quad \sum_{(\bs{s},\alpha)\text{ allowable}}\Cmixpminus\LZmixpminus(\bs{\theta})=0, \quad\text{for }\bs{\theta}\in\LX_n^{\U}. 
\end{align*}
We apply~\eqref{eqn::LZmix_ASY} recursively and obtain $\Cmixp=\Cmixpminus=0$. This completes the proof of linear independence. 
\end{proof}

\begin{proof}[Proof of Proposition~\ref{prop::radialsolutions_linearindept_evenn}]
First, let us check that there are $2^n$ solutions in the collection~\eqref{eqn::radialsolutions_evenn}. 
\begin{itemize}
\item For $p=0$, set $n=2N$. 
The functions in Proposition~\ref{prop::LZalphaCR} are constructed  in this way: we first choose $\alpha\in\LP_N$ and sample $\bs{\eta}=(\eta^1, \ldots, \eta^N)\sim\Pchord_{\alpha}$ as $N$-chordal $\SLE_{\kappa}$ associated with $\alpha$. The domain $\Omega\setminus\bs{\eta}$ has $(N+1)$ connected components and we enumerate them as $\Omega_{\bs{\eta}}^j$ for $1\le j\le N+1$, and define $\LZalphaCR$ as in~\eqref{eqn::LZalphaCR}. 
The number of solutions in this family is 
\begin{equation*} %\label{eqn::num_solutionalphaCR}
	(N+1)\#\LP_N=\binom{2N}{N}=\binom{n}{n/2}.	
\end{equation*}
\item For each fixed $p\in\mathsf{P}_n$, we write $2N=n-p$. There are $2\binom{n}{N}$ solutions in $\{\LZmix(\bs{\theta}): \mu=\pm \mu_p, (\bs{s}, \alpha)\text{ is allowable}\}$ due to Lemma~\ref{lem::allowable_pair}. 
\end{itemize}
Thus, the number of solutions in the collection~\eqref{eqn::radialsolutions_evenn} is
\[\binom{n}{n/2}+\sum_{N=0}^{n/2-1} 2\binom{n}{N}=\sum_{N=0}^n\binom{n}{N}=2^n.\] 

Next, let us check that the $2^n$ solutions are linearly independent. Suppose $\{\CalphaCR, \Cmixp, \Cmixpminus\}$ are constants such that, for $\bs{\theta}\in\LX_n^{\U}$, 
\begin{align*}
\sum_{\alpha\in\LP_{n/2}}\sum_{\ell=1}^{n/2+1}\CalphaCR\LZalphaCR(\bs{\theta})+\sum_{\substack{1\le p\le n\\ p\text{ even}}}\sum_{(\bs{s},\alpha)\text{ allowable}}\left(\Cmixp\LZmixp(\bs{\theta})+\Cmixpminus\LZmixpminus(\bs{\theta})\right)=0.
\end{align*}
We apply~\eqref{eqn::LZalphaCR_rotation} and~\eqref{eqn::LZmix_rotation}, for $a\in\R$, we have, for $\bs{\theta}\in\LX_n^{\U}$, 
\begin{align*}
\sum_{\alpha\in\LP_{n/2}}\sum_{\ell=1}^{n/2+1}\CalphaCR\LZalphaCR(\bs{\theta})
+&\sum_{\substack{1\le p\le n\\ p\text{ even}}}\exp\left(\frac{\mu_p}{\kappa}pa\right)\sum_{(\bs{s},\alpha)\text{ allowable}}\Cmixp\LZmixp(\bs{\theta})\\
+&\sum_{\substack{1\le p\le n\\ p\text{ even}}}\exp\left(-\frac{\mu_p}{\kappa}pa\right)\sum_{(\bs{s},\alpha)\text{ allowable}}\Cmixpminus\LZmixpminus(\bs{\theta})=0.
\end{align*}
Therefore, for all $p\in\mathsf{P}_n$, 
\begin{align*}
\sum_{\alpha\in\LP_{n/2}}\sum_{\ell=1}^{n/2+1}\CalphaCR\LZalphaCR(\bs{\theta})=0, \qquad &\sum_{(\bs{s},\alpha)\text{ allowable}}\Cmixp\LZmixp(\bs{\theta})=0\\
&\sum_{(\bs{s},\alpha)\text{ allowable}}\Cmixpminus\LZmixpminus(\bs{\theta})=0. 
\end{align*}
We apply~\eqref{eqn::LZalphaCR_ASY} recursively and obtain $\CalphaCR=0$. 
We apply~\eqref{eqn::LZmix_ASY} recursively and obtain $\Cmixp=\Cmixpminus=0$. This completes the proof of linear independence. 
\end{proof}

\subsection{GFF level lines: proof of Proposition~\ref{prop::GFF_rad}}
\label{subsec::GFFrad}

\paragraph*{Gaussian free field (GFF).}
For a domain $D\subsetneq \mathbb{C}$ and two functions $ f, g \in L^2(D)$, we denote by $(f, g)$ their inner product in $L^2(D)$. 
We denote by $H_s(D)$ the space of real-valued smooth functions which are compactly supported in $D$. This space has a Dirichlet inner product defined by
\[
	(f, g)_{\nabla} := \frac{1}{2\pi} \int_D \nabla f(z) \cdot \nabla g(z)d^2z.
\]
We denote by $H(D)$ the Hilbert space completion of $ H_s(D) $ with respect to the Dirichlet inner product.
	
The Dirichlet GFF on $D$ is a random sum of the form $\Gamma = \sum_{j=1}^\infty f_j X_j$, where $X_j$ are i.i.d. standard normal random variables and $ (f_j)_{j \geq 1} $ is an orthonormal basis for $H(D)$. This sum almost surely diverges within $H(D)$; however, it does converge almost surely in the space of distributions. The limiting value as a function of \( g \) is almost surely a continuous functional on \( H_s(D) \). See e.g.~\cite{SheffieldGFFMath} and~\cite{BerestyckiPowellGFFandLQG} for more details.  
In general, for any harmonic function $\varphi$ on $D$, we define the GFF with boundary data $\varphi$ by the Dirichlet GFF plus $\varphi$.

\begin{proof}[Proof of Proposition~\ref{prop::GFF_rad}]
Assume $\gamma$ is driven by the radial Loewner chain whose driving function satisfies~\eqref{eqn::LFrad_sigma_SDE}:
\begin{align}\label{eqn::GFF_rad_aux1}
	\begin{split}
		\ud\xi_t=&2\ud B_t+4\left(\partial_j\log\LFrad{n}^{(\sigma;\mu)}\right)
		(\phi_t(\theta_1),\ldots,\phi_t(\theta_{j-1}),\xi_t,\phi_t(\theta_{j+1}),\ldots,\phi_t(\theta_n))\ud t\\
		=&2\ud B_t+\left(\sigma_j\sum_{i\neq j}\sigma_i
		\cot\left(\frac{\xi_t-\phi_t(\theta_i)}{2}\right)+\mu\sigma_j\right)\ud t.
	\end{split}
\end{align}
First, we derive a martingale process for $\gamma$.
For $z\in\U$, let us prove that the process
\begin{align}\label{eqn::GFF_rad_aux2}
\begin{split}
	M_t(z):={}&\frac{\pi}{2}\sum_{i=1}^n\sigma_i	-\frac{\sigma_j}{2}\arg\left(\frac{\ee^{\ii\xi_t}-\mathfrak g_t(z)}{1-\overline{\mathfrak g_t(z)}\ee^{\ii\xi_t}}\right)-\frac12\sum_{i\neq j}\sigma_i\arg\left(		\frac{\ee^{\ii\phi_t(\theta_i)}-\mathfrak g_t(z)}{1-\overline{\mathfrak g_t(z)}\ee^{\ii\phi_t(\theta_i)}}\right)\\
	&+\frac12 \sum_{i=1}^n\sigma_i \arg\mathfrak g_t(z)+\frac{\mu}{2}\log|\mathfrak g_t(z)|,
\end{split}
\end{align}
is a local martingale.
By~\eqref{eqn::GFF_rad_aux1} and It\^o's formula, we have
\begin{align*}
	\ud\arg(\ee^{\ii\xi_t}-\mathfrak g_t(z))
	=&2\Re\left(\frac{\ee^{\ii\xi_t}}{\ee^{\ii\xi_t}-\mathfrak g_t(z)}\right)\ud B_t+\mu\sigma_j\Re\left(\frac{\ee^{\ii\xi_t}}{\ee^{\ii\xi_t}-\mathfrak g_t(z)}\right)\ud t\\
	&-\sigma_j\sum_{i\neq j}\sigma_i\Im\left(	\frac{\ee^{\ii\xi_t}(\ee^{\ii\xi_t}+\ee^{\ii\phi_t(\theta_i)})}{(\ee^{\ii\xi_t}-\mathfrak g_t(z))(\ee^{\ii\xi_t}-\ee^{\ii\phi_t(\theta_i)})}	\right)\ud t +\Im\left(\frac{\mathfrak g_t(z)}{\ee^{\ii\xi_t}-\mathfrak g_t(z)}\right)\ud t,\\
	\ud\arg(\ee^{\ii\phi_t(\theta_i)}-\mathfrak g_t(z))
	=&2\Im\left(\frac{\ee^{2\ii\xi_t}}{(\ee^{\ii\xi_t}-\mathfrak g_t(z))	(\ee^{\ii\xi_t}-\ee^{\ii\phi_t(\theta_i)})}	\right)\ud t,\qquad i\neq j,\\
	\ud\arg\mathfrak g_t(z)
	=&\Im\left(\frac{\ee^{\ii\xi_t}+\mathfrak g_t(z)}{\ee^{\ii\xi_t}-\mathfrak g_t(z)}\right)\ud t, \quad \ud\log|\mathfrak g_t(z)|
	=\Re\left(\frac{\ee^{\ii\xi_t}+\mathfrak g_t(z)}{\ee^{\ii\xi_t}-\mathfrak g_t(z)}\right)\ud t, \\
	\ud\left(\sigma_j\xi_t+\sum_{i\neq j}\sigma_i\phi_t(\theta_i)\right)	=&2\sigma_j\ud B_t+\mu\ud t.
\end{align*}
Plugging these identities into~\eqref{eqn::GFF_rad_aux2}, we obtain
\begin{align}\label{eqn::GFF_rad_martingale_diff}
	\ud M_t(z)=\sigma_j\left(
	1-2\Re\left(\frac{\ee^{\ii\xi_t}}
	{\ee^{\ii\xi_t}-\mathfrak g_t(z)}\right)
	\right)\ud B_t.
\end{align}
Thus, $M_t(z)$ is a local martingale.

\medbreak
Second, let us calculate the variation of Green's function. 
Recall from~\eqref{eqn::Green_U} that $\Green_{\U}(z,w)=\log|(1-w\overline z)/(z-w)|$. Thus, 
\begin{align*}
	\ud \Green_{\U}(\mathfrak g_t(z),\mathfrak g_t(w))
	=&\left(1-2\Re\left(\frac{\ee^{2\ii\xi_t}}{(\ee^{\ii\xi_t}-\mathfrak g_t(z))	(\ee^{\ii\xi_t}-\mathfrak g_t(w))}+\frac{\overline{\mathfrak g_t(z)}\mathfrak g_t(w)}{(\ee^{-\ii\xi_t}-\overline{\mathfrak g_t(z)})(\ee^{\ii\xi_t}-\mathfrak g_t(w))}\right)\right)\ud t\\
	=&-\left(1-2\Re\left(\frac{\ee^{\ii\xi_t}}{\ee^{\ii\xi_t}-\mathfrak g_t(z)}\right)\right) \left(1-2\Re\left( \frac{\ee^{\ii\xi_t}}{\ee^{\ii\xi_t}-\mathfrak g_t(w)} \right)\right)\ud t.
\end{align*}
Comparing with~\eqref{eqn::GFF_rad_martingale_diff}, we have 
\begin{align}\label{eqn::GFF_rad_Hadamard}
	\ud\langle M_t(z),M_t(w)\rangle
	=-\ud \Green_{\U}(\mathfrak g_t(z),\mathfrak g_t(w)).
\end{align}

\medbreak
Finally, we derive the coupling between $\gamma$ and the GFF $\Phi$. Combining~\eqref{eqn::GFF_rad_aux2} and~\eqref{eqn::GFF_rad_Hadamard}, the pair $(\Phi, \gamma)$ can be coupled such that given $\gamma_{[0,t]}$, the field $\Phi(\cdot)+\varphi_{\sigma}(\bs\theta;\cdot)+\frac{\mu}{2}\log|\cdot|$ restricted to $\U\setminus\gamma_{[0,t]}$ has the same law as
\begin{align*}
	\bigg(\Phi(\cdot)+\varphi_{\sigma}(\phi_t(\theta_1),\ldots,\phi_t(\theta_{j-1}),
	\xi_t,\phi_t(\theta_{j+1}),\ldots,\phi_t(\theta_n);\cdot)
	+\frac{\mu}{2}\log|\cdot|\bigg)\circ\mathfrak g_t.
\end{align*}
This confirms the coupling so that $\gamma$ is the level line of
$\Phi+\varphi_{\sigma}(\bs\theta;\cdot)+\frac{\mu}{2}\log|\cdot|$ starting from $\ee^{\ii\theta_j}$ as desired.
\end{proof}

%% file: tex_radial/PDE_analysis.tex
Recall that $\kappa>0$ and $\mathfrak{b}=\frac{6-\kappa}{2\kappa}$. 
We finish the proof for Theorems~\ref{thm::solutionspace_dimension} and~\ref{thm::solutionspace_decompositon_dimension} in this section.
\begin{proof}[Proof of Theorem~\ref{thm::solutionspace_dimension}]
The upper bound on the dimension is proved in Lemma~\ref{lem::dim_upperbound}. 
The lower bound on the dimension is proved in Lemma~\ref{lem::dim_lowerbound}.
Propositions~\ref{prop::radialsolutions_linearindept_oddn} and~\ref{prop::radialsolutions_linearindept_evenn} provide $2^n$ linearly independent positive solutions for the radial BPZ system~\eqref{eqn::BPZ_radial} when $\kappa\in (0,4]$ and $\lambda>0$. 
\end{proof}

\begin{proof}[Proof of Theorem~\ref{thm::solutionspace_decompositon_dimension}]
This is a combination of Lemma~\ref{lem::solutionspace_decompositon_dimension_nondegenerate}, Lemma~\ref{lem::solutionspace_decompositon_dimension_degenerate_oddn}, and Lemma~\ref{lem::solutionspace_decompositon_dimension_degenerate_evenn}.
\end{proof}

\subsection{Upper bound on the dimension}
\label{subsec::dimension_upper}
\begin{lemma} \label{lem::dim_upperbound}
Assume the same setup as in Theorem~\ref{thm::solutionspace_dimension}. We have 
\[\dim\LSrad{n}^{(\lambda)}\le 2^n.\]
\end{lemma}

\begin{proof}
First we prove that any $\LZ\in \LSrad{n}^{(\lambda)}$ is an analytic function in $\LX_n^{\U}$.
Summing over all the equations in~\eqref{eqn::BPZ_radial}, we obtain
\begin{equation}\label{eqn::dim_chordalBPZ_upperbound_aux1}
	\left(	\frac{\kappa}{2}\sum_{j=1}^n\partial_j^2 +\sum_{1\le j<\ell\le n}\cot\left(\frac{\theta_\ell-\theta_j}{2}\right)(\partial_\ell-\partial_j) -\sum_{1\le j<\ell\le n}\frac{\mathfrak{b}}{\sin^2\left(\frac{\theta_\ell-\theta_j}{2}\right)} -\frac{n\lambda}{\kappa} \right)\LZ=0.
\end{equation}
Since $\LZ\in C^2(\LX_n^{\U}\to \R)$ satisfies the elliptic equation~\eqref{eqn::dim_chordalBPZ_upperbound_aux1} with analytic coefficients, the analytic elliptic regularity theory~\cite[Theorem~5.7.1]{MorreyMultipleIntegralsCalculusVariations} (the original result appears in~\cite{MorreyNirenberganalyticityellipticPDE}) implies that $\LZ$ is an analytic function in $\LX_n^{\U}$.
\medbreak	
Next, we prove that $\dim\LSrad{n}^{(\lambda)}\le 2^n$. To this end, we consider the partial derivatives of $\LZ\in \LSrad{n}^{(\lambda)}$ of all orders. Let 
\begin{equation*}
	\bs{\beta}=(\beta_1,\ldots,\beta_n)\in \Z_{\ge 0}^n,\quad |\bs{\beta}|=\sum_{j=1}^n \beta_j, \quad \text{and define } \partial^{\bs{\beta}} \LZ:=\partial_{1}^{\beta_1}\cdots \partial_{n}^{\beta_n} \LZ.
\end{equation*}
Let $\bs{0}:=(0,\ldots,0)$, and let $\bs{e}_j$ be the vector with a single $1$ in the $j$th coordinate and zeros elsewhere.
Fix $\bs{\theta}^0\in \LX_n^{\U}$. We will prove that the following linear map
\begin{equation} \label{eqn::iso_LS_R2n}
	\varsigma: \LSrad{n}^{(\lambda)} \to \R^{2^n}, \quad \LZ \mapsto (\partial^{\bs{\beta}} \LZ(\bs{\theta}^0))_{\bs{\beta}\in \{0,1\}^n}.
\end{equation}
is an injection. Assume $\varsigma(\LZ)=(0,\ldots,0)$. 
Then we have $\partial^{\bs{\beta}} \LZ(\bs{\theta}^0)=0$ for all $\bs{\beta}\in \Z_{\ge 0}^n$ by the following reduction argument.
\begin{itemize}
	\item We have $\partial^{\bs{0}} \LZ(\bs{\theta}^0)=\LZ(\bs{\theta}^0)=0$ since $\varsigma(\LZ)=(0,\ldots,0)$.
	\item Assume, for some $m\ge 1$, that $\partial^{\bs{\beta}'} \LZ(\bs{\theta}^0)=0$ for all $|\bs{\beta}'|< m$. For any $\bs{\beta}\in \Z_{\ge 0}^n$ such that $|\bs{\beta}|=m$, if $\bs{\beta}\in \{0,1\}^n$ then $\partial^{\bs{\beta}} \LZ(\bs{\theta}^0)=0$ since $\varsigma(\LZ)=(0,\ldots,0)$. Otherwise, there exists $j\in \{1,\ldots,n\}$ such that $\beta_j\ge 2$. From~$\eqref{eqn::BPZ_radial}$, we have 
	\begin{equation} \label{eqn::dim_chordalBPZ_upperbound_aux2}
		\partial^{\bs{\beta}} \LZ (\bs{\theta}^0)= \partial^{\bs{\beta}-2\bs{e}_j} \left(-\frac{2}{\kappa}\sum_{\ell\neq j}\cot\left(\frac{\theta_\ell-\theta_j}{2}\right)\partial_\ell +\sum_{\ell\neq j}\frac{\mathfrak{b}/\kappa}{\sin^2\left(\frac{\theta_\ell-\theta_j}{2}\right)} +\frac{2\lambda}{\kappa^2}\right) \LZ(\bs{\theta}^0).
	\end{equation}
	The RHS of~\eqref{eqn::dim_chordalBPZ_upperbound_aux2} is a linear combination of $\partial^{\bs{\beta}'} \LZ(\bs{\theta}^0)$ with $|\bs{\beta}'|<m$. By the induction hypothesis we have $\partial^{\bs{\beta}} \LZ(\bs{\theta}^0)=0$.
\end{itemize}
Since $\LZ$ is analytic and $\partial^{\bs{\beta}} \LZ(\bs{\theta}^0)=0$ for all $\bs{\beta}\in \Z_{\ge 0}^n$, there exists $r>0$ such that 
\begin{equation*}
	\LZ(\bs{\theta})=\sum_{\bs{\beta}\in\Z_{\ge 0}^n}
	\frac{\partial^{\bs{\beta}}\LZ(\bs{\theta}^0)}{\beta_1!\cdots\beta_n!}
	\prod_{j=1}^n(\theta_j-\theta_j^0)^{\beta_j}=0, \quad \text{for all }\bs{\theta}\in B(\bs{\theta}^0,r).
\end{equation*}
Since $\LZ$ is analytic, we obtain that $\LZ(\bs{\theta})=0$ for all $\bs{\theta}\in \LX_n^{\U}$. Therefore, we conclude that $\varsigma$ is an injection, which implies $\dim\LSrad{n}^{(\lambda)}\le \dim \R^{2^n} =2^n$ as desired.
\end{proof}

\subsection{Lower bound on the dimension}
\label{subsec::dimension_lower}
The proof for the lower bound on the dimension of $\LSrad{n}^{(\lambda)}$ consists
 the following Lemma~\ref{lem::radialBPZ_equi_linear_system}--Lemma~\ref{lem::linear_system_injective}. Our main idea is to construct $2^n$ linearly independent solutions to the radial BPZ system~\eqref{eqn::BPZ_radial}. In Lemma~\ref{lem::radialBPZ_equi_linear_system}, we construct a linear system~\eqref{eqn::linear_system_matrix} involving high order square free derivatives for the partition function $\LZ$, which is equivalent to the radial BPZ system~\eqref{eqn::BPZ_radial}. We construct solutions to this linear system~\eqref{eqn::linear_system_matrix} in Lemma~\ref{lem::linear_system_sol_existence} using the existence and uniqueness theorem for initial value problems of ordinary differential equations. The main input of Lemma~\ref{lem::linear_system_sol_existence} is that the linear operators (or matrices) in the linear system~\eqref{eqn::linear_system_matrix} are compatible, which is shown in Lemma~\ref{lem::radialBPZoperator_commute} and Lemma~\ref{lem::linear_system_compatibility}, where we can see the compatibility comes from the commutation relation for the radial BPZ differential operators. In Lemma~\ref{lem::linear_system_injective}, we prove that the solutions we constructed in Lemma~\ref{lem::linear_system_sol_existence} are linearly independent.

Let $\bs{0}:=(0,\ldots,0)$, and let $\bs{e}_j$ be the vector with a single $1$ in the $j$th coordinate and zeros elsewhere. 

\begin{lemma}\label{lem::radialBPZ_equi_linear_system}
	Assume $\LZ\in C^{\infty}(\LX_n^{\U}\to \R)$. Let $\bs{\beta}\in \{0,1\}^n$ and $\bs{u}=(u_{\bs{\beta}})$. Consider the linear system
	\begin{equation} \label{eqn::linear_system_matrix}
		\partial_j\bs{u}(\bs{\theta}) = A_j(\bs{\theta})\bs{u}(\bs{\theta}),	\qquad \text{for }j\in\{1, \ldots, n\},
	\end{equation}
	defined by the following recursion (using Leibniz rule):
	\begin{equation} \label{eqn::linear_system_recursion}
		\begin{cases}
			\partial_j u_{\bs{\beta}} = u_{\bs{\beta}+\bs{e}_j}, &\text{if } \beta_j=0,\\
			\partial_j u_{\bs{\beta}} = \frac{2}{\kappa} \partial^{\bs{\beta}-\bs{e}_j} \left( \frac{\lambda}{\kappa} u_{\bs{0}} -\sum_{\ell\ne j}	\cot\left(\frac{\theta_\ell-\theta_{j}}{2}\right) u_{\bs{e}_{\ell}} + \sum_{\ell\ne j} \frac{2\mathfrak{b}}{4 \sin^2\left(\frac{\theta_\ell-\theta_{j}}{2}\right)}	 u_{\bs{0}} \right), & \text{if }\beta_j=1.
		\end{cases}
	\end{equation}
	The recursion in~\eqref{eqn::linear_system_recursion} uniquely determines the matrix $A_j$.
	Moreover, $\LZ$ satisfies the radial BPZ system~\eqref{eqn::BPZ_radial} if and only if $\bs{u}=(\partial^{\bs{\beta}} \LZ)_{\bs{\beta}\in \{0,1\}^n}$ satisfies the linear system~\eqref{eqn::linear_system_recursion} (or in the matrix form~\eqref{eqn::linear_system_matrix}).
\end{lemma}

\begin{proof}
	We first justify the uniqueness of $A_j$. After applying the Leibniz rule, the right-hand side of the second equation in~\eqref{eqn::linear_system_recursion} only involves terms of the form $\partial^{\bs{\delta}}u_{\bs{0}}$ and $\partial^{\bs{\delta}}u_{\bs{e}_\ell}$, where $\bs{\delta}\in\{0,1\}^n$. The former is $u_{\bs{\delta}}$. For the latter, if $\delta_\ell=0$, then $	\partial^{\bs{\delta}}u_{\bs{e}_\ell}=u_{\bs{\delta}+\bs{e}_\ell}$.
	If $\delta_\ell=1$, write 
	\[\partial^{\bs{\delta}}u_{\bs{e}_\ell}=\partial^{\bs{\delta}-\bs{e}_\ell} \bigl(\partial_\ell u_{\bs{e}_\ell}\bigr)
	\]
	and plug $\bs{\beta}=\bs{e}_\ell$ into the second equation in~\eqref{eqn::linear_system_recursion}. Applying the Leibniz rule again produces terms of the same two types, but with derivative multi-indices of degree at most $|\bs{\delta}|-1$. Induction on $|\bs{\delta}|$ therefore expresses every such term uniquely as a linear combination of the components $u_{\bs{\delta}}$, $\bs{\delta}\in\{0,1\}^n$. Thus every row of $A_j$ is uniquely determined by~\eqref{eqn::linear_system_recursion}. The remaining assertions follow by direct calculations.
\end{proof}

\begin{example}
		For $n=2$, let $\bs u=(u_{00},u_{10},u_{01},u_{11})^{T}$.  We write $A_1$ and $A_2$ explicitly. 
		Set $\vartheta=\frac{1}{2}(\theta_1-\theta_2)$. 
		By~\eqref{eqn::linear_system_recursion}, we have
		\begin{align*}
			A_1={}&
			\begin{pmatrix}
				0&1&0&0\\
				\dfrac{2\lambda}{\kappa^2}+\dfrac{6-\kappa}{2\kappa^2
					\sin^2\vartheta}
				&0&\dfrac{2}{\kappa}\cot\vartheta&0\\
				0&0&0&1\\
				\dfrac{\cot\vartheta}{2\kappa^3}
				\left(8\lambda+\dfrac{(6-\kappa)(\kappa+2)}
				{\sin^2\vartheta}\right)
				&-\dfrac{4}{\kappa^2}\cot^2\vartheta
				&\dfrac{2\lambda}{\kappa^2}+\dfrac{\kappa+6}{2\kappa^2
					\sin^2\vartheta}&0
			\end{pmatrix},\\[1ex]
			A_2={}&
			\begin{pmatrix}
				0&0&1&0\\
				0&0&0&1\\
				\dfrac{2\lambda}{\kappa^2}+\dfrac{6-\kappa}{2\kappa^2
					\sin^2\vartheta}
				&-\dfrac{2}{\kappa}\cot\vartheta&0&0\\
				-\dfrac{\cot\vartheta}{2\kappa^3}
				\left(8\lambda+\dfrac{(6-\kappa)(\kappa+2)}
				{\sin^2\vartheta}\right)
				&\dfrac{2\lambda}{\kappa^2}+\dfrac{\kappa+6}{2\kappa^2
					\sin^2\vartheta}
				&-\dfrac{4}{\kappa^2}\cot^2\vartheta&0
			\end{pmatrix}.
		\end{align*}
		Direct differentiation
		and matrix multiplication give
		\begingroup
		\renewcommand{\arraystretch}{1.8}
		\begin{align*}
			\partial_1A_2-\partial_2A_1
			={}&
			\begin{pmatrix}
				0&0&0&0\\
				\dfrac{(\kappa-6)\cos\vartheta}
				{2\kappa^2\sin^3\vartheta}
				&0&-\dfrac{1}{\kappa\sin^2\vartheta}&0\\
				\dfrac{(\kappa-6)\cos\vartheta}
				{2\kappa^2\sin^3\vartheta}
				&\dfrac{1}{\kappa\sin^2\vartheta}&0&0\\
				0&\dfrac{(2-\kappa)\cos\vartheta}
				{2\kappa^2\sin^3\vartheta}
				&\dfrac{(2-\kappa)\cos\vartheta}
				{2\kappa^2\sin^3\vartheta}&0
			\end{pmatrix}=-[A_2, A_1]. 
		\end{align*}
		\endgroup
		Therefore, we have $\partial_1A_2-\partial_2A_1+[A_2,A_1]=0$, which is exactly~\eqref{eqn::linear_system_compatibility} shown in Lemma~\ref{lem::linear_system_compatibility}.
\end{example}

%{\color{red}For $n=3$, the matrices are two large! Let
%	\[
%	\bs u=(u_{000},u_{100},u_{010},u_{001},u_{110},u_{101},u_{011},u_{111})^{\mathsf T}.
%	\]
%	\input{tex_radial/PDE_analysis_A_matrices}
%}

\begin{lemma}\label{lem::linear_system_compatibility}
The matrices in Lemma~\ref{lem::radialBPZ_equi_linear_system} satisfy
\begin{equation}\label{eqn::linear_system_compatibility}
	\partial_i A_j-\partial_j A_i+[A_j,A_i]=0,	\qquad \text{for }i\ne j.
\end{equation}
\end{lemma}

\begin{proof}
	Consider the following two systems induced by~\eqref{eqn::linear_system_recursion}:
	\begin{equation*} %\label{eqn::linear_system_compatibility_aux1}
		\partial_i \partial_j \bs{u} = \left( (\partial_iA_j)+A_jA_i \right)\bs{u},\qquad\text{and}\qquad\partial_j \partial_i \bs{u} = \left( (\partial_j A_i)+A_i A_j \right)\bs{u}.
	\end{equation*}
	We prove that the two systems are identical. Recall that the matrices $A_j$ in Lemma~\ref{lem::radialBPZ_equi_linear_system} are obtained by the following procedure. 
	\begin{itemize}
		\item Write $\bs{u}=(\partial^{\bs{\beta}} \LZ)_{\bs{\beta}\in \{0,1\}^n}$.
		\item We write $\partial_j\partial^{\bs{\beta}} \LZ$ as a linear combination of $\partial^{\bs{\delta}}$ where $\bs{\delta}\in \{0,1\}^n$, whenever a second order derivative with respect to $j$, $\partial_j^2 \LZ$ occurs in $\partial_j\partial^{\bs{\beta}} \LZ$, we use the radial BPZ equation~\eqref{eqn::BPZ_radial}:
		$(\LD_j-\lambda/\kappa)\LZ=0$, to eliminate it, and continue this procedure using the Leibniz rule, until only the derivatives $\partial^{\bs{\delta}} Z$, $\delta\in\{0,1\}^n$ remain.
		\item Replace $\partial^{\bs{\beta}} \LZ$ by $u_{\bs{\beta}}$.
	\end{itemize}
	By Lemma~\ref{lem::radialBPZ_equi_linear_system}, the procedure above uniquely determines $A_j$. Let us prove that the procedure for the system $\partial_i \partial_j \bs{u} = \left( (\partial_iA_j)+A_jA_i \right)\bs{u}$ and that for the system $\partial_j \partial_i \bs{u} = \left( (\partial_j A_i)+A_iA_j \right)\bs{u}$ agree.
	Indeed, the only possible ambiguity occurs when both $\partial_i^2$ and
	$\partial_j^2$ occur simultaneously in the second step. By~\eqref{eqn::radialBPZoperator_commute}, the difference between eliminating the $i$-derivatives first and eliminating the $j$-derivatives
	first is determined by the relation:
	\[
	(\LD_i-\lambda/\kappa)(\LD_j-\lambda/\kappa)-(\LD_j-\lambda/\kappa)(\LD_i-\lambda/\kappa)=-\frac{(\LD_i-\lambda/\kappa)-(\LD_j-\lambda/\kappa)}{\sin^2((\theta_i-\theta_j)/2)}=0.
	\]
	Applying further derivatives, the Leibniz rule only produces derivatives of these same relations.
	Hence the two systems $\partial_i \partial_j \bs{u} = \left( (\partial_iA_j)+A_jA_i \right)\bs{u}$ and $\partial_j \partial_i \bs{u} = \left( (\partial_j A_i)+A_i A_j \right)\bs{u}$ are identical, which implies~\eqref{eqn::linear_system_compatibility} as desired.
\end{proof}

\begin{lemma}\label{lem::linear_system_sol_existence}
	Fix $\bs{\theta}^{0}\in\LX_n^{\U}$ and $\bs{v}\in\R^{\{0,1\}^n}$. There exists a unique smooth solution $\bs{u}$ of~\eqref{eqn::linear_system_matrix} on $\LX_n^{\U}$ such that $\bs{u}(\bs{\theta}^{0})=\bs{v}$.
\end{lemma}

\begin{proof}
	Let us construct the solutions using the existence and uniqueness theorem for initial value problems of ordinary differential equations. 
	For $\bs{\theta}\in \LX_n^{\U}$, consider a piecewise $C^1$ path $\gamma=(\gamma^1,\ldots,\gamma^n):[0,1]\to\LX_n^{\U}$ from $\bs{\theta}^{0}$ to $\bs{\theta}$. We solve the usual linear ODE for $\bs{w}:[0,1]\to \R^{2^n}$:
	\begin{equation}\label{eqn::linear_system_sol_existence_aux1}
		\dot{\bs{w}}(t)=\sum_{j=1}^n\dot\gamma_j(t) A_j(\gamma(t))\bs{w}(t),
		\qquad \bs{w}(0)=\bs{v}.
	\end{equation}
	We show that $\bs{w}(1)$ only depends on the endpoint $\bs{\theta}$. For two piecewise $C^1$ paths $\gamma$ and $\tilde{\gamma}$, consider the homotopy between them:
	\begin{equation*} 
		H(s,t):= s \gamma(t) + (1-s) \tilde{\gamma}(t), \quad 0\le s\le 1,\, 0\le t\le 1.
	\end{equation*}
	Since $\LX_n^{\U}$ is convex, $H(s,\cdot)=(H_1(s,\cdot),\ldots,H_n(s,\cdot))$ is also a piecewise $C^1$ path from $\bs{\theta}^{0}$ to $\bs{\theta}$. Let $\bs{w}(s,t)$ be the solution of~\eqref{eqn::linear_system_sol_existence_aux1} along the path $H(s,\cdot)$:
	\begin{equation} \label{eqn::linear_system_sol_existence_aux2}
		\partial_t \bs{w}(s,t)=\sum_{j=1}^n \partial_t H_j(s,t) A_j(H(s,t)) \bs{w}(s,t).
	\end{equation} 
	By the usual smooth dependence of a linear ODE on parameters, we may differentiate the solution $w(s,t)$ with respect to $s$. By~\eqref{eqn::linear_system_sol_existence_aux2} we have
	\begin{align*}
		\begin{split}
			&\partial_t \underbrace{\left( \partial_s w(s,t) - \sum_{j=1}^{n} \partial_s H_j(s,t) A_j(H(s,t)) w(s,t) \right)}_{:=\Lambda(s,t)} \\
			= & \left( \sum_{i=1}^{n} \partial_t H_i(s,t) A_i(H(s,t)) \right) \left( \partial_s w(s,t) - \sum_{j=1}^{n} \partial_s H_j(s,t) A_j(H(s,t)) \right) \\
			&+ \sum_{i,j=1}^{n} \partial_t H_i(s,t) \partial_s H_j(s,t) \underbrace{\left( \partial_j A_i(H(s,t)) - \partial_i A_j(H(s,t)) + [A_i(H(s,t)),A_j(H(s,t))] \right)}_{=0, \text{ due to}~\eqref{eqn::linear_system_compatibility}}w(s,t).
		\end{split}
	\end{align*}
	That is
	\begin{equation} \label{eqn::linear_system_sol_existence_aux3}
		\partial_t \Lambda(s,t) = \left( \sum_{i=1}^{n} \partial_t H_i(s,t) A_i(H(s,t)) \right) \Lambda(s,t).	
	\end{equation}
	At $t=0$, since the initial point $H(s,0)=\bs{\theta}^0$ and initial value $\bs{w}(s,0)=\bs{v}$ are fixed, we have $\Lambda(s,0)=0$. Uniqueness of the solution for the ODE~\eqref{eqn::linear_system_sol_existence_aux3} gives $\Lambda(s,t)=0$ for all $0\le s,t\le 1$. Since the endpoint $H(s,1)=\bs{\theta}$ is fixed and $\Lambda(s,1)=0$, we conclude that $\partial_s \bs{w}(s,1)=0$ which implies $\bs{w}(s,1)$ is independent of $s$. Hence the solution $w(1)$ to~\eqref{eqn::linear_system_sol_existence_aux1} only depends on the endpoint of $\gamma$. 
	
	Therefore, we may define $\bs{u}(\bs{\theta})$ as $w(1)$, the
	endpoint value of~\eqref{eqn::linear_system_sol_existence_aux1}. Considering a short segment in the $\theta_j$ direction,~\eqref{eqn::linear_system_sol_existence_aux1} shows that
	$\partial_j\bs{u}(\bs{\theta})=A_j(\bs{\theta}) \bs{u}(\bs{\theta})$. The uniqueness follows by restricting any other solution to a path from $\bs{\theta}^{0}$. This completes the proof.
\end{proof}

\begin{lemma} \label{lem::linear_system_injective}
	Let $\bs{0}=(0,\ldots,0)$. If $\bs{u}$ is a solution of~\eqref{eqn::linear_system_matrix} and $u_{\bs{0}}(\bs{\theta})=0$ for all $\bs{\theta}\in \LX_{n}^{\U}$, then $u_{\bs{\beta}}(\bs{\theta})=0$ for all $\bs{\theta}\in \LX_n^{\U}$. Moreover, the map $\bs{v}\mapsto u_{\bs{0}}$ in Lemma~\ref{lem::linear_system_sol_existence} is injective.
\end{lemma}
\begin{proof}
	This follows from the first equation in~\eqref{eqn::linear_system_recursion}.
\end{proof}

\begin{lemma}\label{lem::dim_lowerbound}
Assume the same setup as in Theorem~\ref{thm::solutionspace_dimension}. We have 
\[\dim\LSrad{n}^{(\lambda)}\ge 2^n.\]
\end{lemma}

\begin{proof}
	Pick a basis $\bs{v}^1,\ldots,\bs{v}^{2^n}$ of $\R^{\{0,1\}^n}= \R^{2^n}$. Lemma~\ref{lem::linear_system_sol_existence} gives the corresponding solutions $\bs{u}^{r}$ such that $\bs{u}^r(\bs{\theta}^0)=\bs{v}^r$. Lemma~\ref{lem::radialBPZ_equi_linear_system} shows that $u^{r}_{\bs{0}}(\bs{\theta})$ are solutions of~\eqref{eqn::BPZ_radial}. Lemma~\ref{lem::linear_system_injective} shows that the functions $u^{r}_{\bs{0}}(\bs{\theta})$ are linearly independent. This proves $\dim\LSrad{n}^{(\lambda)}\ge 2^n$ as desired.
\end{proof}

\begin{corollary}\label{cor::iso_LS_R2n}
Fix $\bs{\theta}^0\in\LX_n^{\U}$. Then the linear map $\varsigma(\LZ)=(\partial^{\bs{\beta}} \LZ(\bs{\theta}^0))_{\bs{\beta}\in \{0,1\}^n}$ in~\eqref{eqn::iso_LS_R2n} is an isomorphism between $\LSrad{n}^{(\lambda)}$ and $\R^{2^n}$.
\end{corollary}
\begin{proof}
The proof of Lemma~\ref{lem::dim_upperbound} shows that $\varsigma$ is injective. To prove surjectivity, fix $\bs{v}\in\R^{\{0,1\}^n}=\R^{2^n}$. Lemma~\ref{lem::linear_system_sol_existence} constructs a smooth solution $\bs{u}$ of~\eqref{eqn::linear_system_matrix} such that $\bs{u}(\bs{\theta}^0)=\bs{v}$. By Lemma~\ref{lem::radialBPZ_equi_linear_system}, $u_{\bs{0}}\in\LSrad{n}^{(\lambda)}$ and $u_{\bs{\beta}}=\partial^{\bs{\beta}}u_{\bs{0}}$ for all $\bs{\beta}\in\{0,1\}^n$, which shows that $\varsigma(u_{\bs{0}})=\bs{u}(\bs{\theta}^0)=\bs{v}$.
Thus, $\varsigma$ is also surjective, which gives the desired result.
\end{proof}

\subsection{Decomposition of the solution space}
\label{subsec::decomposition}

The proof of Theorem~\ref{thm::solutionspace_decompositon_dimension} consists of the following Lemma~\ref{lem::solutionspace_decompositon_dimension_kappa04lambdage0}--Lemma~\ref{lem::solutionspace_decompositon_dimension_degenerate_evenn}. Our main idea is to translate the conformal Ward identity into an eigenvalue problem for a finite-dimensional matrix. Our strategy is listed as below.
\begin{itemize}
\item In Lemma~\ref{lem::solutionspace_decompositon_dimension_kappa04lambdage0}, we show that Theorem~\ref{thm::solutionspace_decompositon_dimension} holds for $\kappa\in (0,4]$ and $\lambda>0$. Its proof relies on the construction in Propositions~\ref{prop::radialsolutions_linearindept_oddn} and~\ref{prop::radialsolutions_linearindept_evenn}.
\item In Lemma~\ref{lem::ward_matrix_eigenspace}, we construct an isomorphism $\varsigma$ between $\LSrad{n}^{(\lambda;\nu)}$ and eigenspace of the Ward matrix $B$ with respect to the linear system we constructed in Lemma~\ref{lem::radialBPZ_equi_linear_system}. In Lemma~\ref{lem::ward_polynomial_identities}, we study the characteristic polynomial and an annihilating polynomial of the Ward matrix $B$. When $\kappa\in (0,4]$ and $\lambda>0$, Lemma~\ref{lem::solutionspace_decompositon_dimension_kappa04lambdage0} fully decomposes $\LSrad{n}^{(\lambda)}$ and yields all eigenvalues with multiplicities, thus determining the characteristic and annihilating polynomial. By continuous extension, the same polynomials are obtained for general parameters.
\item Lemma~\ref{lem::solutionspace_decompositon_dimension_nondegenerate} addresses the case when the annihilating polynomial of $B$ has no repeated roots (generic case). In this case $B$ has no Jordan blocks, and the result directly follows from the characteristic and minimal polynomials.
\item Lemma~\ref{lem::solutionspace_decompositon_dimension_degenerate_oddn} addresses the degenerate case when $n$ is odd. In this case we obtain the result by showing each $0$-Jordan blocks of $B$ has size two.
Lemma~\ref{lem::solutionspace_decompositon_dimension_degenerate_evenn} addresses the degenerate case when $n$ is even. This is the most complicated case. Our idea is to vary the matrix parameter continuously, comparing with the degenerate case (intuitively letting $\kappa\to \infty$ in~\eqref{eqn::linear_system_recursion}) to obtain an upper bound, and comparing with the non-critical case in Lemma~\ref{lem::solutionspace_decompositon_dimension_nondegenerate} to obtain a lower bound.
\end{itemize}
Recall that $\mathsf{P}_n=\{p\in\{1, \ldots, n\}: n-p\text{ is even}\}$ is defined in~\eqref{eqn::parity_n_p} and $\nu_p=p\sqrt{2\lambda+p^2-1}$ is defined in~\eqref{eqn::nup_def} and $\mu_0=\lambda-\frac{1}{8}(6-\kappa)(\kappa-2)$ and $\mu_p=\sqrt{2\lambda+p^2-1}$ are defined in~\eqref{eqn::mu_p_def}. Note that $\nu_p=p\mu_p$.

\begin{lemma} \label{lem::solutionspace_decompositon_dimension_kappa04lambdage0}
Theorem~\ref{thm::solutionspace_decompositon_dimension} holds for $\kappa\in (0,4]$ and $\lambda>0$. 
\end{lemma}

\begin{proof}
Propositions~\ref{prop::radialsolutions_linearindept_oddn} and~\ref{prop::radialsolutions_linearindept_evenn}, and Theorem~\ref{thm::solutionspace_dimension} give us the following observations.
\begin{itemize}
	\item For each $p\in\mathsf{P}_n$, the collection $\{\LZmix(\bs{\theta}): \mu=\pm\mu_p, (\bs{s}, \alpha)\text{ is allowable}\}$ contains $\binom{n}{(n-p)/2}$  solutions in $\LSrad{n}^{(\lambda;\pm \nu_p)}$. 
	\item If $n$ is even, the collection $\{\LZalphaCR(\bs{\theta}): \mu=\mu_0, \alpha\in\LP_{n/2}, 1\le \ell\le n/2+1\}$ contains $\binom{n}{n/2}$ solutions in $\LSrad{n}^{(\lambda;0)}$.
	\item  There are $2^n$ functions listed in the above two items, and they form a basis of $\LSrad{n}^{(\lambda)}$.
\end{itemize}
Therefore, 
\begin{align}  \label{eqn::decompose_solutionspace004}
	\LSrad{n}^{(\lambda)}=\begin{cases}
		\bigoplus_{p\in\mathsf{P}_n}
		\left(\LSrad{n}^{\left(\lambda; \nu_p \right)}
		\oplus\LSrad{n}^{\left(\lambda; -\nu_p \right)}\right),
		& n\text{ odd};\\
		\LSrad{n}^{(\lambda;0)}
		\oplus\bigoplus_{p\in\mathsf{P}_n}
		\left(\LSrad{n}^{\left(\lambda;\nu_p \right)}
		\oplus\LSrad{n}^{\left(\lambda;-\nu_p \right)}\right),
		& n\text{ even};
	\end{cases}
\end{align}
which gives the desired result.
\end{proof}

\begin{lemma}\label{lem::ward_matrix_eigenspace}
Fix $\bs{\theta}^0\in\LX_n^{\U}$. Assume the same setup as in Lemma~\ref{lem::radialBPZ_equi_linear_system}. Let
\begin{equation} \label{eqn::ward_matrix}
	B=B(\bs{\theta}^0;\kappa,\lambda):=\kappa\sum_{j=1}^nA_j(\bs{\theta}^0).
\end{equation}
Then the linear map $\varsigma(\LZ)=(\partial^{\bs{\beta}} \LZ(\bs{\theta}^0))_{\bs{\beta}\in \{0,1\}^n}$ in~\eqref{eqn::iso_LS_R2n} is an isomorphism between $\LSrad{n}^{(\lambda;\nu)}$ and $\ker(B-\nu I)$, and
\begin{equation}\label{eqn::ward_eigenspace}
	\dim\LSrad{n}^{(\lambda;\nu)}	=\dim\ker(B-\nu I).
\end{equation}
\end{lemma}

\begin{proof}
We show that the linear map $\varsigma(\LZ)=(\partial^{\bs{\beta}} \LZ(\bs{\theta}^0))_{\bs{\beta}\in \{0,1\}^n}$ in~\eqref{eqn::iso_LS_R2n} is an isomorphism between $\LSrad{n}^{(\lambda;\nu)}$ and $\ker(B-\nu I)$.
For $\LZ\in\LSrad{n}^{(\lambda)}$, by~\eqref{eqn::radialBPZ_Ward_commute}, we have $\LR\LZ$ satisfies the radial BPZ system~\eqref{eqn::BPZ_radial}. Lemma~\ref{lem::radialBPZ_equi_linear_system} shows that
\begin{equation} \label{eqn::ward_matrix_jet}
	\varsigma(\kappa\LR\LZ)=\kappa((\LR\partial^{\bs{\beta}}\LZ)(\bs{\theta}^0))_{\bs{\beta}\in\{0,1\}^n}=\kappa\sum_{j=1}^n A_j(\bs{\theta}^0;\kappa,\lambda) (\partial^{\bs{\beta}} \LZ(\bs{\theta}^0))_{\bs{\beta}\in\{0,1\}^n}=B\varsigma(\LZ).
\end{equation}
Now, we restrict $\varsigma$ to $\LSrad{n}^{(\lambda;\nu)}$.

First, let us show that $\varsigma$ is injective from $\LSrad{n}^{(\lambda;\nu)}$ to $\ker(B-\nu I)$.  If $\LZ\in\LSrad{n}^{(\lambda;\nu)}$, then $\kappa \LR\LZ= \nu\LZ$, and~\eqref{eqn::ward_matrix_jet} gives
\[
(B-\nu I)\varsigma(\LZ)
=\varsigma(\kappa\LR\LZ)-\nu\varsigma(\LZ)=0.
\]
Thus, $\varsigma$ maps $\LSrad{n}^{(\lambda;\nu)}$ into $\ker(B-\nu I)$. It is injective on $\LSrad{n}^{(\lambda;\nu)}$ by Corollary~\ref{cor::iso_LS_R2n}.

Next, let us show that $\varsigma$ is surjective from $\LSrad{n}^{(\lambda;\nu)}$ to $\ker(B-\nu I)$. Fix $\bs{v}\in\ker(B-\nu I)$. By Corollary~\ref{cor::iso_LS_R2n}, there exists $\LZ\in\LSrad{n}^{(\lambda)}$ such that $\varsigma(\LZ)=\bs{v}$. It remains to show $\kappa \LR \LZ = \nu \LZ$ on $\LX_n^{\U}$. For each $j\in\{1,\ldots,n\}$, the commutation relation~\eqref{eqn::radialBPZ_Ward_commute} gives
\[
\left(\LD_j-\lambda/\kappa\right) \left(\LR-\nu/\kappa\right)\LZ=\left(\LR-\nu/\kappa\right)\left(\LD_j-\lambda/\kappa\right)\LZ=0.
\]
Therefore, $(\kappa \LR-\nu)\LZ \in\LSrad{n}^{(\lambda)}$. Moreover, Eq.~\eqref{eqn::ward_matrix_jet} shows
\begin{equation*}
	\varsigma((\kappa \LR-\nu)\LZ)=(B-\nu I)\bs{v}=0,
\end{equation*}
which implies all square free derivatives of $(\kappa \LR-\nu)\LZ$ vanish at $\bs{\theta}^0$. Applying the same reduction argument in the proof of Lemma~\ref{lem::dim_upperbound} to $(\kappa \LR-\nu)\LZ$, all of its partial derivatives of every order vanish at $\bs{\theta}^0$. Since $(\kappa \LR-\nu)\LZ$ is analytic on $\LX_n^{\U}$, we have $(\kappa \LR-\nu)\LZ$ vanishes in a neighborhood of $\bs{\theta}^0$, and hence, vanishes in the connected chamber $\LX_n^{\U}$. Therefor we have $\kappa \LR\LZ=\nu\LZ$, which shows that $\LZ\in\LSrad{n}^{(\lambda;\nu)}$ and that $\varsigma$ is surjective.

We conclude that $\varsigma$ restricted in $\LSrad{n}^{(\lambda;\nu)}$ is an isomorphism onto $\ker(B-\nu I)$, which shows~\eqref{eqn::ward_eigenspace} as desired.
\end{proof}

\begin{lemma}\label{lem::ward_polynomial_identities}
Let $\kappa>0$ and $\lambda\in \R$. 
Define 
\[r_{n;q}=\binom{n}{(n-q)/2},\quad\text{for }q\in\mathsf{P}_n;\qquad r_{n;0}=\binom{n}{n/2},\quad\text{for even }n.\]
The characteristic polynomial of the matrix $B$ in~\eqref{eqn::ward_matrix} is
\begin{equation}\label{eqn::ward_characteristic_polynomial}
	\det(tI-B)=\begin{cases}
		\prod_{q\in\mathsf{P}_n} \bigl(t^2-q^2(2\lambda+q^2-1)\bigr)^{r_{n;q}}, &n\text{ odd},\\
		t^{r_{n;0}}\prod_{q\in\mathsf{P}_n} \bigl(t^2-q^2(2\lambda+q^2-1)\bigr)^{r_{n;q}}, &n\text{ even}.
	\end{cases}
\end{equation}
Moreover, the following $Q_\lambda(t)$ is an annihilating polynomial of $B$:
\begin{equation}\label{eqn::ward_annihilating_polynomial}
	Q_\lambda(t):=\begin{cases}
		\prod_{q\in\mathsf{P}_n}\bigl(t^2-q^2(2\lambda+q^2-1)\bigr), &n\text{ odd},\\
		t\prod_{q\in\mathsf{P}_n}\bigl(t^2-q^2(2\lambda+q^2-1)\bigr), &n\text{ even}.
	\end{cases}	
\end{equation}
\end{lemma}

\begin{proof}
First, we show~\eqref{eqn::ward_characteristic_polynomial} and~\eqref{eqn::ward_annihilating_polynomial} for $\kappa\in(0,4]$ and $\lambda>0$. From~\eqref{eqn::ward_eigenspace} and~\eqref{eqn::decompose_solutionspace004}, $\{\pm\nu_q: q\in\mathsf{P}_n\}$ are eigenvalues of $B$, and 
\[\dim\ker(B-\nu_q I) = \dim\ker(B+\nu_q I) = \binom{n}{(n-q)/2}.\] When $n$ is even, $0$ is also an eigenvalue of $B$, and $\dim\ker B=\binom{n}{n/2}$. This accounts for all $2^n$ dimensions. Therefore characteristic polynomial of the matrix $B$ is given by~\eqref{eqn::ward_characteristic_polynomial}, and the minimal polynomial of the matrix $B$ is given by~\eqref{eqn::ward_annihilating_polynomial}. 
\medbreak
Next, we show~\eqref{eqn::ward_characteristic_polynomial} and~\eqref{eqn::ward_annihilating_polynomial} for $\kappa>0$ and $\lambda\in \R$. The finite recursion~\eqref{eqn::linear_system_recursion} only uses differentiation in the angles, addition, multiplication, and division by powers of $\kappa$. Since $\mathfrak{b}=\frac{6-\kappa}{2\kappa}$, every entry of $A_j$, and hence of $B$, is polynomial in $\lambda$ and a Laurent polynomial in $\kappa$. Therefore, for some $n_0\in \Z_{\ge 0}$ sufficiently large,
\begin{equation*}
	\kappa^{n_0}\det(tI-B)- \kappa^{n_0} \begin{cases}
		\prod_{q\in\mathsf{P}_n} \bigl(t^2-q^2(2\lambda+q^2-1)\bigr)^{r_{n;q}}, &n\text{ odd},\\
		t^{r_{n;0}}\prod_{q\in\mathsf{P}_n} \bigl(t^2-q^2(2\lambda+q^2-1)\bigr)^{r_{n;q}}, &n\text{ even},
	\end{cases}
\end{equation*}
as well as every entry of $\kappa^{n_0} Q_\lambda(B)$, is a polynomial in $(\kappa,\lambda,t)$. These polynomials vanish on the open set $(0,4)\times(0,\infty)\times \R$, so they vanish identically. This proves~\eqref{eqn::ward_characteristic_polynomial} and~\eqref{eqn::ward_annihilating_polynomial} for every $\kappa>0$ and $\lambda\in\R$.
\end{proof}

\begin{lemma} \label{lem::solutionspace_decompositon_dimension_nondegenerate}
Assume the same setup as in Theorem~\ref{thm::solutionspace_decompositon_dimension}. 
\begin{itemize}
\item If $n$ is odd, we have
\begin{align} \label{eqn::solutionspace_decompositon_dimension_odd_repeat2nd}
	\dim\LSrad{n}^{(\lambda;\nu)}=\begin{cases}
		\binom{n}{(n-p)/2},	&\nu\in\{\pm \nu_p\}, \quad p^2>1-2\lambda,\quad p\in\mathsf{P}_n, \\
		0,&\lambda\in \R, \quad \nu\notin \{ \pm \nu_p: p\in \mathsf{P}_n, p^2\ge 1-2\lambda \}.
	\end{cases}
\end{align}
\item If $n$ is even, we have
\begin{align} \label{eqn::solutionspace_decompositon_dimension_even_repeat2nd}
	\dim\LSrad{n}^{(\lambda;\nu)}=
	\begin{cases}
		\binom{n}{(n-p)/2},	& \nu\in\{\pm \nu_p\},\quad	p^2>1-2\lambda,\quad p\in\mathsf{P}_n, \\
		\binom{n}{n/2},&\nu=0,\quad \lambda\notin\{ (1-p^2)/2: p\in\mathsf{P}_n \} \\
		0,&\lambda\in \R, \quad \nu\notin \{ 0,\pm \nu_p: p\in \mathsf{P}_n, p^2\ge 1-2\lambda \}.
	\end{cases}
\end{align}
\end{itemize}
\end{lemma}
\begin{proof}
We first check the second case in~\eqref{eqn::solutionspace_decompositon_dimension_odd_repeat2nd} and the third case in~\eqref{eqn::solutionspace_decompositon_dimension_even_repeat2nd}. Assume $\lambda\in \R$. If $\nu\in \R$ is not the root of the characteristic polynomial of $B$~\eqref{eqn::ward_characteristic_polynomial}: 
\begin{align*}
	&\nu\notin \begin{cases}
		\{ \pm \nu_p: p\in \mathsf{P}_n, p^2 \ge 1-2\lambda \}, & n\text{ odd}, \\
		\{ 0, \pm \nu_p: p\in \mathsf{P}_n, p^2 \ge 1-2\lambda \}, & n\text{ even},
	\end{cases} 
\end{align*}
then 
\begin{equation} \label{eqn::solutionspace_decompositon_dimension_nondegenerate_aux1}
	\dim \LSrad{n}^{(\lambda;\nu)}= \dim \ker (B-\nu I)=0,
\end{equation}
by~\eqref{eqn::ward_eigenspace}. This gives the second case in~\eqref{eqn::solutionspace_decompositon_dimension_odd_repeat2nd} and the third case in~\eqref{eqn::solutionspace_decompositon_dimension_even_repeat2nd}. 
\medbreak
Second, we deal with the first case in~\eqref{eqn::solutionspace_decompositon_dimension_odd_repeat2nd} and the first case in~\eqref{eqn::solutionspace_decompositon_dimension_even_repeat2nd}. Assume $\nu\in \{ \pm \nu_p: p\in \mathsf{P}_n, p^2>1-2\lambda \}$, which is the set of all the non-zero real roots of the characteristic polynomial of $B$~\eqref{eqn::ward_characteristic_polynomial}.
Since $\nu_p>q\sqrt{2\lambda+q^2-1}$ for $p>q>\sqrt{1-2\lambda}$, each non-zero real root $\pm \nu_p$ is simple in $Q_\lambda$. By~\eqref{eqn::ward_annihilating_polynomial}, every Jordan block of $B$ for the eigenvalue $\pm \nu_p$ has size one. Therefore, we have
\begin{equation} \label{eqn::solutionspace_decompositon_dimension_nondegenerate_aux2}
	\dim \LSrad{n}^{(\lambda;\pm \nu_p)}= \dim \ker (B\mp \nu_p I)=\binom{n}{(n-p)/2},
\end{equation}
by~\eqref{eqn::ward_eigenspace} and~\eqref{eqn::ward_characteristic_polynomial}. This gives the first case in~\eqref{eqn::solutionspace_decompositon_dimension_odd_repeat2nd} and the first case in~\eqref{eqn::solutionspace_decompositon_dimension_even_repeat2nd}. 
\medbreak
Finally, we deal with the second case in~\eqref{eqn::solutionspace_decompositon_dimension_even_repeat2nd}. Assume $\nu=0$. When $n$ is even, $0$ has algebraic multiplicity $\binom{n}{n/2}$ and is a simple root of $Q_\lambda$. Therefore, every Jordan block of $B$ for the eigenvalue $0$ has size one, and
\begin{equation} \label{eqn::solutionspace_decompositon_dimension_nondegenerate_aux3}
	\dim \LSrad{n}^{(\lambda;0)}= \dim \ker (B)=\binom{n}{n/2},  \quad n\text{ even}.
\end{equation}
Combining~\eqref{eqn::solutionspace_decompositon_dimension_nondegenerate_aux1},~\eqref{eqn::solutionspace_decompositon_dimension_nondegenerate_aux2} and~\eqref{eqn::solutionspace_decompositon_dimension_nondegenerate_aux3}, we obtain~\eqref{eqn::solutionspace_decompositon_dimension_odd_repeat2nd} and~\eqref{eqn::solutionspace_decompositon_dimension_even_repeat2nd} as desired. 
\end{proof}

\begin{lemma}\label{lem::solutionspace_decompositon_dimension_degenerate_oddn}
Fix odd $n$. Theorem~\ref{thm::solutionspace_decompositon_dimension} holds in the following degenerate case: suppose $\lambda=(1-p^2)/2$ for some $p\in\mathsf{P}_n$ and $\nu=0$, 
\begin{align} \label{eqn::solutionspace_decompositon_dimension_odd_repeat3rd}
	\dim\LSrad{n}^{((1-p^2)/2;0)}= \binom{n}{(n-p)/2}.
\end{align}
\end{lemma}

\begin{proof}
We have the following observations.
\begin{itemize}
	\item Note that $0$ is a double root of the annihilating polynomial of $B$~\eqref{eqn::ward_annihilating_polynomial}. Thus all Jordan blocks of $B$ associated with the eigenvalue $0$ have size at most $2$.
	\item Let us show that all Jordan blocks of $B$ associated with the eigenvalue $0$ have size at least $2$. Assume $v\in \ker B\setminus\{0\}$.
	Write $B'=\partial_\lambda B(\bs{\theta}^0;\kappa,(1-p^2)/2)$, where $\bs{\theta}^0$ is fixed as in the setup of Lemma~\ref{lem::ward_matrix_eigenspace} and $B(\bs{\theta}^0;\kappa,\lambda)$ is defined in~\eqref{eqn::ward_matrix}. Since~\eqref{eqn::ward_annihilating_polynomial} holds for all $\lambda\in \R$, we have
	\begin{align*} 
		\begin{split}
			\partial_{\lambda} Q_{\lambda}(B(\bs{\theta}^0;\kappa,\lambda)) = \partial_{\lambda} \Bigg( &\Big( \prod_{q\in\mathsf{P}_n \setminus \{p\}} \bigl(B(\bs{\theta}^0;\kappa,\lambda)^2-q^2(2\lambda+q^2-1)\bigr) \Big) \\
			& \times \bigl(B(\bs{\theta}^0;\kappa,\lambda)^2-p^2(2\lambda+p^2-1)\bigr) \Bigg) = 0.	
		\end{split}
	\end{align*}
	When $\lambda=(1-p^2)/2$, applying this identity to $v\in\ker B\setminus\{0\}$, we have
	\begin{equation} \label{eqn::solutionspace_decompositon_dimension_degenerate_oddn_aux1}
		2p^2\prod_{q\in\mathsf{P}_n, q\ne p}\bigl(-q^2(q^2-p^2)\bigr)v=B\big(\prod_{q\in\mathsf{P}_n,q\ne p}\bigl(B^2-q^2(q^2-p^2)I\bigr)\big)B'v.		
	\end{equation}
	Let 
	\[ \hat{v}:= \bigl( 2p^2\prod_{q\in\mathsf{P}_n, q\ne p}\bigl(-q^2(q^2-p^2)\bigr) \bigr)^{-1} \times \big(\prod_{q\in\mathsf{P}_n,q\ne p}\bigl(B^2-q^2(q^2-p^2)I\bigr)\big)B'v.\]
	Then~\eqref{eqn::solutionspace_decompositon_dimension_degenerate_oddn_aux1} shows that $B^2 \hat{v}=0$ and $B \hat{v}=v\neq 0$, which implies that all Jordan blocks of $B$ associated with the eigenvalue $0$ have size at least $2$.
	\item By~\eqref{eqn::ward_characteristic_polynomial}, the dimension of the generalized eigenspace of $B$ associated with the eigenvalue $0$ (i.e. $\ker B^{2}$) is $2\binom{n}{(n-p)/2}$.
\end{itemize}  
Combining the observations above, we conclude that there are $\binom{n}{(n-p)/2}$ Jordan blocks of $B$ associated with the eigenvalue $0$, and every block has size $2$. By~\eqref{eqn::ward_eigenspace}, we conclude that
\begin{equation*}
	\dim \LSrad{n}^{((1-p^2)/2;0)} = \dim \ker B(\bs{\theta}^0;\kappa,(1-p^2)/2) = \binom{n}{(n-p)/2}.
\end{equation*}
This proves~\eqref{eqn::solutionspace_decompositon_dimension_odd_repeat3rd} as desired.
\end{proof}

\begin{lemma}\label{lem::solutionspace_decompositon_dimension_degenerate_evenn}
Fix even $n$. Theorem~\ref{thm::solutionspace_decompositon_dimension} holds in the following degenerate case: suppose $\lambda=(1-p^2)/2$ for some $p\in\mathsf{P}_n$ and $\nu=0$, 
	\begin{align} \label{eqn::solutionspace_decompositon_dimension_even_repeat3rd}
		\dim\LSrad{n}^{((1-p^2)/2;0)}= \binom{n}{n/2}.
	\end{align}
\end{lemma}

\begin{proof}
First, let us show the upper bound
\begin{equation}\label{eqn::ward_zero_dimension_upperbound}
	\dim\ker B\le\binom{n}{n/2}.
\end{equation}
We index the rows and columns of $B$ by $\bs{\beta},\bs{\delta}\in\{0,1\}^n$, as in Lemma~\ref{lem::radialBPZ_equi_linear_system}. The recursion~\eqref{eqn::linear_system_recursion} determines all entries with $|\bs{\delta}|>|\bs{\beta}|$:
\begin{equation} \label{eqn::ward_zero_dimension_upperbound_aux1}
	B_{\bs{\beta},\bs{\delta}}=\begin{cases}
		\kappa,	& \bs{\delta}=\bs{\beta}+\bs e_j\text{ for some }j\text{ with }\beta_j=0,\\
		0, & \text{other }|\bs{\delta}|>|\bs{\beta}|.
	\end{cases}
\end{equation}
When $\beta_j=0$, the first line of~\eqref{eqn::linear_system_recursion} places $1$ in the $(\bs{\beta},\bs{\beta}+\bs e_j)$ entry of $A_j$. When $\beta_j=1$, the recursive reduction in the second line of~\eqref{eqn::linear_system_recursion} only gives entries indexed by $\bs{\delta}$ with $|\bs{\delta}|\le|\bs{\beta}|$. Combining with~\eqref{eqn::ward_matrix}, we obtain~\eqref{eqn::ward_zero_dimension_upperbound_aux1}.

Define the auxiliary matrices $U$ and $(D_t, t>0)$:
\[
U_{\bs{\beta},\bs{\delta}}:=\begin{cases}
	1, & \bs{\delta}=\bs{\beta}+\bs e_j \text{ for some }j\text{ with }\beta_j=0,\\
	0,&\text{otherwise},
\end{cases}
\qquad
(D_t)_{\bs{\beta},\bs{\delta}}:=\begin{cases}
	t^{|\bs{\beta}|},&\bs{\beta}=\bs{\delta},\\
	0,&\bs{\beta}\ne\bs{\delta}.
\end{cases}
\]
By~\eqref{eqn::ward_zero_dimension_upperbound_aux1}, we have
\[
\frac{1}{\kappa t}D_t^{-1}BD_t\longrightarrow U,
\qquad t\to\infty.
\]
Therefore, we have
\begin{equation} \label{eqn::ward_zero_dimension_upperbound_aux2}
	\operatorname{rank}(U) = \operatorname{rank} \left(\lim_{t\to \infty} \frac{1}{\kappa t}D_t^{-1}BD_t \right)\le \operatorname{rank}\left(\frac{1}{\kappa t}D_t^{-1}BD_t\right)=\operatorname{rank}(B).
\end{equation}
It remains to calculate $\operatorname{rank}(U)$. We decompose $\R^{\{0,1\}^n}$ by:
\begin{equation} \label{eqn::ward_zero_dimension_upperbound_aux3}
	\R^{\{0,1\}^n}=\bigoplus_{m=0}^n H_m,\quad \text{where }H_m:=\{\bs v\in\R^{\{0,1\}^n}:v_{\bs{\beta}}=0\text{ if }|\bs{\beta}|\neq m\}.
\end{equation}
Then $\dim H_m=\binom{n}{m}$ and $U(H_m)\subseteq H_{m-1}$ for $1\le m\le n$. Let $U^T$ be the transpose of $U$. Its entries are
\begin{equation*}
	(U^T)_{\bs{\delta},\bs{\beta}}=U_{\bs{\beta},\bs{\delta}}=\begin{cases}
		1, & \bs{\delta}=\bs{\beta}+\bs e_j	\text{ for some }j\text{ with }\beta_j=0,\\
		0, & \text{otherwise}.
	\end{cases}
\end{equation*}
Then $U^T(H_m)\subseteq H_{m+1}$ for $0\le m\le n-1$, and
\begin{align*}
	\left.(UU^T-U^TU)\right|_{H_m}=(n-2m)I,\qquad \|U^Tv\|^2-\|Uv\|^2=(n-2m)\|v\|^2,\quad v\in H_m.
\end{align*}
This shows that $U_m:=\left.U\right|_{H_{m+1}}:H_{m+1}\longrightarrow H_m$ is injective for $m\ge n/2$, and $U_m^{T}:H_{m}\to H_{m+1}$ is injective for $m< n/2$. Therefore, we have
\begin{equation*}
	\operatorname{rank}(U_m)= \operatorname{rank}(U_m^{T})=\begin{cases}
		\dim H_m= \binom{n}{m}, & 0\le m<n/2,\\
		\dim H_{m+1}= \binom{n}{m+1}, & n/2\le m\le n-1.
	\end{cases}
\end{equation*}
According to the orthogonal decomposition~\eqref{eqn::ward_zero_dimension_upperbound_aux3}, write the matrix $U$ in the following block form:
\[
U=\begin{pmatrix}
	0&U_0&0&\cdots&0\\
	0&0&U_1&\ddots&\vdots\\
	\vdots&\ddots&\ddots&\ddots&0\\
	0&\cdots&0&0&U_{n-1}\\
	0&\cdots&\cdots&0&0
\end{pmatrix}.
\]
Therefore, we have
\[ \operatorname{rank} (U)=\sum_{m=0}^{n-1} \operatorname{rank} (U_m) = \sum_{m=0}^{n/2-1} \binom{n}{m} + \sum_{m=n/2}^{n-1} \binom{n}{m+1}=2^n-\binom{n}{n/2}. \]
Plugging into~\eqref{eqn::ward_zero_dimension_upperbound_aux2}, we obtain
\[ \dim \ker B \le \dim \ker U = 2^n- \operatorname{rank} (U)=\binom{n}{n/2}, \]
and finish the proof of~\eqref{eqn::ward_zero_dimension_upperbound}.
\medbreak
Next, let us deal with the lower bound. Note that $B(\bs{\theta}^0;\kappa,\lambda)$ is continuous with respect to $\lambda$, combining with~\eqref{eqn::solutionspace_decompositon_dimension_nondegenerate_aux3}, we have
\begin{align*}
	\dim \ker B(\bs{\theta}^0;\kappa,(1-p^2)/2) = & 2^n - \operatorname{rank} ( B(\bs{\theta}^0;\kappa,(1-p^2)/2)) \\
	\ge & 2^n - \lim_{\eps\to 0} \operatorname{rank} B(\bs{\theta}^0;\kappa,(1-p^2)/2+ \eps) = \binom{n}{n/2}.
\end{align*}
Combining with~\eqref{eqn::ward_zero_dimension_upperbound}, we conclude that $\dim \ker B= \binom{n}{n/2}$. By~\eqref{eqn::ward_eigenspace}, we obtain $\dim \LSrad{n}^{((1-p^2)/2;0)}= \binom{n}{n/2}$, which gives~\eqref{eqn::solutionspace_decompositon_dimension_even_repeat3rd} as desired.
\end{proof}

%% file: tex_radial/commutation_radial.tex
\paragraph*{Axioms of commutation relation.}
Let $D\subset\U$ be a simply connected domain containing the origin, and let $x_1,\ldots,x_n$ be distinct prime ends of $\partial D$, listed in counterclockwise order. Let $U_1,\ldots,U_n$ be closed neighborhoods of $x_1,\ldots,x_n$ in $\overline D$ such that $0\notin U_j$ and $U_i\cap U_j=\emptyset$ for $i\neq j$. We consider probability measures $\PP_{(D;x_1,\ldots,x_n;0)}^{(U_1,\ldots,U_n)}$ on $n$-tuples of unparametrized continuous curves $(\eta^1,\ldots,\eta^n)$, where $\eta^j$ starts from $x_j$, stays in $U_j$ until its first exit, and exits $U_j$ almost surely. Such a family of measures, indexed by the choices of $(U_1,\ldots,U_n)$, is called compatible if, whenever $U_j\subset U'_j$ for all $j$, the measure $\PP_{(D;x_1,\ldots,x_n;0)}^{(U_1,\ldots,U_n)}$ is obtained by restricting the curves sampled from $\PP_{(D;x_1,\ldots,x_n;0)}^{(U'_1,\ldots,U'_n)}$ to their parts before they first exit $U_1,\ldots,U_n$, respectively.

Fix $\kappa>0$ and $n\ge 2$. A locally commuting $n$-radial
	$\SLE_\kappa$ is a compatible family
$\PP_{(D;x_1,\ldots,x_n;0)}^{(U_1,\ldots,U_n)}$ on $n$-tuples of continuous
non-self-crossing curves $(\eta^1,\ldots,\eta^n)$, for all $D$, $(x_1,\ldots,x_n)$,
and $(U_1,\ldots,U_n)$ as above, satisfying the following conditions.
\begin{itemize}
\item[\textup{(CI)}] Conformal invariance.
If $\varphi:D\to D'$ is a conformal map such that $\varphi(0)=0$, then
\begin{equation*}
	\varphi^*\PP_{(D';\varphi(x_1),\ldots,\varphi(x_n);0)}^{(\varphi(U_1),\ldots,\varphi(U_n))}=\PP_{(D;x_1,\ldots,x_n;0)}^{(U_1,\ldots,U_n)}.
\end{equation*}
Consequently, it suffices to describe the measures in $(\U;\ee^{\ii\theta_1},\ldots,\ee^{\ii\theta_n};0)$ for $\bs{\theta}=(\theta_1,\ldots,\theta_n)\in\LX_n^{\U}$.

\item[\textup{(DMP)}] Domain Markov property.
Let $(\eta^1,\ldots,\eta^n)\sim\PP_{(D;x_1,\ldots,x_n;0)}^{(U_1,\ldots,U_n)}$, and parametrize the curves by their own radial capacities seen from the origin. Let $\bs{t}=(t_1,\ldots,t_n)$, where $t_j$ is a stopping time for $\eta^j$ such that $\eta^j_{[0,t_j]}$ is contained in the interior of $U_j$. Let
\begin{equation*}
	\widetilde U_j=U_j\setminus\eta^j_{[0,t_j]},\qquad
	\widetilde\eta^j=\eta^j\setminus\eta^j_{[0,t_j]},\qquad 1\le j\le n,
\end{equation*}
and let $\widetilde D$ be the connected component containing $0$ of $D\setminus\bigcup_{j=1}^n\eta^j_{[0,t_j]}$. Conditionally on the initial segments $\bigcup_{j=1}^n\eta^j_{[0,t_j]}$, we have
\begin{equation*}
	(\widetilde\eta^1,\ldots,\widetilde\eta^n) \sim \PP_{(\widetilde D;\eta^1_{t_1},\ldots,\eta^n_{t_n};0)}^{(\widetilde U_1,\ldots,\widetilde U_n)}.
\end{equation*}

\item[\textup{(MARG)}] Marginal laws.
Let $\bs{\theta}=(\theta_1, \ldots, \theta_n)\in\LX_n^{\U}$ and $(\eta^1,\ldots,\eta^n)\sim \PP_{(\U;\ee^{\ii\theta_1},\ldots,\ee^{\ii\theta_n};0)}^{(U_1,\ldots,U_n)}$. For each $1\le j\le n$, assume that there is a $C^2$ function $b_j:\LX_n^{\U}\to\R$ with the following property. Parametrize $\eta^j$ by radial capacity, denote its driving function by $\xi^j$, and denote by $\phi_t^j$ the covering map of $\eta^j$-s radial Loewner chain. Then, up to the first exit of $\eta^j$ from $U_j$,
\begin{equation}\label{eqn::commutation_radial_marginals}
	\begin{cases}
		\ud\xi_t^j=\sqrt{\kappa}\,\ud B_t^j +b_j(\phi_t^j(\theta_1),\ldots,\phi_t^j(\theta_{j-1}),\xi_t^j,\phi_t^j(\theta_{j+1}),\ldots,\phi_t^j(\theta_{n}))\,\ud t,
		&\xi_0^j=\theta_j,\\
		\ud\phi_t^j(\theta_k)=
		\cot\left(\frac{\phi_t^j(\theta_k)-\xi_t^j}{2}\right)\ud t,
		& k\neq j,
	\end{cases}
\end{equation}
where $B^j$ is a one-dimensional standard Brownian motion.
\end{itemize}

The above description of the commutation relation is the natural $n$-curve extension of the two-curve definition in~\cite[Section~2.1]{KrusellWangWuCommutationRelation} (see also~\cite{zhang2025multipleradialslekappaquantum}). We do not impose the additional interchangeability condition considered there. 
The infinitesimal commutation calculation of~\cite[Section~2.2]{KrusellWangWuCommutationRelation} applies pairwise in this setting.

\begin{proposition}\label{prop::commutation_radial_classify}
Fix $\kappa>0$ and $n\ge 2$. 
The existence of locally commuting $n$-radial $\SLE_{\kappa}$ is equivalent to the existence of $\LZ: \LX_n^{\U}\to \R_{>0}$, called partition function, such that 
\begin{align*}%\label{eqn::bj_pf}
	b_j=\kappa\partial_j\log\LZ, \qquad \text{for all }1\le j\le n, 
\end{align*}
where $\LZ$ satisfies the BPZ system~\eqref{eqn::BPZ_radial} and the conformal Ward identity~\eqref{eqn::ward_radial} for some constants $\lambda,\nu\in\R$.
\end{proposition}

\begin{proof}
The result and its proof has appeared in previous literature, see~\cite[Theorem~1.3 and Section~3.2]{zhang2025multipleradialslekappaquantum} (see also~\cite[Section~8.2]{DubedatCommutationSLE} and~\cite[Section~2.2 and Section~3]{KrusellWangWuCommutationRelation}). To be self-contained, we sketch the proof here.

First, suppose that the curves locally commute. The infinitesimal generators associated with~\eqref{eqn::commutation_radial_marginals} are
\begin{equation*}
	\mathcal{L}_j=\frac{\kappa}{2}\partial_j^2+b_j\partial_j+\sum_{\ell\neq j}\cot\left(\frac{\theta_\ell-\theta_j}{2}\right)\partial_\ell,\qquad j\in\{1,\ldots,n\}.
\end{equation*}
Comparing the two orders of growing short initial segments of $\eta^i$ and $\eta^j$, \textup{(DMP)} gives
\begin{equation} \label{eqn::commutation_radial_classify_aux1}
	[\mathcal{L}_i,\mathcal{L}_j]=\frac{\mathcal{L}_j-\mathcal{L}_i}{\sin^2\left(\frac{\theta_j-\theta_i}{2}\right)},\qquad i\neq j.
\end{equation}
Comparing the coefficients of $\partial_i\partial_j$, we obtain $\partial_i b_j=\partial_j b_i$. Then there exists a positive function $\LZ$ such that $b_j=\kappa\partial_j\log\LZ$ for $1\le j\le n$. Plugging $b_j=\kappa\partial_j\log\LZ$ into the first-order coefficients~\eqref{eqn::commutation_radial_classify_aux1}, we obtain
\begin{equation*}
	\partial_i\left(\frac{\LD_j\LZ}{\LZ}\right)=0,\qquad i\neq j,
\end{equation*}
where $\LD_j$ is defined in~\eqref{eqn::radialBPZ_operatordef}. Hence $\LD_j\LZ=F_j(\theta_j)\LZ$ for some functions $F_j$. Applying \textup{(CI)} to rotations of $\U$ shows that the drift functions are rotation-invariant:
\begin{equation*}%\label{eqn::bj_rotation_inv}
	b_j(\bs{\theta}+a)=b_j(\bs{\theta}),\qquad\text{for all }j\in\{1, \ldots, n\}\text{ and }a\in\R.
\end{equation*}
It follows that $\kappa\partial_j(\LR\log\LZ)=\LR b_j=0$ for all $j$, where $\LR=\sum_j\partial_j$. Therefore $\kappa\LR\log\LZ=\nu$ for some $\nu\in\R$, which gives~\eqref{eqn::ward_radial}. Moreover, $\LD_j\LZ/\LZ$ is rotation-invariant, so each $F_j$ is a constant. By elliptic regularity applied to the sum of these equations, $\LZ$ is smooth. Applying the commutator identity~\eqref{eqn::radialBPZoperator_commute} to $\LZ$ now gives
\begin{equation*}
	0=[\LD_i,\LD_j]\LZ=\frac{F_j-F_i}{\sin^2\left(\frac{\theta_j-\theta_i}{2}\right)}\LZ.
\end{equation*}
Since $\LZ>0$, all the constants $F_j$ agree. Writing their common value as $\lambda/\kappa$, we obtain~\eqref{eqn::BPZ_radial}.
\medbreak

Next, suppose that $\LZ>0$ satisfies~\eqref{eqn::BPZ_radial} and~\eqref{eqn::ward_radial}, and let $b_j=\kappa\partial_j\log\LZ$. We construct the locally commuting $\SLE$ by tilting independent radial SLEs by a multi-time martingale formed from $\LZ$ (see also~\cite[Proposition~2.4]{HuangPeltolaWuMultiradialSLE} for a direct construction). The BPZ equations cancel the drift terms in the It\^o's calculation, and Girsanov's theorem gives the marginal laws~\eqref{eqn::commutation_radial_marginals}. The martingale property gives \textup{(DMP)}. Finally, the Ward identity implies rotation invariance of the drifts, so rotation of $\U$ gives \textup{(CI)}. This yields the required locally commuting system; we refer to the above references for the remaining calculations.
\end{proof}

%% file: tex_radial/chordalBPZ.tex
\paragraph*{Chordal BPZ system.}
Fix $\kappa>0$ and $\lambda\in\R$. We consider functions $\LZ$ that are defined on the space
\[\LX_n^{\HH}=\{(x_1, \ldots, x_n)\in\R^n: x_1<\cdots<x_n\}\]
and satisfy the chordal BPZ system:
\begin{equation}\label{eqn::BPZ_chordal}
\frac{\kappa}{2}\partial_j^2\LZ+\sum_{\ell\neq j}\left(\frac{2}{x_{\ell}-x_j}\partial_{\ell}\LZ-\frac{(6-\kappa)/\kappa}{(x_{\ell}-x_j)^2}\LZ\right)=\frac{\lambda}{\kappa}\LZ ,\qquad \text{for }j\in\{1, \ldots, n\}. 
\end{equation}
In this appendix, we prove the corresponding conclusion of Theorem~\ref{thm::solutionspace_dimension} in the chordal case.  

\begin{theorem}\label{thm::chordalBPZ_dimension}
Fix $\kappa>0$ and $\lambda\in\R$ and $n\ge 2$. 
Define 
\begin{equation*} %\label{eqn::solutionspace_chordal_def}
\LSchord{n}^{(\lambda)}:=\{F\in C^{\infty}(\LX_n^{\HH}\to \R): F \text{ satisfies chordal BPZ system~\eqref{eqn::BPZ_chordal}}\}.
\end{equation*}
The dimension of the solution space $\LSchord{n}^{(\lambda)}$ is $2^n$:
\begin{equation*}
\dim\LSchord{n}^{(\lambda)}=2^n.
\end{equation*}
\end{theorem}

\begin{proof}
The proof is the similar as that of Theorem~\ref{thm::solutionspace_dimension} in Section~\ref{sec::PDE_analysis}. We sketch the proof below.
First, summing the equations in~\eqref{eqn::BPZ_chordal} gives an elliptic equation with analytic coefficients on $\LX_n^{\HH}$, thus every solution is analytic. Fix $\bs{x}^0\in\LX_n^{\HH}$. As in the proof of Lemma~\ref{lem::dim_upperbound}, repeated use of~\eqref{eqn::BPZ_chordal} determines all derivatives at $\bs{x}^0$ from the square free derivatives. By analyticity, the linear map
\begin{equation*}
	\varsigma:\LSchord{n}^{(\lambda)}\to\R^{2^n},\qquad
	\LZ\mapsto\left(\partial^{\bs{\beta}}\LZ(\bs{x}^0)\right)_{\bs{\beta}\in\{0,1\}^n}
\end{equation*}
is injective. This proves $\dim\LSchord{n}^{(\lambda)}\le 2^n$.
\medbreak

Next, we prove the lower bound. The same recursion as in Lemma~\ref{lem::radialBPZ_equi_linear_system}, using~\eqref{eqn::BPZ_chordal} to eliminate second order derivatives, gives an equivalent linear system
\begin{equation*}
	\partial_j\bs{u}(\bs{x})=A_j(\bs{x})\bs{u}(\bs{x}),\qquad j\in\{1,\ldots,n\},
\end{equation*}
where $\bs{u}=(\partial^{\bs{\beta}}\LZ)_{\bs{\beta}\in\{0,1\}^n}$ and the $2^n\times 2^n$ matrices $A_j$ have analytic coefficients on $\LX_n^{\HH}$. Denote the differential operator on the left-hand side of~\eqref{eqn::BPZ_chordal} by $\LD_j^{\HH}$:
\begin{equation*}
	\LD_j^{\HH}:=\frac{\kappa}{2}\partial_j^2
	+\sum_{\ell\neq j}\left(
	\frac{2}{x_{\ell}-x_j}\partial_{\ell}
	-\frac{(6-\kappa)/\kappa}{(x_{\ell}-x_j)^2}
	\right),\qquad j\in\{1,\ldots,n\}.
\end{equation*} 
Direct calculation gives the chordal counterpart of~\eqref{eqn::radialBPZoperator_commute} (see also~\cite[page~1807]{DubedatCommutationSLE}):
\begin{equation*}
	[\LD_i^{\HH},\LD_j^{\HH}]+\frac{4(\LD_i^{\HH}-\LD_j^{\HH})}{(x_i-x_j)^2}=0,\qquad i\neq j.
\end{equation*}
The same relation holds with $\LD_j^{\HH}$ replaced by $\LD_j^{\HH}-\lambda/\kappa$. Hence the reduction argument in Lemma~\ref{lem::linear_system_compatibility} gives
\begin{equation*}
	\partial_i A_j-\partial_j A_i+[A_j,A_i]=0,\qquad i\neq j.
\end{equation*}

Since $\LX_n^{\HH}$ is convex, the ODE construction in Lemma~\ref{lem::linear_system_sol_existence} gives a unique smooth solution $\bs{u}$ on $\LX_n^{\HH}$ for every initial value $\bs{u}(\bs{x}^0)=\bs{v}\in\R^{2^n}$. The defining recursion implies that $u_{\bs{\beta}}=\partial^{\bs{\beta}}u_{\bs{0}}$ and that $u_{\bs{0}}$ satisfies~\eqref{eqn::BPZ_chordal}. Thus, as in Lemma~\ref{lem::linear_system_injective}, the linear map $\bs{v}\mapsto u_{\bs{0}}$ is injective. Taking a basis of initial values gives $2^n$ linearly independent solutions, as in Lemma~\ref{lem::dim_lowerbound}. This proves $\dim\LSchord{n}^{(\lambda)}\ge 2^n$ and completes the proof.
\end{proof}